\documentclass[twoside,11pt]{article}

\usepackage[preprint]{jmlr2e}
\usepackage{nicefrac}              
\usepackage{microtype}             
\usepackage{xcolor}                
\usepackage{algorithm, algorithmicx}
\usepackage{mathtools}
\usepackage{enumitem}

\newtheorem{assumption}{Assumption}

\newcommand{\R}{{\mathbb R}}
\newcommand{\tr}{\operatorname{tr}}
\newcommand{\proj}{\operatorname{proj}}
\newcommand{\argmin}{\operatornamewithlimits{argmin}}
\newcommand{\dom}{\operatorname{dom}}
\newcommand{\prox}{\operatorname{prox}}

\def\shortdisplay{\setlength{\abovedisplayskip}{5pt}\setlength{\belowdisplayskip}{5pt}\setlength{\abovedisplayshortskip}{2pt}\setlength{\belowdisplayshortskip}{2pt}}

\let\oldselectfont\selectfont
\def\selectfont{\oldselectfont\shortdisplay}

\usepackage{lastpage}
\jmlrheading{24}{2023}{1-\pageref{LastPage}}{11/21; Revised
10/22}{2/23}{21-1410}{Joshua Cutler, Dmitriy Drusvyatskiy, Zaid Harchaoui}
\ShortHeadings{Stochastic Optimization under Distributional Drift}{Cutler, Drusvyatskiy, Harchaoui}
\firstpageno{1}

\begin{document}

\tolerance 1414
\hbadness 1414
\emergencystretch 1.5em
\hfuzz 0.3pt
\widowpenalty=10000
\vfuzz \hfuzz
\raggedbottom

\title{Stochastic Optimization under Distributional Drift}

\author{\name Joshua Cutler \email jocutler@uw.edu\\
	\addr Department of Mathematics\\
	University of Washington\\
	Seattle, WA 98195-4322, USA
	\AND
	\name Dmitriy Drusvyatskiy \email ddrusv@uw.edu\\
	\addr Department of Mathematics\\
	University of Washington\\
	Seattle, WA 98195-4322, USA
	\AND
	\name Zaid Harchaoui \email zaid@uw.edu\\
	\addr Department of Statistics\\
	University of Washington\\
	Seattle, WA 98195-4322, USA}

\editor{Alekh Agarwal}	

\maketitle
	
\begin{abstract}
We consider the problem of minimizing a convex function that is evolving according to unknown and possibly stochastic dynamics, which may depend jointly on time and on the decision variable itself. Such problems abound in the machine learning and signal processing literature, under the names of concept drift, stochastic tracking, and performative prediction. We provide novel non-asymptotic convergence guarantees for stochastic algorithms with iterate averaging, focusing on bounds valid both in expectation and with high probability. The efficiency estimates we obtain clearly decouple the contributions of optimization error, gradient noise, and time drift. Notably, we identify a low drift-to-noise regime in which the tracking efficiency of the proximal stochastic gradient method benefits significantly from a step decay schedule. Numerical experiments illustrate our results.
\end{abstract}	


\section{Introduction}
\label{sec:introduction} 
Stochastic optimization underpins much of machine learning theory and practice. Significant progress has been made over the last two decades in the finite-time analysis of stochastic approximation algorithms~\citep{bottou, bottou:bousquet, bottou:2012, srebro:sridharan:tewari, agarwal:etal:2014, lang:etal,bach:moulines, sra2012optimization, nemirovski2009robust}. The predominant assumption in much of the work on stochastic optimization for machine learning is that the distribution generating the data is fixed throughout the run of the process. There is no shortage of problems, however, where this assumption is grossly violated. There are two main sources of such distributional shifts. The first is temporal, wherein the distribution varies slowly
in time due to reasons that are independent of the learning process. This setting is often called dynamic stochastic approximation~\citep{dupac} and is the basis for adaptive algorithms for stochastic tracking~\citep{benveniste2012adaptive}.
The second common source is due to a feedback mechanism, wherein the distribution generating the data may depend on, or react to, the decisions made by the learner. This setting has been a subject of increased interest recently in the context of strategic classification and performative prediction~\citep{mendler2020stochastic, drusvyatskiy2020stochastic}.

In this work, we present finite-time efficiency estimates in expectation and with high probability for the tracking error of the proximal stochastic gradient method under time drift. The results are presented in a single framework that encompasses both the purely temporal and the decision-dependent time drift. Our results concisely explain the interplay between the learning rate, the noise variance in the gradient oracle, and the strength of the time drift.  
While conventional wisdom and previous work recommend the use of a constant step size under time drift, we identify a low drift-to-noise regime in which tracking efficiency benefits significantly from a step size schedule that geometrically decays to a ``critical step size''.

Setting the stage, consider the sequence of stochastic optimization problems 
\begin{equation}\label{eqn:prob_online_intro}
	\min_x\, \varphi_t(x) := f_t(x) + r_t(x)
\end{equation}
indexed by time $t\in \mathbb{N}$. In typical machine learning and signal processing settings, the function $f_t$ corresponds to an average loss that varies in time, while the regularizer $r_t$ models constraints or promotes structure (e.g., sparsity) in the variable $x$. 
Two examples are worth highlighting. The first is a classical problem in signal processing related to stochastic tracking~\citep{kushner:lin,sayed2003fundamentals}, wherein the learning algorithm aims to track over time a moving target driven by an unknown stochastic process. The second example is the concept drift phenomenon in online learning~\citep{HazanS09, ZhangLZ18}, wherein the true hypothesis may be changing over time.

The main goal of a learning algorithm for problem~\eqref{eqn:prob_online_intro} is to generate a sequence of points $\{x_t\}$ that minimize some natural performance metric. To make progress, we impose the standard  assumption that each function $f_t$ is $\mu$-strongly convex with $L$-Lipschitz continuous gradient, while each regularizer $r_t$ is proper, closed, and convex. The online proximal stochastic gradient method (PSG) naturally applies to the sequence of problems \eqref{eqn:prob_online_intro}. At each iteration $t$, the method simply takes the step 
$$x_{t+1}=\prox_{\eta_t r_t}(x_t-\eta_t g_t),$$
where the vector $g_t$ is an unbiased estimator of the true gradient of $f_t$ at $x_t$, the step size (learning rate) $\eta_t>0$ is user-specified, and $\prox_{\eta_t r_t}(\cdot)$ is the proximal map of the scaled regularizer $\eta_t r_t$.
In this work, we analyze two types of tracking error for PSG: the squared distance $\|x_t-x_t^\star\|^2$ and the suboptimality gap $\varphi_t(\hat x_t)-\varphi_t(x_{t}^{\star})$. Here, $x_t^\star$ denotes the minimizer of the function $\varphi_t$ which may evolve stochastically in time, and $\hat x_t$ denotes a weighted average of iterates up to time $t$. We next outline the main results of the paper; the results in Sections~\ref{sec:intro_track_dist} and \ref{sec:intro_track_val} below appeared in a preliminary version of this paper at NeurIPS \citep{cutler21stochastic}.

\subsection{Tracking the Minimizer}\label{sec:intro_track_dist}
We begin with a simple bound on distance tracking of the constant-step PSG: 
	\begin{equation}\label{eqn:one_step_intro}
		\mathbb{E}\|x_t-x^{\star}_t\|^2 \lesssim \underbrace{(1-\mu\eta)^t\|x_0-x^{\star}_0\|^2}_\text{optimization} + \underbrace{\frac{\eta\sigma^2}{\mu}}_\text{noise} + \underbrace{\left(\frac{\Delta}{\mu\eta}\right)^2}_\text{drift}.
	\end{equation}
Here $\eta\in (0,1/2L]$ is the constant step size used by PSG, $\sigma^2$ upper-bounds the variance of the stochastic gradient, and $\Delta^2$ upper-bounds the minimizer variations $\mathbb{E}\|x_t^\star - x_{t+1}^\star\|^2$; the symbol $\lesssim$ indicates an inequality that holds up to an absolute constant factor, i.e., up to multiplying the upper bound by a positive numerical constant independent of the problem parameters.
Inequality~(\ref{eqn:one_step_intro}) asserts that the tracking error $\mathbb{E}\|x_t-x^{\star}_t\|^2$ decays linearly in time $t$, 
until it reaches the ``noise\,+\,drift'' error $\eta\sigma^2/\mu + (\Delta/\mu\eta)^2$. 
Notice that the ``noise\,+\,drift'' error cannot be made arbitrarily small by tuning $\eta$. 
This is perfectly in line with intuition: a step size $\eta$ that is too small prevents the algorithm from catching up with the minimizers $x_t^{\star}$. We note that the individual error terms due to the optimization and noise are classically known to be tight for PSG; tightness of the drift term is proved by \citet[Theorem 3.2]{madden2021bounds}. Though the estimate~\eqref{eqn:one_step_intro} is likely known, we were unable to find a precise reference in this generality.

Letting $t$ tend to infinity in~\eqref{eqn:one_step_intro}, the optimization error tends to zero, leaving only the ``noise\,+\,drift'' term. Optimizing this remaining term over $\eta$, it is natural to define the asymptotic distance tracking error of PSG and the corresponding optimal  learning rate as
\begin{equation*}
	\mathcal{E}:=\min_{\eta\in (0,1/2L]} \left\{\frac{\eta\sigma^2}{\mu} + \left(\frac{\Delta}{\mu\eta}\right)^2\right\}\quad \textnormal{and}\quad \eta_\star := \min\left\{\frac{1}{2L}, \left(\frac{2\Delta^2}{\mu\sigma^2}\right)^{1/3}\right\}\!.
\end{equation*}
Two regimes of variation are brought to light: the \emph{high drift-to-noise regime} ${\Delta/\sigma \geq \sqrt{\mu/16L^3}}$, and the \emph{low drift-to-noise regime} $\Delta/\sigma< \sqrt{\mu/16L^3}$. The high drift-to-noise regime is uninteresting from the viewpoint of stochastic optimization because in this case the optimal learning rate $\eta_\star \asymp 1/L$ is as large as in the deterministic setting (here, the symbol $\asymp$ indicates an equality that holds up to an absolute constant factor).
In contrast, the low drift-to-noise regime is interesting because it necessitates using a smaller learning rate $\eta_{\star}\asymp (\Delta^2/\mu\sigma^2)^{1/3}$ that exhibits a nontrivial scaling with the problem parameters. Consequently, for the rest of the introduction we focus on the low drift-to-noise regime.

A central question is to find a learning rate schedule that achieves a tracking error $\mathbb{E}\|x_t-x^{\star}_t\|^2 $ that is within a constant factor of $\mathcal{E}$ in the shortest possible time. The simplest strategy is to execute PSG 
with the constant learning rate $\eta_\star$. Then a direct application of \eqref{eqn:one_step_intro} yields the efficiency estimate $\mathbb{E}\|x_t-x^{\star}_t\|^2 \lesssim \mathcal{E}$ in time $t\lesssim (\sigma^2/\mu^2\mathcal{E}) \log(\|x_0-x_0^\star\|^2/\mathcal{E})$. 
This efficiency estimate can be significantly improved by gradually decaying the learning rate using a ``step decay schedule'', wherein the algorithm is implemented in epochs with the new learning rate chosen to be the midpoint between the current learning rate and $\eta_{\star}$. 
Such schedules are well known to improve efficiency in the static (stationary objective) setting, as was  discovered by  \citet{ghadimi2013optimal}, and can be used here. The end result is an algorithm that produces a point $x_t$ satisfying 
\begin{equation}\label{eqn:geo_dec_intro}
\mathbb{E}\|x_t-x_{t}^{\star}\|^2\lesssim \mathcal{E} \quad\text{in time}\quad t\lesssim \frac{L}{\mu}\log\!\left(\frac{ \|x_0-x_0^\star\|^2}{\mathcal{E}}\right)+\frac{\sigma^2}{\mu^2\mathcal{E}}.
\end{equation}
This efficiency estimate is remarkably similar to that in the static setting \citep{ghadimi2013optimal}, with $\mathcal{E}$ playing the role of the target accuracy $\varepsilon$. 
An elementary computation shows that \eqref{eqn:geo_dec_intro} improves the constant learning rate efficiency estimate when $\mathcal{E}$ is small, e.g., when $\mathcal{E}\leq\|x_0-x_0^{\star}\|^2/e^2$, where $e$ denotes Euler's number.

The efficiency estimate~\eqref{eqn:geo_dec_intro} is a baseline guarantee for PSG with step decay. Since the result is stated in terms of the \emph{expected} tracking error $\mathbb{E}\|x_t-x_{t}^{\star}\|^2$, it is only meaningful if the entire algorithm can be repeated from scratch multiple times on the same problem.\footnote{Specifically, error bounds holding in expectation yield concentration inequalities for the average of i.i.d.\ errors arising from executing a stochastic algorithm multiple times from scratch on the same problem (e.g., via Chebyshev's inequality); if the algorithm cannot be executed in this fashion to generate i.i.d.\ errors, then alternative high-probability error bounds are called for.} However, there is no shortage of situations in which a learning algorithm is operating in real time and the time drift is irreversible; in such settings, the algorithm may only be executed once. These situations call for efficiency estimates that hold with high probability, rather than only in expectation. With this in mind, we show that under mild light-tail assumptions, PSG with step decay produces a point $x_t$ satisfying $\|x_t-x_{t}^{\star}\|^2\lesssim \mathcal{E}\log\left(1/\delta\right)$ with probability at least $1-\delta$ in the same order of iterations as in \eqref{eqn:geo_dec_intro}.
The proof follows closely the probabilistic techniques developed by \citet{pmlr-v99-harvey19a} for bounding moment generating functions.

\subsection{Tracking the Minimum Value}\label{sec:intro_track_val}
The results outlined so far have focused on tracking the minimizer $x_t^{\star}$; stronger guarantees may be obtained for tracking the minimum value $\varphi_{t}^{\star}$. To this end, we require stronger assumptions on the variation of the functions $f_t$ beyond control on the minimizer drift $\|x_t^\star-x_{t+1}^\star\|^2$. Similar in spirit to the measure of cumulative gradient variation in the dynamic online learning literature~\citep[e.g., see][]{pmlr-v38-jadbabaie15}, we will be concerned with the \emph{gradient drift} $$G_{i,t}:=\sup_{x}\|\nabla f_i(x)-\nabla f_t(x)\|$$ 
and assume the bound $\mathbb{E}[G_{i,t}^2/\mu^2]\leq \Delta^2|i-t|^2$ for all times $i$ and $t$. Thus, the second moment of the gradient drift is assumed to grow at most quadratically in the time horizon. Assuming henceforth that the regularizers $r_t\equiv r$ are identical for all times $t$, this condition on the gradient drift implies the weaker assumption $\mathbb{E}\|x_t^\star - x_{t+1}^\star\|^2\leq\Delta^2$ used in Section~\ref{sec:intro_track_dist}.

Analogous to \eqref{eqn:one_step_intro}, we show that PSG generates a point $\hat x_t$ (an average iterate) satisfying 
	\begin{equation*}
		\mathbb{E}\big[\varphi_t(\hat x_t)-\varphi_t^\star\big] \lesssim  \underbrace{\big(1 - \tfrac{\mu\eta}{2}\big)^t\big(\varphi_0(x_0)-\varphi_0^\star\big)}_\text{optimization} + \underbrace{\eta\sigma^2}_\text{noise} + \underbrace{\frac{\Delta^2}{\mu\eta^2}}_\text{drift}.
	\end{equation*}
	This estimate again decouples nicely into three terms, signifying the error due to optimization, gradient noise, and time drift.
Taking the limit as $t$ tends to infinity, we obtain the asymptotic function gap tracking error $\mathcal{G}:=\mu\mathcal{E}$.  
Similar to \eqref{eqn:geo_dec_intro}, we show that PSG with step decay produces a point $\hat{x}_t$ satisfying 
\begin{equation}\label{eqn:gap_eqn_intro}
\mathbb{E}\big[\varphi_t(\hat{x}_t) - \varphi_t^\star\big] \lesssim \mathcal{G}\quad\text{in time}\quad t\lesssim \frac{L}{\mu}\log\!\left(\frac{\varphi_0(x_0) - \varphi_0^\star}{\mathcal{G}}\right)+\frac{\sigma^2}{\mu\mathcal{G}}.
\end{equation}
Again, the similarity to the static setting \citep{ghadimi2013optimal}, with $\mathcal{G}$ playing the role of a target accuracy, is striking. We then provide a high-probability extension of this estimate: under mild light-tail assumptions, PSG with step decay produces a point $\hat{x}_t$ satisfying $\varphi_t(\hat{x}_t) - \varphi_t^\star\lesssim \mathcal{G}\log(1/\delta)$ with probability at least $1-\delta$ in the same order of iterations as in \eqref{eqn:gap_eqn_intro} up to a factor of $\log\log(1/\delta)$. The proofs are based on the generalized Freedman inequality of  \citet{pmlr-v99-harvey19a}---a remarkably flexible tool for analyzing stochastic gradient-type algorithms.

\subsection{Extension to Decision-Dependent Problems with Time Drift}
We have so far focused on stochastic optimization problems that undergo a temporal shift. A primary reason for this phenomenon in machine learning, and data science more broadly, is that data distributions often evolve in time independently of the learning process. Recent literature, on the other hand, highlights a different source of distributional shift due to decision-dependent or performative effects. Namely, the distribution generating the data in iteration $t$ may depend on, or react to, the current ``decision'' $x_t$. For example, deployment of a classifier by a learning system, when made public, often causes the population to adapt their attributes in order to increase the likelihood of being positively labeled---a process called “gaming”. Even when
the population is agnostic to the classifier, the decisions made by the learning system (e.g., loan
approval) may inadvertently alter the profile of the population (e.g., credit score). The goal of the
learning system therefore is to find a classifier that generalizes well under the response distribution.
Recent research in strategic classification \citep{hardt2016strategic,bruckner2012static,bechavod2020causal,dalvi2004adversarial} and performative prediction \citep{perdomo2020performative,mendler2020stochastic} has highlighted the prevalence of this phenomenon. 

Combining time-dependence and decision-dependence yields a class of problems \eqref{eqn:prob_online_intro} where the loss function $f_t(x)$ takes the special form
$f_t(x)=\mathbb{E}_{\xi\sim \mathcal{D}(t, x)}\ell(x,\xi)$.
Here $\mathcal{D}(t, x)$ is a distribution that depends on both time $t$ and the decision variable $x$. Thus for any fixed time $t$, the problem \eqref{eqn:prob_online_intro}  becomes the performative risk problem considered by  \citet{perdomo2020performative} and \citet{mendler2020stochastic}.
Following this line of work, instead of tracking the true minimizer of $\varphi_t$---typically a challenging task---we will settle for tracking the equilibrium points $\bar x_t$. These are the points satisfying
$$\bar{x}_t\in\argmin_{x}\underset{\xi\sim \mathcal{D}(t, \bar x_{t})}{\mathbb{E}}\ell(x,\xi)+r(x).$$
Equilibrium points are sure to exist and are unique under mild Lipschitzness and strong convexity assumptions.
We refer the reader to \citet{perdomo2020performative} for a compelling motivation for considering such equilibrium points. The problem of tracking equilibrium points is yet again an instance of \eqref{eqn:prob_online_intro}, but now with the different function $f_t(x)=\mathbb{E}_{\xi\sim \mathcal{D}(t, \bar x_{t})}\ell(x,\xi)$ induced by the equilibrium distributions. The PSG algorithm is not directly applicable here since the learner cannot typically sample from $\mathcal{D}(t, \bar x_{t})$ directly. Instead, a natural algorithm for this problem class draws in each iteration $t$ a sample $\xi_t$ from the current distribution $\mathcal{D}(t, x_t)$ and declares $x_{t+1}=\prox_{\eta_t r}(x_t-\eta_t\nabla\ell(x_t,\xi_t))$. 
Notice that the sample gradient $\nabla\ell(x_t,\xi_t)$ is a biased estimator of the true gradient $\nabla f_t(x_t) = \mathbb{E}_{\xi\sim \mathcal{D}(t, \bar x_{t})}\nabla\ell(x_t,\xi)$ because $\xi_t$ is sampled from the wrong distribution. Nonetheless, as pointed out by \citet{drusvyatskiy2020stochastic}, the gradient bias is small for any fixed time, decaying linearly with the distance to $\bar x_t$. Using this perspective, we show that all guarantees for PSG in the time-dependent setting naturally extend to this biased PSG algorithm for tracking equilibrium points, with essentially no loss in efficiency.

\subsection{Related Work}
\label{sec:related} 
\smallskip
Our current work fits within the broader literature on stochastic tracking, online optimization with dynamic regret, high-probability guarantees in stochastic optimization, and performative prediction.
We now survey the most relevant literature in these areas.

\paragraph{Stochastic tracking.}
Stochastic optimization with time drift was considered soon after the Robbins-Monro approach for stochastic optimization was introduced; see~\citet{kushner:lin} for a survey. 
Early results can be traced back to~\citet{dupac} in sequential estimation and~\citet{gaivoronskii1978nonstationary} in stochastic optimization; 
see also~\citet{fujita1972convergence}, \citet{ruppert:1979}, \citet{1971adaptation}, \citet{tsypkin:1992}, and \citet{uosaki1974some}.
Stochastic algorithms have also been extensively studied as adaptive algorithms for stochastic tracking~\citep{1971adaptation, kushner:lin, benveniste2012adaptive}, 
for their ability to indeed track parameters under time drift. Most works have focused on the so-called least mean-squares (LMS) algorithm and its variants, which can be viewed as a stochastic gradient method on a least-squares loss-based objective. Other stochastic algorithms that have been studied in these settings with a larger cost per iteration include recursive least-squares and related Kalman filtering algorithms~\citep{guo1995exponential}. 

Recent works have revisited these methods from a more modern viewpoint~\citep{besbes2015non,wilson2018adaptive,madden2021bounds}.  In particular, the paper of \citet{madden2021bounds} focuses on (accelerated) gradient methods for deterministic tracking problems, while \citet{wilson2018adaptive} present a framework for online stochastic gradient methods with parameter estimation. The work of \citet{besbes2015non} analyzes the dynamic regret of stochastic algorithms for time-varying problems, focusing both on lower and upper complexity bounds. Though the proof techniques in our paper share many aspects with those available in the literature, the results we obtain are distinct. In particular, the guarantees \eqref{eqn:geo_dec_intro} and \eqref{eqn:gap_eqn_intro} for PSG with step decay, along with their high-probability variants, are new to the best of our knowledge.

\paragraph{Online optimization with dynamic regret.}
Online optimization for sequences of convex objectives with domain $\mathcal{X}$ has been studied through the lens of adaptive regret \citep{HazanS09, DanielyGS15} and dynamic regret \citep{besbes2015non, zinkevich03, ZhangLZ18, zhao20dynreg, pmlr-v38-jadbabaie15, MokhtariSJR16}. The adaptive regret 
	$$
	\sup_{[r,s]\subset[T]}\left\{ \sum_{t=r}^s f_t(x_t) - \inf_{x\in\mathcal{X}}\sum_{t=r}^s f_t(x) \right\}
	$$ 
	reads as the maximum static regret over any contiguous time interval; more relevant to our analysis is the dynamic regret 
	$$
	\textrm{Reg}_T^\star = \sum_{t=1}^T \big(f_t(x_t) - f_t(x_t^\star)\big),
	$$
	which reads as the cumulative difference between the instantaneous loss and the minimum loss. More generally, one can consider the dynamic regret against an arbitrary comparator sequence $\{u_t\}_{t=1}^T$ in $\mathcal{X}$, given by
	$$
	\textrm{Reg}_T(u_1,\ldots,u_T) = \sum_{t=1}^T \big(f_t(x_t) - f_t(u_t)\big).
	$$

\citet{pmlr-v38-jadbabaie15} apply an adaptive step size strategy to an optimistic mirror descent algorithm, thereby obtaining a comprehensive dynamic regret guarantee for $\textrm{Reg}_T^\star$ in terms of the cumulative loss variation $V_T = \sum_{t=2}^T \sup_{x\in\mathcal{X}}|f_t(x) - f_{t-1}(x)|$, the cumulative gradient variation $D_T = \sum_{t=1}^T \|\nabla f_t(x_t) - M_t\|^2$ using a causally predictable sequence $M_t$ available to the algorithm prior to time $t$ (e.g., $M_t = \nabla f_{t-1}(x_{t-1})$), and the cumulative minimizer variation $C_T^\star = \sum_{t=2}^T \|x_t^\star - x_{t-1}^\star\|$. Under strong convexity, \citet{MokhtariSJR16} show that online projected gradient descent satisfies $\textrm{Reg}_T^\star \leq \mathcal{O}(1 + C_T^\star)$. For convex losses with bounded domain $\mathcal{X}$, \citet{ZhangLZ18} present an adaptive online gradient method  that achieves an optimal dynamic regret bound in terms of the cumulative comparator variation $C_T = \sum_{t=2}^T\|u_{t} - u_{t-1}\|$, namely, $\textrm{Reg}_T(u_1,\ldots,u_T) \leq \mathcal{O}\big(\sqrt{T(1 + C_T)}\big)$. This last guarantee is enhanced by \citet{zhao20dynreg} through exploiting smoothness to replace the time horizon $T$ by problem-dependent quantities that are at most $\mathcal{O}(T)$ but often much smaller in easy problems.

As is standard in the dynamic online optimization literature, the preceding works assume that either the losses $f_t(x)$ or their gradients $\nabla f_t(x)$ are uniformly bounded in both $t$ and $x$, and a priori knowledge of these uniform bounds is required for the aforementioned guarantees. In contrast, we take great care to make no such uniform boundedness assumptions and we work instead with bounded second moments or light tails of minimizer or gradient drift, which we allow to evolve stochastically. Furthermore, we only assume stochastic gradient access, and the presence of stochasticity in the drift and the gradient noise requires guarantees that hold both in expectation and with high probability. Our bounds depend on a characteristic quantity of the problem difficulty encapsulating the drift and the noise level and hence delineate two regimes depending on the drift-to-noise ratio. In general, regret bounds do not entail last-iterate bounds of the type presented in our work.

\paragraph{High-probability guarantees in stochastic optimization.}
A large part of our work revolves around high-probability guarantees for stochastic optimization. Classical references on the subject in static settings and for minimizing regret in online optimization include the work of \citet{bartlett2008high}, \citet{hazan2014beyond}, \citet{lan2012optimal}, and \citet{rakhlin2011making}. There exists a variety of techniques for establishing high-probability guarantees based on Freedman's inequality and doubling tricks \citep[e.g., see][]{bartlett2008high, hazan2014beyond}. A more recent line of work by \citet{pmlr-v99-harvey19a} establishes a generalized Freedman inequality that is custom-tailored for analyzing stochastic gradient-type methods and results in the best known high-probability guarantees. Our arguments closely follow the paradigm of~\citet{pmlr-v99-harvey19a} based on the generalized Freedman inequality.

\paragraph{Performative prediction and decision-dependent learning.}
Recent works on strategic classification \citep{hardt2016strategic,bruckner2012static,bechavod2020causal,dalvi2004adversarial} and performative prediction \citep{perdomo2020performative,mendler2020stochastic} have highlighted the importance of strategic behavior in machine learning. That is, common learning systems exhibit a feedback mechanism, wherein the distribution generating the data in iteration $t$ may depend on, or react to, the current ``decision'' of an algorithm $x_t$. The recent paper by \citet{perdomo2020performative} put forth an elegant framework for thinking about such problems, while \citet{mendler2020stochastic} develop stochastic algorithms for this setting. The subsequent work of \citet{drusvyatskiy2020stochastic} shows that a variety of stochastic algorithms for performative prediction can be understood as biased variants of the same algorithms on a certain static problem in equilibrium. Building on the techniques of \citet{drusvyatskiy2020stochastic}, we show how all our results for time-dependent problems extend to problems that simultaneously depend on time and on the decision variable. We note that during the final stage of completing this paper, the closely related and complementary work by \citet{wood2021online} was posted on arXiv.\footnote{More precisely, a short version of our paper \citep{cutler21stochastic} was submitted to NeurIPS in May '21, the paper by \citet{wood2021online} appeared on arXiv in July '21, and our full paper was posted on arXiv in August '21. Our paper \citep{cutler21stochastic} was presented at NeurIPS in December '21.} The paper by \citet{wood2021online} considers decision-dependent projected stochastic gradient descent under time drift in the distributional framework proposed by \citet{perdomo2020performative}, establishing distance tracking bounds in expectation and with high probability under sub-Weibull gradient noise. In particular, the light-tail assumption on gradient noise used by \citet{wood2021online} for obtaining high-probability guarantees is more general than the one in our paper. On the other hand, we analyze tracking of both the minimizer and the minimum value of more general stochastically evolving objectives, allow presence of general convex regularizers, and propose a step decay schedule for improved efficiency.

\subsection{Outline}

The outline of the paper is as follows. Section~\ref{sec:framework} formalizes the problem setting of time-dependent stochastic optimization and records the relevant assumptions. Sections~\ref{sec:trckminimizer}--\ref{decision-dependent} summarize the main results of the paper. Specifically, Section~\ref{sec:trckminimizer}  focuses on efficiency estimates for tracking the minimizer, Section~\ref{sec:trckminimum} focuses on efficiency estimates for tracking the minimum value, and Section~\ref{decision-dependent} develops an extension to the decision-dependent setting via tracking equilibria. Section~\ref{mainproofs} presents the proofs of the main results in a unified framework. Illustrative numerical results appear in Section~\ref{sec:exp}. Appendix~\ref{apdx:avg} describes the averaging technique used for tracking function values, and additional proofs appear in Appendix~\ref{missingproofs}.
 \section{Framework and Assumptions}
\label{sec:framework} 
Throughout Sections~\ref{sec:framework}--\ref{sec:trckminimum}, we consider the sequence of stochastic optimization problems 
\begin{equation}\label{eqn:prob_online}
	\min_{x\in\mathbb{R}^d}\, \varphi_t(x) := f_t(x) + r_t(x)
\end{equation}
indexed by time $t\in \mathbb{N}$, where $\mathbb{R}^d$ denotes a fixed $d$-dimensional Euclidean space with inner product $\langle \cdot,\cdot\rangle$ and Euclidean norm $\|x\|=\sqrt{\langle x, x \rangle}$, and the following standard regularity assumptions hold:
\begin{enumerate}[label=(\roman*)]
\item Each function $f_t\colon\mathbb{R}^d\to\mathbb{R}$ is $\mu$-strongly convex and $C^1$-smooth with $L$-Lipschitz continuous gradient for some common parameters $\mu,L>0$.
\item Each regularizer $r_t\colon\mathbb{R}^d\to\mathbb{R}\cup\{\infty\}$ is proper, closed, and convex.\footnote{We assume $\dom{f_{t}}=\mathbb{R}^d$ for simplicity, but this is not essential. For example, it suffices to assume that each function $f_{t}\colon\mathbb{R}^d\to\mathbb{R}\cup\{\infty\}$ is closed and $\mu$-strongly convex and that there exists an open convex set $U\subset\R^d$ such that for all $t\in\mathbb{N}$, $\dom r_t \subset U \subset \dom f_{t}$ and $f_{t}$ is $L$-smooth on $U$.}  
\end{enumerate}
The  minimizer and minimum value of $\varphi_t$ will be denoted by $x_t^{\star}$ and $\varphi^{\star}_t$, respectively. We will be concerned with settings in which $\varphi_t$ evolves stochastically in time. As motivation, we describe two classical examples of \eqref{eqn:prob_online} that are worth keeping in mind and that guide our framework: stochastic tracking of a drifting target and online learning under distributional drift.

\begin{example}[Stochastic tracking of a drifting target]\label{ex:1}
\textup{The problem of stochastic tracking, related to the filtering problem in signal processing, is to track a moving target $x_t^{\star}$ from observations
$$b_t=c_t(x_t^{\star})+\epsilon_t,$$
where $c_t(\cdot)$ is a known measurement map and $\epsilon_t$ is a mean-zero noise vector. A typical time-dependent problem formulation takes the form
$$\min_x~\underset{\epsilon_t}{\mathbb{E}}~\ell_t(b_t-c_t(x))+r_t(x),$$
where the loss $\ell_t(\cdot)$ derives from the distribution of $\epsilon_t$ and the regularizer $r_t(\cdot)$  encodes available side information about the target $x_t^{\star}$. Common choices for $r_t$ are the $1$-norm and the squared $2$-norm. The motion of the target $x_t^{\star}$  is typically driven by a random walk or a diffusion~\citep{guo1995exponential,sayed2003fundamentals}.}
\end{example}

\begin{example}[Online learning under distributional drift]\label{ex:2}
\textup{The problem of online learning under distributional drift is to learn while the data distribution changes over time. More formally, a typical problem formulation takes the form
$$\min_x\underset{\xi\sim \mathcal{D}(v_t)}{\mathbb{E}} \ell(x,\xi)+r(x),$$
where $\mathcal{D}(v_t)$ is a data distribution that depends on an unknown parameter sequence $\{v_t\}$, which itself may evolve stochastically. 
}
\end{example}

The main goal of a learning algorithm for problem~\eqref{eqn:prob_online} is to generate a sequence of points $\{x_t\}$ that minimize some natural performance metric. The most prevalent performance metrics in the literature are the \emph{tracking error} and the \emph{dynamic regret}. We will focus on two types of tracking error: the squared distance $\|x_t-x_t^\star\|^2$ and the suboptimality gap $\varphi_t(\hat x_t)-\varphi_t(x_{t}^{\star})$, where $\hat x_t$ denotes a weighted average of iterates up to time $t$.

We make the standing assumption that at every time $t$, and at every query point $x$, the learner can select an \emph{unbiased estimator} $\widetilde{\nabla}f_t(x)$ of the true gradient $\nabla f_t(x)$ in order to proceed with a stochastic gradient-like optimization algorithm. With this oracle access, the online proximal stochastic gradient method---recorded as Algorithm~\ref{alg:sgd} below---selects in each iteration $t$ the stochastic gradient $g_t = \widetilde{\nabla}f_t(x_t)$ and takes the step 
\begin{equation*}
	x_{t+1} := \prox_{\eta_t r_t}(x_t-\eta_tg_t) = \argmin_{u\in\mathbb{R}^d}\left\{r_t(u)+\tfrac{1}{2\eta_t}\|u-(x_t-\eta_tg_t)\|^2\right\} 
\end{equation*}
using step size $\eta_t>0$. The goal of our work is to obtain efficiency estimates for this procedure that hold both in expectation and with high probability.
 
\begin{algorithm}[h!]
	\caption{Online Proximal Stochastic Gradient\hfill $\mathtt{PSG}(x_{0},\{\eta_t\},T)$}\label{alg:sgd}
	{\bf Input}: initial $x_0$ and step sizes $\{\eta_t\}_{t=0}^{T-1}\subset (0,\infty)$
	
	{\bf Step} $t=0,\ldots,T-1$: 
	\begin{equation*}
		\begin{aligned}
			&\text{Select~} g_t=\widetilde{\nabla}f_t(x_t) \\
			&\text{Set~} x_{t+1}=\prox_{\eta_t r_t}(x_t-\eta_t g_t)
		\end{aligned}
	\end{equation*}
	{\bf Return} $x_T$
\end{algorithm}

The guarantees we obtain allow both the iterates $x_t$ \emph{and the minimizers} $x_t^\star$ to evolve stochastically. This is convenient for example when tracking a moving target $x_t^\star$ whose motion may be governed 
by a stochastic process such as a random walk or a diffusion (Example~\ref{ex:1}), or when tracking the minimizer of an expected loss over a stochastically evolving data distribution (Example~\ref{ex:2}). Given $\{x_t\}$ and $\{g_t\}$ as in Algorithm~\ref{alg:sgd}, we let $$z_t := \nabla f_t(x_t) - g_t$$ denote the \emph{gradient noise} at time $t$ and we impose the following assumption modeling stochasticity on a fixed probability space $(\Omega,\mathcal{F},\mathbb{P})$ throughout Sections~\ref{sec:framework}--\ref{sec:trckminimum}.

\begin{assumption}[Stochastic framework]\label{assmp:stochfram}
	\textup{There exists a probability space $(\Omega,\mathcal{F},\mathbb{P})$ with filtration $(\mathcal{F}_t)_{t\geq0}$ such that $\mathcal{F}_0=\left\{\emptyset,\Omega\right\}$ and the following two conditions hold for all $t\geq0$:
	\begin{enumerate}[label=(\roman*)]
		\item\label{it1} $x_t,x_t^\star\colon\Omega\rightarrow\mathbb{R}^d$ are $\mathcal{F}_t$-measurable.
		\item\label{it2} $z_t\colon\Omega\rightarrow\mathbb{R}^d$ is $\mathcal{F}_{t+1}$-measurable with $\mathbb{E}[z_t\,|\,\mathcal{F}_t]=0$.
	\end{enumerate}
	}
\end{assumption}
The first item of Assumption~\ref{assmp:stochfram} formalizes the assertion that $x_t$ and $x_t^\star$ are fully determined by information up to time $t$. The second item of Assumption~\ref{assmp:stochfram} formalizes the assertion that the gradient noise $z_t$ is fully determined by information up to time $t+1$ and has zero mean conditioned on the information up to time $t$, i.e., $g_t$ is an unbiased estimator of $\nabla f_t(x_t)$; for example, this holds naturally in Example~\ref{ex:2} under typical regularity assumptions if $g_t = \nabla\ell(x_t,\xi_t)$ with $\xi_t\sim \mathcal{D}(v_t)$, where $\nabla \ell(x_t,\xi_t)$ denotes the gradient of $\ell(\cdot,\xi_t)$ at $x_t$.

Efficiency estimates for Algorithm~\ref{alg:sgd} must clearly take into account the variation of the problem \eqref{eqn:prob_online} in time $t$. One of the standard metrics for measuring this variation is the \emph{minimizer drift} $$\Delta_t := \|x_t^\star - x_{t+1}^\star\|.$$ 
Another popular metric is the \emph{gradient drift} $$\sup_{x}\|\nabla f_t(x)-\nabla f_{t+1}(x)\|.$$ Our efficiency estimates for tracking the minimizer will depend on the minimizer drift, while our efficiency estimates for tracking the minimum value will depend on the gradient drift. As the following elementary lemma shows, the minimizer drift scaled by $\mu$ is dominated by the gradient drift  whenever the regularizers do not vary in time.\footnote{Lemma~\ref{lem:grad_vs_dist} provides a bound similar in spirit to the bound $\mu \|x^{\star}_{i} - x^{\star}_{t}\|^2\leq 4\sup_{x\in\dom{r}}|f_{i}(x) - f_t(x)|$ in terms of variation in function value, which is also an elementary consequence of $\mu$-strong convexity \citep[e.g., see][Section~4.1]{ZhaoZ21}.}

\begin{lemma}[Minimizer vs. gradient drift]
	\label{lem:grad_vs_dist}
	Suppose that $i$ and $t$ are indices for which the regularizers $r_i$ and $r_t$ are identical. Then \[\mu \|x^{\star}_i-x^{\star}_{t}\|\leq \|\nabla f_{i}(x_t^\star)-\nabla f_t(x_t^\star)\|.\]
\end{lemma}

\begin{proof}
	Let $r$ denote the common regularizer: $r = r_i = r_t$. Then the first-order optimality condition $$0\in \partial \varphi_t(x_t^{\star}) = \nabla f_t(x_t^{\star}) + \partial r(x_t^\star)$$ implies $-\nabla f_t(x_t^{\star})\in \partial r(x_t^{\star})$, so the vector $v:=\nabla f_i(x_t^{\star})-\nabla f_t(x_t^{\star})$ lies in $\partial \varphi_i(x_{t}^{\star})$. Hence the $\mu$-strong convexity of $\varphi_i$ and the inclusion $0\in \partial \varphi_i(x_i^{\star})$ imply 
	$\mu\|x_i^{\star}-x_t^{\star}\|\leq \|0-v\|$.
\end{proof}
  \section{Tracking the Minimizer}
\label{sec:trckminimizer}
This section presents bounds on the tracking error $\|x_t-x_t^{\star}\|^2$ that are valid both in expectation and with high probability under light-tail assumptions. Further, we show that a geometrically decaying learning rate schedule may be superior to a constant learning rate in terms of efficiency. 

\subsection{Bounds in Expectation}\label{sec:minimizerexp}

We begin with bounding the expected value $\mathbb{E}\|x_t-x_t^{\star}\|^2$. Proofs appear in Section~\ref{minimizerproofsexp}.
The starting point for our analysis is the following standard one-step improvement guarantee.
\begin{lemma}[One-step improvement]\label{lem:onestep0}
	For all $x\in \mathbb{R}^d$, the iterates $\{x_t\}$ produced by Algorithm~\ref{alg:sgd} with $\eta_t<1/L$ satisfy the bound:
	\begin{equation*}
		2\eta_t(\varphi_t(x_{t+1})-\varphi_t(x)) \leq (1-\mu\eta_t)\|x_t-x\|^2 - \|x_{t+1}-x\|^2 + 2\eta_t\langle z_t,x_t-x \rangle + \tfrac{\eta^2_t}{1-L\eta_t}\|z_t\|^2.
	\end{equation*}
\end{lemma}

For simplicity, we state the main results under the assumption that the second moments $\mathbb{E}\,\Delta_t^2$ and $\mathbb{E}\|z_t\|^2$ are uniformly bounded; more general guarantees that take into account weighted averages of the moments and allow for time-dependent learning rates follow from Lemma~\ref{lem:onestep0} as well.

\begin{assumption}[Bounded second moments]\label{assump:sec_moment}
	\textup{There exist constants $\Delta,\sigma>0$ such that the following two conditions hold for all $t\geq0$:
	\begin{enumerate}[label=(\roman*)]
		\item {\bf (Drift) } The minimizer drift $\Delta_t$ satisfies $\mathbb{E}\,\Delta_t^2\leq \Delta^2$. 
		\item {\bf (Noise) } The gradient noise $z_t$ satisfies $\mathbb{E}\|z_t\|^2\leq\sigma^2$.
	\end{enumerate}
	}
\end{assumption}

The following theorem establishes an expected improvement guarantee for Algorithm~\ref{alg:sgd}, and serves as the basis for much of what follows.

\begin{theorem}[Expected distance]\label{thm:exp_dist}
	Suppose that Assumption~\ref{assump:sec_moment} holds. Then the iterates produced by Algorithm~\ref{alg:sgd} with constant learning rate $\eta\leq 1/2L$ satisfy the bound:  
	\begin{equation*}
		\mathbb{E}\|x_t-x^{\star}_t\|^2 \lesssim \underbrace{(1-\mu\eta)^t\|x_0-x^{\star}_0\|^2}_\text{optimization} + \underbrace{\frac{\eta\sigma^2}{\mu}}_\text{noise} + \underbrace{\left(\frac{\Delta}{\mu\eta}\right)^2}_\text{drift}.
	\end{equation*}
\end{theorem}

\paragraph{Interplay of optimization, noise, and drift.}
Theorem~\ref{thm:exp_dist} states that when using a constant learning rate, the error $\mathbb{E}\|x_t-x^{\star}_t\|^2$ decays linearly in time $t$, 
until it reaches the ``noise\,+\,drift'' error $\eta\sigma^2/\mu + (\Delta/\mu\eta)^2$. 
Notice that the ``noise\,+\,drift'' error cannot be made arbitrarily small. 
This is perfectly in line with intuition: a learning rate that is too small prevents the algorithm from catching up with  $x_t^{\star}$. We note that the individual error terms due to the optimization and noise are classically known to be tight for PSG; tightness of the drift term is proved by \citet[Theorem 3.2]{madden2021bounds}.

With Theorem~\ref{thm:exp_dist} in hand, we  define the asymptotic tracking error of Algorithm~\ref{alg:sgd} corresponding to $\mathbb{E}\|x_t-x^{\star}_t\|^2$, together with the corresponding optimal step size:
\begin{equation*}
	\mathcal{E}:=\min_{\eta\in (0,1/2L]} \left\{\frac{\eta\sigma^2}{\mu} + \left(\frac{\Delta}{\mu\eta}\right)^2\right\} \quad\text{and}\quad\eta_\star := \min\left\{\frac{1}{2L}, \left(\frac{2\Delta^2}{\mu\sigma^2}\right)^{1/3}\right\}\!.
\end{equation*}
Plugging $\eta_{\star}$ into the definition of $\mathcal{E}$, we see that Algorithm~\ref{alg:sgd} exhibits qualitatively different behaviors in settings with high or low drift-to-noise ratio $\Delta/\sigma$. Explicitly, 
\begin{equation*}
	\mathcal{E}\asymp\begin{cases} 
		\frac{\sigma^2}{\mu L}+\left(\frac{L\Delta}{\mu}\right)^2 & \text{if~} \frac{\Delta}{\sigma}\geq \sqrt{\frac{\mu}{16L^3}} \\
		\left(\frac{\Delta\sigma^2}{\mu^2}\right)^{2/3} & \text{otherwise}.
	\end{cases}
\end{equation*}
Two regimes of variation are brought to light by the above computation: the \emph{high drift-to-noise regime} $\Delta/\sigma \geq \sqrt{\mu/16L^3}$ and the \emph{low drift-to-noise regime} $\Delta/\sigma< \sqrt{\mu/16L^3}$. The high drift-to-noise regime is uninteresting from the viewpoint of stochastic optimization because in this case the optimal learning rate $\eta_\star \asymp 1/L$ is as large as in the deterministic setting. In contrast, the low drift-to-noise regime is interesting because it necessitates using a smaller learning rate $\eta_{\star}\asymp (\Delta^2/\mu\sigma^2)^{1/3}$ that exhibits a nontrivial scaling with the problem parameters.

\paragraph{Learning rate vs. rate of variation.}
A central question is to find a learning rate schedule that achieves a tracking error $\mathbb{E}\|x_t-x^{\star}_t\|^2 $ that is within a constant factor of $\mathcal{E}$ in the shortest possible time. The answer is clear in the high drift-to-noise regime $\Delta/\sigma \geq \sqrt{\mu/16L^3}$. Indeed, in this case, Theorem~\ref{thm:exp_dist} directly implies that Algorithm~\ref{alg:sgd} with the constant learning rate $\eta_\star=1/2L$ will find a point $x_t$ satisfying 
$\mathbb{E}\|x_t-x_t^{\star}\|^2\lesssim \mathcal{E}$ in time $t \lesssim (L/\mu)\log(\|x_0-x_0^{\star}\|^2/\mathcal{E})$. Notice that this efficiency estimate is logarithmic in $1/\mathcal{E}$; intuitively, the reason for the absence of a sublinear component is that the error due to the drift $\Delta$ dominates the error due to the variance $\sigma^2$ in the stochastic gradient. 

The low drift-to-noise regime $\Delta/\sigma< \sqrt{\mu/16L^3}$ is more subtle. Namely, the simplest strategy is to execute Algorithm~\ref{alg:sgd} 
with the constant learning rate $\eta_\star = (2\Delta^2/\mu\sigma^2)^{1/3}$. Then a direct application of Theorem~\ref{thm:exp_dist} yields the estimate 
$\mathbb{E}\|x_t-x^{\star}_t\|^2 \lesssim \mathcal{E}$ in time $t\lesssim (\sigma^2/\mu^2\mathcal{E}) \log(\|x_0-x_0^\star\|^2/\mathcal{E})$. 
This efficiency estimate can be significantly improved by gradually decaying the learning rate using a ``step decay schedule'', wherein the algorithm is implemented in epochs with the new learning rate chosen to be the midpoint between the current learning rate and $\eta_{\star}$. 
Such schedules are well known to improve efficiency in the static setting, as was  discovered by  \citet{ghadimi2013optimal}, and can be used here. The end result is the following theorem; see Theorem~\ref{thm:time_to_track} for the formal statement.

\begin{theorem}[Time to track in expectation, informal]\label{thm:dist_schedul}
	Suppose that Assumption~\ref{assump:sec_moment} holds. Then there is a learning rate schedule $\{\eta_t\}$ such that Algorithm~\ref{alg:sgd} produces a point $x_t$ satisfying $$\mathbb{E}\|x_t-x_{t}^{\star}\|^2\lesssim \mathcal{E} \quad\text{in time}\quad t\lesssim \frac{L}{\mu}\log\!\left(\frac{ \|x_0-x_0^\star\|^2}{\mathcal{E}}\right)+\frac{\sigma^2}{\mu^2\mathcal{E}}.$$
\end{theorem}

The efficiency estimate in Theorem~\ref{thm:dist_schedul} is strikingly similar to the efficiency estimate in the static setting \citep{ghadimi2013optimal}, with $\mathcal{E}$ playing the role of the target accuracy $\varepsilon$. An elementary computation shows that in the low drift-to-noise regime, Theorem~\ref{thm:dist_schedul} improves the constant learning rate efficiency estimate when $\mathcal{E}$ is small, e.g., when $\mathcal{E}\leq\|x_0-x_0^{\star}\|^2/e^2$. Theorems~\ref{thm:exp_dist} and \ref{thm:dist_schedul} provide useful baseline guarantees for the performance of Algorithm~\ref{alg:sgd}. Nonetheless, these guarantees are all stated in terms of the \emph{expected} tracking error $\mathbb{E}\|x_t-x_{t}^{\star}\|^2$, and are therefore only meaningful if the entire algorithm can be repeated from scratch multiple times. There is no shortage of situations in which a learning algorithm is operating in real time and the time drift is irreversible; in such settings, the algorithm may only be executed once. These situations call for efficiency estimates that hold with high probability, rather than only in expectation.

\subsection{High-Probability Guarantees}\label{sec:hp1}
We next present high-probability guarantees for the tracking error $\|x_t-x_t^{\star}\|^2$. Proofs appear in Section~\ref{minimizerproofsprob}. We make the following standard light-tail assumptions on the minimizer drift and gradient noise~\citep{pmlr-v99-harvey19a,lan2012optimal, nemirovski2009robust}.

\begin{assumption}[Sub-Gaussian drift and noise]\label{assumpt_light}
	\textup{There exist constants $\Delta,\sigma >0$ such that the following two conditions hold for all $t\geq0$:
	\begin{enumerate}[label=(\roman*)]
		\item {\bf (Drift) } The drift $\Delta_t^2$ is sub-exponential conditioned on $\mathcal{F}_t$ with  parameter $\Delta^2$:
		\begin{equation*}
			\mathbb{E}\big[{\exp}(\lambda \Delta_t^2)\,|\, \mathcal{F}_t\big]\leq \exp(\lambda \Delta^2) \quad \text{for all} \quad  0\leq\lambda\leq \Delta^{-2}. 
		\end{equation*} 
		\item {\bf (Noise) } The noise $z_t$ is norm sub-Gaussian conditioned on $\mathcal{F}_t$ with parameter $\sigma/2$:
		\begin{equation*}
			\mathbb{P}\big\{\|z_t\| \geq \tau \,|\,\mathcal{F}_t\big\}\leq 2\exp(-2\tau^2/\sigma^2) \quad \text{for all} \quad \tau>0.
		\end{equation*}
	\end{enumerate}
	}
\end{assumption}

Note that the first item of Assumption~\ref{assumpt_light} is equivalent to asserting that the minimizer drift $\Delta_t$ is sub-Gaussian conditioned on $\mathcal{F}_t$ \citep[see][Lemma~2.7.6]{vershynin2018high}. Clearly Assumption~\ref{assumpt_light} implies Assumption~\ref{assump:sec_moment} with the same constants $\Delta,\sigma$. It is worthwhile to note some common settings in which Assumption \ref{assumpt_light} holds; the claims in Remark~\ref{rem:subg} below follow from standard results on sub-Gaussian random variables \citep{jin2019short,vershynin2018high}. 

\begin{remark}[Common settings for Assumption \ref{assumpt_light}]\label{rem:subg}
\textup{Fix constants $\Delta,\sigma>0$. If $\Delta_t$ is bounded by $\Delta$, then clearly $\Delta_t^2$ is sub-exponential (conditioned on $\mathcal{F}_t$) with parameter $\Delta^2$. Similarly, if $\|z_t\|$ is bounded by $\sigma/2$,
	then $z_t$ is norm sub-Gaussian (conditioned on $\mathcal{F}_t$) with parameter $\sigma/2$ (by Markov's inequality). Alternatively, if the increment $x_t^{\star} - x_{t+1}^{\star}$ is mean-zero sub-Gaussian conditioned on $\mathcal{F}_t$ with parameter $\Delta/\sqrt{d}$, then $x_t^{\star} - x_{t+1}^{\star}$ is mean-zero norm sub-Gaussian conditioned on $\mathcal{F}_t$ with parameter $2\sqrt{2}\cdot\Delta$ and hence $\Delta_t^2$ is sub-exponential conditioned on $\mathcal{F}_t$ with parameter $c\cdot\Delta^2$ for some absolute constant $c>0$. Similarly, if $z_t$ is sub-Gaussian conditioned on $\mathcal{F}_t$ with parameter $\sigma/4\sqrt{2d}$, then  $z_t$ is norm sub-Gaussian conditioned on $\mathcal{F}_t$ with parameter $\sigma/2$.}
\end{remark}

The following theorem shows that if Assumption~\ref{assumpt_light} holds, then the expected bound on $\|x_t-x_t^{\star}\|^2$ derived in Theorem~\ref{thm:exp_dist} holds with high probability.

\begin{theorem}[High-probability distance tracking]\label{HPtrack}
	Let $\{x_t\}$ be the iterates produced by Algorithm~\ref{alg:sgd} with constant learning rate $\eta\leq1/2L$, and suppose that Assumption~\ref{assumpt_light} holds.  
	Then there is an absolute constant $c>0$ such that for any specified $t\in \mathbb{N}$ and $\delta\in(0,1)$, the following estimate holds with probability at least $1-\delta$:
	\begin{samepage}
	\begin{equation*}
		\|x_t-x^\star_t\|^2\leq\big(1-\tfrac{\mu\eta}{2}\big)^t\|x_0-x^\star_0\|^2+c\left(\frac{\eta\sigma^2}{\mu}+\left(\frac{\Delta}{\mu\eta}\right)^2\right)\log\!\left(\frac{e}{\delta}\right)\!.
	\end{equation*}
	\end{samepage}
\end{theorem}

The proof of Theorem \ref{HPtrack} employs a technique used by \citet{pmlr-v99-harvey19a}. The idea is to build a careful recursion for the moment generating function of $\|x_t-x_t^{\star}\|^2$, leading to a one-sided sub-exponential tail bound. As a consequence of Theorem \ref{HPtrack}, we can again implement a step decay schedule to obtain the following efficiency estimate with high probability; see Theorem~\ref{thm:time_to_track2} for the formal statement.

\begin{theorem}[Time to track with high probability, informal]\label{thm:dist_schedul2}
	Suppose that Assumption~\ref{assumpt_light} holds and that we are in the low drift-to-noise regime $\Delta/\sigma<\sqrt{\mu/16L^3}$. Then there is a learning rate schedule $\{\eta_t\}$ such that for any specified $\delta\in(0,1)$, Algorithm~\ref{alg:sgd} produces a point $x_t$ satisfying 
	\begin{equation*}
		\|x_t-x_{t}^{\star}\|^2\lesssim \mathcal{E}\log\!\left(\frac{e}{\delta}\right)
	\end{equation*}
	with probability at least $1-\delta$ in time 
	\begin{equation*}
		t\lesssim \frac{L}{\mu}\log\!\left(\frac{ \|x_0-x_0^\star\|^2}{\mathcal{E}}\right)+\frac{\sigma^2}{\mu^2\mathcal{E}}.
	\end{equation*}
\end{theorem} \section{Tracking the Minimum Value}
\label{sec:trckminimum} 
The results outlined so far have focused on tracking the minimizer $x_t^{\star}$. In this section, we present results for tracking the minimum value $\varphi_{t}^{\star}$. These two goals are fundamentally different. Generally speaking, good bounds on the function gap along with strong convexity imply good bounds on the distance to the minimizer; the reverse implication is false. To this end, we require a stronger assumption on the variation of the functions $f_t$ in time $t$: rather than merely controlling the minimizer drift $\Delta_t$, we will assume control on the \emph{gradient drift} $$G_{i,t}:=\sup_{x}\|\nabla f_i(x)-\nabla f_t(x)\|.$$

Our strategy is to track the minimum value along the running average $\hat x_t$ of the iterates $x_t$ produced by Algorithm~\ref{alg:sgd}, as defined in Algorithm~\ref{alg:sgdavg} below.
The reason behind using this particular running average is brought to light in Section~\ref{minimalvalproofsexp}, where we apply a standard averaging technique (Appendix~\ref{apdx:avg}) to a one-step improvement along $x_t$ (Lemma~\ref{lem:avgrecdec}) to obtain the desired progress along $\hat{x}_t$ (Proposition~\ref{gapbound}). 

\begin{algorithm}[h!]
	\caption{Averaged Online Proximal Stochastic Gradient\hfill $\mathtt{\overline{PSG}}(x_{0},\mu,\{\eta_t\},T)$} \label{alg:sgdavg}
	
	{\bf Input}: initial $x_0 = \hat{x}_0$, strong convexity parameter $\mu$, and step sizes $\{\eta_t\}_{t=0}^{T-1}\subset (0,1/\mu)$
	
	{\bf Step} $t=0,\ldots,T-1$: 
	\begin{equation*}
		\begin{aligned}
			&\text{Select~} g_t = \widetilde{\nabla}f_t(x_t)\\
			&\text{Set~} x_{t+1}=\prox_{\eta_t r_t}(x_t-\eta_t g_t)\\
			&\text{Set~} \hat{x}_{t+1}=\Big(1-\tfrac{\mu\eta_t}{2-\mu\eta_t}\Big){\hat x}_t+\tfrac{\mu\eta_t}{2-\mu\eta_t}x_{t+1}
		\end{aligned}
	\end{equation*}
	
	{\bf Return} $\hat{x}_T$
	
\end{algorithm}

\subsection{Bounds in Expectation}\label{sec:minvalexp}
We begin with bounding the expected value $\mathbb{E}[\varphi_t(\hat x_t)-\varphi_t^\star]$. Proofs appear in Section~\ref{minimalvalproofsexp}. Analogous to Assumption \ref{assump:sec_moment}, we make the following assumption regarding drift and noise.

\begin{assumption}[Bounded second moments]\label{assmp:funcgap0}
\textup{The regularizers $r_t\equiv r$ are identical for all times $t$ and there exist constants $\Delta,\sigma>0 $ such that the following two conditions hold for all $0\leq i<t$:
\begin{enumerate}[label=(\roman*)]
	\item {\bf (Drift)} The gradient drift $G_{i,t}$ satisfies $\mathbb{E}\,G_{i,t}^2\leq (\mu\Delta|i-t|)^2.$
	\item {\bf (Noise)} The gradient noise $z_i$ satisfies $\mathbb{E}\|z_i\|^2\leq\sigma^2$ and $\mathbb{E}\langle z_i, x_t^\star \rangle=0$.
\end{enumerate}
}
\end{assumption}

These two assumptions are natural indeed. Taking into account Lemma~\ref{lem:grad_vs_dist}, it is clear that Assumption~\ref{assmp:funcgap0} implies the earlier Assumption~\ref{assump:sec_moment} with the same constants $\Delta, \sigma$. The drift assumption intuitively asserts that second moment of $G_{i,t}$ grows at most quadratically in time $|i-t|$. In particular, returning to Example~\ref{ex:2}, suppose that the distribution map $\mathcal{D}(\cdot)$ is $\varepsilon$-Lipschitz continuous in the Wasserstein-1 distance, the loss $\ell(\cdot,\xi)$ is $C^1$-smooth for all $\xi$, and the gradient $\nabla \ell (x,\cdot)$ is $\beta$-Lipschitz continuous for all $x$. Then the Kantorovich-Rubinste\u{\i}n duality theorem \citep{kantorovich1958space} directly implies $\mathbb{E}\,G_{i,t}^2 \leq (\varepsilon\beta)^2\,\mathbb{E}\|v_i-v_t\|^2$. Therefore, as long as the second moment $\mathbb{E}\|v_i-v_t\|^2$ scales quadratically in $|i-t|$, the desired drift assumption holds.
The assumption on the gradient noise stipulates a uniform bound on the second moment $\mathbb{E}\|z_i\|^2$ and that the condition $\mathbb{E}\langle z_i, x_t^\star \rangle=0$ holds. The latter property confers a weak form of uncorrelatedness between the gradient noise $z_i$ and the future minimizer $x_t^\star$, and holds automatically if the gradient noise and the minimizers evolve independently of each other, as would typically be the case for instance in Example~\ref{ex:2}.

The following theorem establishes an expected improvement guarantee for Algorithm~\ref{alg:sgdavg}.

\begin{theorem}[Expected function gap]\label{thm:exp_gap}
	Let $\{\hat{x}_t\}$ be the iterates produced by Algorithm~\ref{alg:sgdavg} with constant learning rate $\eta \leq 1/2L$, and suppose that Assumption~\ref{assmp:funcgap0} holds. Then the following bound holds for all $t\geq 0$:
	\begin{equation*}
		\mathbb{E}\big[\varphi_t(\hat x_t)-\varphi_t^\star\big] \lesssim  \underbrace{\big(1 - \tfrac{\mu\eta}{2}\big)^t\big(\varphi_0(x_0) - \varphi_0^\star\big)}_\text{optimization} + \underbrace{\eta\sigma^2}_\text{noise} + \underbrace{\frac{\Delta^2}{\mu\eta^2}}_\text{drift}.
	\end{equation*}
\end{theorem}

The ``noise\,+\,drift'' error term in Theorem~\ref{thm:exp_gap} coincides with $\mu$ times the error term in Theorem~\ref{thm:exp_dist}, as expected. With Theorem~\ref{thm:exp_gap} in hand, we are led to define the following asymptotic tracking error of Algorithm~\ref{alg:sgdavg} corresponding to $\mathbb{E}[\varphi_t(\hat{x}_t) - \varphi_t^\star]$:
\begin{equation*}
	\mathcal{G}:=\mu\mathcal{E}=\min_{\eta\in (0,1/2L]} \left\{\eta\sigma^2 + \frac{\Delta^2}{\mu\eta^2}\right\}\!.
\end{equation*}
The corresponding asymptotically optimal choice of $\eta$ is again given by $$\eta_\star = \min\left\{\frac{1}{2L}, \left(\frac{2\Delta^2}{\mu\sigma^2}\right)^{1/3}\right\}\!,$$ and the dichotomy governed by the drift-to-noise ratio $\Delta/\sigma$ remains:
\begin{equation*}
	\mathcal{G}\asymp\begin{cases} 
		\frac{\sigma^2}{L}+\frac{(L\Delta)^2}{\mu} & \text{if~} \frac{\Delta}{\sigma}\geq \sqrt{\frac{\mu}{16L^3}} \\
		\mu\!\left(\frac{\Delta\sigma^2}{\mu^2}\right)^{2/3} & \text{otherwise}.
	\end{cases}
\end{equation*}

In the high drift-to-noise regime $\Delta/\sigma \geq \sqrt{\mu/16L^3}$, Theorem~\ref{thm:exp_gap} directly implies that Algorithm~\ref{alg:sgdavg} with the constant learning rate $\eta_\star=1/2L$ finds a point $\hat{x}_t$ satisfying \mbox{$\mathbb{E}[\varphi_t(\hat{x}_t) - \varphi_t^\star] \lesssim \mathcal{G}$} in time $t \lesssim (L/\mu)\log((\varphi_0(x_0) - \varphi_0^\star)/\mathcal{G})$. In the low drift-to-noise regime $\Delta/\sigma< \sqrt{\mu/16L^3}$, another direct application of Theorem~\ref{thm:exp_gap} shows that Algorithm~\ref{alg:sgdavg} 
with the constant learning rate $\eta_\star = (2\Delta^2/\mu\sigma^2)^{1/3}$ finds a point $\hat{x}_t$ satisfying $\mathbb{E}[\varphi_t(\hat{x}_t) - \varphi_t^\star] \lesssim \mathcal{G}$ in time $t \lesssim (\sigma^2/\mu\mathcal{G})\log((\varphi_0(x_0) - \varphi_0^\star)/\mathcal{G})$. As before, this efficiency estimate can be significantly improved by implementing a step decay schedule. The end result is the following theorem; see Theorem~\ref{thm:time_to_track_gap_exp} for the formal statement.

\begin{theorem}[Time to track in expectation, informal]\label{thm:gap_schedul}
	Suppose that Assumption~\ref{assmp:funcgap0} holds. Then there is a learning rate schedule $\{\eta_t\}$ such that Algorithm~\ref{alg:sgdavg} produces a point $\hat{x}_t$ satisfying $$\mathbb{E}\big[\varphi_t(\hat{x}_t) - \varphi_t^\star\big] \lesssim \mathcal{G}\quad\text{in time}\quad t\lesssim \frac{L}{\mu}\log\!\left(\frac{\varphi_0(x_0) - \varphi_0^\star}{\mathcal{G}}\right)+\frac{\sigma^2}{\mu\mathcal{G}}.$$
\end{theorem}
In the low drift-to-noise regime, Theorem~\ref{thm:gap_schedul} improves the constant learning rate efficiency estimate when $\mathcal{G}$ is small, e.g., when $\mathcal{G} \leq (\varphi_0({x}_0) - \varphi_0^\star)/e^2$. 

\subsection{High-Probability Guarantees}\label{sec:minvalprob}
Next, we obtain high-probability analogues of Theorems~\ref{thm:exp_gap} and \ref{thm:gap_schedul}. Proofs appear in Section~\ref{minimalvalproofsprob}.
Naturally, such results should rely on light-tail assumptions on the gradient drift $G_{i,t}$ and the norm of the gradient noise $\|z_i\|$. We state the guarantees under an assumption of sub-Gaussian drift and noise (Assumption~\ref{assmp:funcgap2} below). In particular, we require that the gradient noise $z_i$ is mean-zero conditioned on the $\sigma$-algebra $$\mathcal{F}_{i,t}:=\sigma(\mathcal{F}_i,x_t^\star)$$ for all $0 \leq i < t$; the property $\mathbb{E}[z_i\,|\,\mathcal{F}_{i,t}]=0$ would follow from independence of the gradient noise $z_i$ and the future minimizer $x_t^\star$ and is very reasonable in light of Examples~\ref{ex:1} and \ref{ex:2}.

\begin{assumption}[Sub-Gaussian drift and noise]\label{assmp:funcgap2}
	\textup{The regularizers $r_t\equiv r$ are identical for all times $t$ and there exist constants $\Delta,\sigma>0 $ such that the following two conditions hold for all $0\leq i<t$:
	\begin{enumerate}[label=(\roman*)]
		\item {\bf (Drift)} The squared gradient drift $G_{i,t}^2$ is sub-exponential with parameter $(\mu\Delta|i-t|)^2$:
		\begin{equation*}
			\mathbb{E}\big[{\exp}\big(\lambda G_{i,t}^2\big)\big]\leq {\exp}\big(\lambda (\mu\Delta|i-t|)^2\big) \quad \text{for all} \quad  0\leq\lambda\leq (\mu\Delta|i-t|)^{-2}. 
		\end{equation*} 
		\item {\bf (Noise)} The gradient noise $z_i$ is mean-zero norm sub-Gaussian conditioned on $\mathcal{F}_{i,t}$ with parameter $\sigma/2$, i.e.,
		$\mathbb{E}[z_i \,|\, \mathcal{F}_{i,t}]=0$ and  
		\begin{equation*}
			\mathbb{P}\big\{\|z_i\| \geq \tau \,|\,\mathcal{F}_{i,t}\big\}\leq 2\exp(-2\tau^2/\sigma^2) \quad\text{for all}\quad \tau > 0.
		\end{equation*}
	\end{enumerate}
	}
\end{assumption}
Clearly the chain of implications holds:
\begin{quote}\centering 
	Assumption~\ref{assmp:funcgap2}$\,\implies\,$Assumption~\ref{assmp:funcgap0}$\,\implies\,$Assumption~\ref{assump:sec_moment}.
\end{quote}

The following theorem shows that if Assumption~\ref{assmp:funcgap2} holds, then the expected bound on $\varphi_t(\hat x_t)-\varphi_t^\star$ derived in Theorem~\ref{thm:exp_gap} holds with high probability.

\begin{theorem}[Function gap with high probability]\label{thm:hpbound2}
	Let $\{\hat{x}_t\}$ be the iterates produced by Algorithm \ref{alg:sgdavg} with constant learning rate $\eta \leq 1/2L$, and suppose that Assumption~\ref{assmp:funcgap2} holds. Then there is an absolute constant $c>0$ such that for any specified $t\in\mathbb{N}$ and $\delta\in(0,1)$, the following estimate holds with probability at least $1-\delta$:
	\begin{align*}
		\varphi_t(\hat x_t)-\varphi_t^\star \leq c\left(\big(1 - \tfrac{\mu\eta}{2}\big)^t\big(\varphi_0(x_0) - \varphi_0^\star\big) + \eta\sigma^2 + \frac{\Delta^2}{\mu\eta^2}\right)\log\!\left(\frac{e}{\delta}\right)\!.
	\end{align*}
\end{theorem}

The proof of Theorem~\ref{thm:hpbound2} is based on combining the generalized Freedman inequality of \citet{pmlr-v99-harvey19a} with careful control on the drift and noise in improvement guarantees for the proximal stochastic gradient method. The key observation is that although we do not have simple recursive control on the moment generating function of $\varphi_t(\hat x_t)-\varphi_t^\star$ (as we do with $\|x_t - x_t^\star\|^2$), we can instead control the tracking error $\varphi_t(\hat x_t)-\varphi_t^\star$  by leveraging control on the martingale $\sum_{i=0}^{t-1}\langle z_i, x_i - x_t^\star \rangle \zeta^{t - 1 - i}$, where $ \zeta = 1 - \mu\eta/(2 - \mu\eta)$. This martingale is self-regulating in the sense that its total conditional variance is bounded by the history of the process; the generalized Freedman inequality is precisely suited to bound such martingales with high probability. 

With Theorem~\ref{thm:hpbound2} in hand, we may implement a step decay schedule as before to obtain the following efficiency estimate; see Theorem~\ref{thm:time_to_track_gap_hp2} for the formal statement.

\begin{theorem}[Time to track with high probability, informal]\label{thm:gap_schedulhp2}
	Suppose that Assumption~\ref{assmp:funcgap2} holds and that we are in the low drift-to-noise regime $\Delta/\sigma<\sqrt{\mu/16L^3}$. Fix $\delta\in(0,1)$. Then there is a learning rate schedule $\{\eta_t\}$ such that Algorithm~\ref{alg:sgdavg} produces a point $\hat{x}_t$ satisfying 
	\begin{equation*}
		\varphi_t(\hat{x}_t) - \varphi_t^\star\lesssim \mathcal{G}\log\!\left(\frac{e}{\delta}\right)
	\end{equation*}
	with probability at least $1-K\delta$ in time 
	\begin{equation*}
		t\lesssim \frac{L}{\mu}\log\!\left(\frac{\varphi_0(x_0) - \varphi_0^\star}{\mathcal{G}}\right)+\frac{\sigma^2}{\mu\mathcal{G}}\log\!\left(\log\!\left(\frac{e}{\delta}\right)\!\right)\!, \quad \text{where~~}K\lesssim \log_2\!\left(\frac{1}{L}\cdot\left(\frac{\sigma^2\mu}{\Delta^2}\right)^{1/3}\right)\!.
	\end{equation*}
\end{theorem}
 \section{Extension to the Decision-Dependent Setting}
\label{decision-dependent}
In this section, we extend the framework and results of the previous sections to a much wider class of tracking problems. In particular, the material in this section is a strict generalization of all the results in the previous sections and can model the performative prediction framework of \citet{perdomo2020performative} in a time-dependent setting. 

Setting the stage, suppose that we have a family of functions $\{f_{t,x}(\cdot)\}$ indexed by time $t\in \mathbb{N}$ and points $x\in \mathbb{R}^d$. Upon replacing the function $f_t$ in the time-dependent problem \eqref{eqn:prob_online} by the function $f_{t,x}$ depending not only on the time $t$ but also on the decision variable $x$, we obtain the sequence of decision-dependent stochastic optimization problems
\begin{equation}\label{eq:prob_onlinedecnoeq}
	\min_{x\in\R^d}\, f_{t,x}(x) + r_t(x)
\end{equation}
indexed by $t$. Tracking the solutions to \eqref{eq:prob_onlinedecnoeq} is typically a challenging task due to the dual dependency of $f_{t,x}(x)$ on the decision $x$. To obtain a more tractable tracking problem, we may decouple this dependency on $x$ by introducing an auxiliary decision variable $u$ and considering the family of stochastic optimization problems
\begin{equation}\label{eq:prob_onlinedecdecoupled}
	\min_{u\in\R^d}\, f_{t,x}(u) + r_t(u)
\end{equation}
indexed by both $t$ and $x$. Instead of tracking the \emph{optimal} decisions solving \eqref{eq:prob_onlinedecnoeq}, our aim becomes to track the \emph{stable} decisions 
\begin{equation}\label{eqpoint}
	\bar{x}_t\in \argmin_{u\in\R^d}\, f_{t,\bar x_t}(u) + r_t(u)
\end{equation}
arising from \eqref{eq:prob_onlinedecdecoupled}. We call a point $\bar x_t$ satisfying \eqref{eqpoint} an \emph{equilibrium point} of \eqref{eq:prob_onlinedecdecoupled} at time $t$; observe that $\bar x_t$ is stable in the sense that it is a fixed point of the map $x\mapsto\argmin_u\{ f_{t,x}(u) + r_t(u) \}$.

 Reasonable regularity assumptions on the family $\{f_{t,x}(\cdot)\}$ ensure that \eqref{eq:prob_onlinedecdecoupled} admits a unique equilibrium point at each time $t$ (see Assumption~\ref{assmp:spacegradvar} and Lemma~\ref{lem:eqpoint} below).
 When the functions $f_{t,x}$ are independent of $x$, the equilibrium points are simply the minimizers of $f_t +r_t$---the content of the previous sections. Our goal in this section is to track the equilibrium points $\bar{x}_t$, or equivalently to track the minimizers of the time-dependent stochastic optimization problem 
\begin{equation}\label{eqn:prob_onlinedec}
	\min_{u\in\R^d}\, f_{t,\bar x_t}(u) + r_t(u).
\end{equation}
Formally, \eqref{eqn:prob_onlinedec} is an example of \eqref{eqn:prob_online}, but this viewpoint is not directly useful since ${\bar x}_t$ is unknown.
 This more general framework allows us to model more dynamic settings. The main example stems from the setting of performative prediction introduced by \citet{perdomo2020performative}. This will be a running example throughout the section.

\begin{example}[Performative prediction]\label{ex:perfpred}
{\rm Within the framework of performative prediction, the functions take the form $f_{t,x}(u) = \mathbb{E}_{\xi\sim \mathcal{D}(t,x)}\ell(u,\xi)$ for some family of distributions $\mathcal{D}(t,x)$ indexed by both the time $t$ and the decision variable $x$. The motivation for the dependence of the distribution on $x$ is that often deployment of a learning rule parametrized by $x$ causes the population to change their profile to increase the likelihood of a better personal outcome---a process called ``gaming''. In other words, the population data is a function of the decision taken by the learner. Moreover, the dependence of the population data on time appears naturally when the population evolves due to exogenous temporal effects (e.g., seasonal, economic). The equilibrium points ${\bar x}_t$ have a clear meaning in this context. Namely, ${\bar x}_t$ is an equilibrium point if the learner has no reason deviate from the learning rule ${\bar x}_t$ based on the response distribution $\mathcal{D}(t,{\bar x}_t)$ alone. 
	
Whenever we refer back to this example, we will impose the following assumptions that are direct extensions of \citet{perdomo2020performative} to the time-dependent setting.	
Namely, fix a nonempty metric space $M$ equipped with its Borel $\sigma$-algebra and let $P_1(M)$ denote the space of Radon probability measures on $M$ with finite first moment, equipped with the Wasserstein-1 distance $W_1$. We make the natural assumption that there exist constants $\theta,\varepsilon\geq0$ such that the distribution map $\mathcal{D}(\cdot,\cdot)$
satisfies the following Lipschitz condition:
\begin{equation*}\label{distvar}
	W_1\big(\mathcal{D}(i,x),\mathcal{D}(t,y)\big)\leq \theta|i-t| + \varepsilon\|x-y\|  \quad\text{for all}\quad(i,x),(t,y)\in\mathbb{N}\times\mathbb{R}^d.
\end{equation*}
Moreover, we suppose that the loss function $\ell\colon\mathbb{R}^d\times M\rightarrow\mathbb{R}$ has the following three properties: $\ell(u,\cdot)\in L^1(\pi)$ for all $u\in\mathbb{R}^d$ and $\pi\in P_1(M)$; $\ell(\cdot,\xi)$ is $C^1$-smooth for all $\xi\in M$; and there is a constant $\beta\geq0$ such that the map $\xi\mapsto\nabla\ell(u,\xi)$ is $\beta$-Lipschitz continuous for all $u\in\mathbb{R}^d$, where $\nabla \ell(u,\xi)$ denotes the gradient of $\ell(\cdot,\xi)$ evaluated at $u$.
These assumptions directly imply the following stability property of the gradients with respect to distributional perturbations \citep[see][Lemma~2.1]{drusvyatskiy2020stochastic}:
\begin{equation}\label{jointgradvar}
	\sup_{u\in\mathbb{R}^d}\|\nabla f_{i,x}(u) - \nabla f_{t,y}(u)\| \leq \theta\beta|i-t| + \varepsilon\beta\|x-y\|  \quad\text{for all}\quad(i,x),(t,y)\in\mathbb{N}\times\mathbb{R}^d.
\end{equation}
Suppose now that each expected loss $f_{t,x}(\cdot)$ is $\mu$-strongly convex. In this case, \citet{perdomo2020performative} identified the optimal parameter regime $\varepsilon\beta<\mu$ wherein the repeated minimization procedure $y_{k+1} = \argmin_u\{ f_{t,y_{k}}(u) + r_t(u)\}$ converges to the unique equilibrium point $\bar{x}_t$ at a linear rate as $k\rightarrow\infty$  \citep[see][Theorem~3.5 and  Proposition~3.6]{perdomo2020performative}. The gradient stability property \eqref{jointgradvar} will lead us to the corresponding parameter regime in our generalized setting. 
}
\end{example}

\subsection{Decision-Dependent Framework}\label{sec:decframework}

We begin by recording the assumptions of our framework. Similar to the previous sections, we assume that the following standard regularity conditions hold:
\begin{enumerate}[label=(\roman*)]
\item Each function $f_{t,x}\colon\mathbb{R}^d\to\mathbb{R}$ is $\mu$-strongly convex and $C^1$-smooth with $L$-Lipschitz continuous gradient for some common parameters $\mu,L>0$.
\item Each regularizer $r_t\colon\mathbb{R}^d\to\mathbb{R}\cup\{\infty\}$ is proper, closed, and convex.
\end{enumerate}
For each $t\in\mathbb{N}$ and $x,u\in\mathbb{R}^d$, we let $\nabla f_{t,x}(u)$ denote the gradient of the function $f_{t,x}(\cdot)$ evaluated at $u$. In order to control the variation of the family $\{f_{t,x}(\cdot)\}$ in the decision variable $x$, we introduce the parameter
	\begin{equation*}
	\gamma:= \text{\raisebox{3pt}{$\sup_{\substack{t\in\mathbb{N},\, u, x, y\in\mathbb{R}^d \\ x\ne y}}$}} \frac{\|\nabla f_{t,x}(u) - \nabla f_{t,y}(u)\|	}{\|x-y\|}
	\end{equation*}
	and impose throughout Section~\ref{decision-dependent} the following stability property of the gradients with respect to the decision variable.

\begin{assumption}[Gradient stability in the decision variable]\label{assmp:spacegradvar}
	\textup{The following parameter regime holds: $\gamma < \mu.$}
\end{assumption}

In particular, $\gamma$ is finite and the following Lipschitz bound holds: 
		\begin{equation*}
				\sup_{t\in\mathbb{N},\,u\in\mathbb{R}^d}\|\nabla f_{t,x}(u) - \nabla f_{t,y}(u)\| \leq \gamma \|x-y\|\quad\text{for all}\quad x,y\in\mathbb{R}^d.
		\end{equation*}
Returning to Example~\ref{ex:perfpred}, it follows from (\ref{jointgradvar}) that Assumption~\ref{assmp:spacegradvar} holds whenever $\varepsilon\beta<\mu$. As the following lemma shows, the requirement $\gamma<\mu$ guarantees that for each $t\in\mathbb{N}$, the equilibrium point $\bar{x}_t$ is well defined and unique.

\begin{lemma}[Existence of equilibrium]\label{lem:eqpoint}
	For each $t\in\mathbb{N}$, the map $S_t\colon\R^d\rightarrow\R^d$ given by 
	\begin{equation*}
		S_t(x) = \argmin_{u\in\R^d}\, f_{t,x}(u)+r_t(u)
	\end{equation*}
	is $(\gamma/\mu)$-contractive and therefore has a unique fixed point $\bar{x}_t$ to which the repeated minimization procedure $y_{k+1} = S_t(y_{k})$ converges at a linear rate as $k\rightarrow\infty$.
\end{lemma}
\begin{proof}
	Note first that $S_t$ is well defined by the strong convexity of each function $\varphi_{t,x}:= f_{t,x}+r_t$. Next, given $x,y\in\mathbb{R}^d$, observe that we have the first-order optimality conditions $0\in\partial\varphi_{t,x}(S_t(x))$ and $0\in\partial\varphi_{t,y}(S_t(y))$; this last inclusion implies $-\nabla f_{t,y}(S_t(y)) \in\partial r_t(S_t(y))$ and hence $\nabla f_{t,x}(S_t(y)) - \nabla f_{t,y}(S_t(y))\in \partial\varphi_{t,x}(S_t(y))$. On the other hand, the $\mu$-strong convexity of $\varphi_{t,x}$ implies that for all $u,u'\in \dom\varphi_{t,x}$, $w\in\partial\varphi_{t,x}(u)$, and $w'\in\partial\varphi_{t,x}(u')$, we have 
		\begin{equation*}
			\mu\|u - u'\| \leq \|w - w'\|. 
		\end{equation*}
		Thus, taking $u = S_t(x)$, $w = 0$, $u' = S_t(y)$, and $w'=\nabla f_{t,x}(S_t(y)) - \nabla f_{t,y}(S_t(y))$ yields $$\mu\|S_t(x)-S_t(y)\| \leq \|\nabla f_{t,x}(S_t(y)) - \nabla f_{t,y}(S_t(y))\| \leq \gamma\|x-y\|,$$ where the last inequality holds by the definition of $\gamma$. Hence $\|S_t(x)-S_t(y)\| \leq (\gamma/\mu)\|x-y\|$, so $S_t$ is $(\gamma/\mu)$-contractive since $\gamma<\mu$. An application of the Banach fixed-point theorem completes the proof.
\end{proof}

It is easy to see that the parameter regime $\gamma <\mu$ is optimal in the sense that equilibrium points can fail to exist whenever $\gamma \geq \mu$, as illustrated in the following example.

\begin{example}[Optimality of the regime $\gamma<\mu$]
	\textup{Consider the time-independent family of functions given by $f_{x}(u) = \tfrac{1}{2}\|u - ax - b\|^2$ for any fixed constant $a\geq 1$ and vector $b\in\mathbb{R}^d$. Then $f_x$ is $1$-strongly convex and smooth with $1$-Lipschitz continuous gradient, and  
	    \begin{equation*}
		\gamma = \text{\raisebox{3pt}{$\sup_{\substack{u, x, y\in\mathbb{R}^d \\ x\ne y}}$}} \frac{\|\nabla f_{x}(u) - \nabla f_{y}(u)\|	}{\|x-y\|} = a\geq 1 = \mu.
		\end{equation*}
	   Now let $\mathcal{X}\subset\mathbb{R}^d$ be a nonempty closed convex set and take $r = \delta_\mathcal{X}$ to be the  convex indicator of $\mathcal{X}$. Then $\bar{x}$ is an equilibrium point of the decoupled family of problems $\min_u\{f_{x}(u)+r(u)\}$ if and only if
	    $$\bar{x}\in \argmin_{u\in\R^d}\{f_{\bar x}(u)+r(u)\} = \{\proj_{\mathcal{X}}(a\bar{x}+b)\}.$$
	Taking $a=1$, $b\neq0$, and $\mathcal{X}=\mathbb{R}^d$, we see that $\gamma = \mu$ and no equilibrium point exists. On the other hand, if we take $\mathcal{X} = \mathbb{R}_{+}^d$ to be the nonnegative orthant and $b>0$, then for any $a \geq 1$ and $x\in \mathcal{X}$ we have
	\begin{equation*}
		x < ax + b = \proj_{\mathcal{X}}(ax+b)
	\end{equation*}
	and hence no equilibrium point exists for this problem with any value $\gamma \in [\mu,\infty)$.}
\end{example}

Next, we turn to tracking the equilibria $\bar{x}_t$ furnished by Lemma~\ref{lem:eqpoint} using a decision-dependent proximal stochastic gradient method. Specifically, we make the standing assumption that at every time $t$, and at every query point $x$, the learner may obtain an \emph{unbiased estimator} $\widetilde{\nabla}f_{t,x}(x)$ of the true gradient $\nabla f_{t,x}(x)$. With this oracle access, the decision-dependent proximal stochastic gradient method---recorded as Algorithm~\ref{alg:decsgd} below---selects in each iteration $t$ the stochastic gradient $g_t = \widetilde{\nabla}f_{t,x_t}(x_t)$ and takes the step 
\begin{equation*}
	x_{t+1} := \prox_{\eta_t r_t}(x_t-\eta_tg_t) = \argmin_{u\in\mathbb{R}^d}\left\{r_t(u)+\tfrac{1}{2\eta_t}\|u-(x_t-\eta_tg_t)\|^2\right\} 
\end{equation*}
using step size $\eta_t>0$. As before, our goal is to obtain efficiency estimates for this procedure that hold both in expectation and with high probability. 

\begin{algorithm}[h!]
	\caption{Decision-Dependent PSG\hfill $\mathtt{D\text{-}PSG}(x_{0},\{\eta_t\},T)$}\label{alg:decsgd}
	
	{\bf Input}: initial $x_0$ and step sizes $\{\eta_t\}_{t=0}^{T-1}\subset (0,\infty)$
	
	{\bf Step} $t=0,\ldots,T-1$: 
	\begin{equation*}
		\begin{aligned}
			&\text{Select~} g_t = \widetilde{\nabla} f_{t,x_t}(x_t) \\
			&\text{Set~} x_{t+1}=\prox_{\eta_t r_t}(x_t-\eta_t g_t)
		\end{aligned}
	\end{equation*}
	
	{\bf Return} $x_T$
	
\end{algorithm}

The guarantees we obtain allow both the iterates $x_t$ \emph{and the equilibria} $\bar{x}_t$ to evolve stochastically. Given $\{x_t\}$ and $\{g_t\}$ as in Algorithm~\ref{alg:decsgd}, we let $$z_t := \nabla f_{t,x_t}(x_t) - g_t$$ denote the \emph{gradient noise} at time $t$ and we impose the following assumption modeling stochasticity on a fixed probability space $(\Omega,\mathcal{F},\mathbb{P})$ throughout Section~\ref{decision-dependent}.

\begin{assumption}[Stochastic framework]\label{assmp:stochframdec}
	\textup{There exists a probability space $(\Omega,\mathcal{F},\mathbb{P})$ with filtration $(\mathcal{F}_t)_{t\geq0}$ such that $\mathcal{F}_0=\left\{\emptyset,\Omega\right\}$ and the following two conditions hold for all $t\geq0$:
		\begin{enumerate}[label=(\roman*)]
			\item\label{it1dec} $x_t,\bar{x}_t\colon\Omega\rightarrow\mathbb{R}^d$ are $\mathcal{F}_t$-measurable.
			\item\label{it2dec} $z_t\colon\Omega\rightarrow\mathbb{R}^d$ is $\mathcal{F}_{t+1}$-measurable with $\mathbb{E}[z_t\,|\,\mathcal{F}_t]=0$.
		\end{enumerate} 
	}
\end{assumption}

The first item of Assumption~\ref{assmp:stochframdec} formalizes the assertion that $x_t$ and $\bar{x}_t$ are fully determined by information up to time $t$. The second item of Assumption~\ref{assmp:stochframdec} formalizes the assertion that the gradient noise $z_t$ is fully determined by information up to time $t+1$ and has zero mean conditioned on the information up to time $t$, i.e., $g_t$ is an unbiased estimator of $\nabla f_{t,x_t}(x_t)$; for example, this holds naturally in Example~\ref{ex:perfpred} if we take $g_t = \nabla\ell(x_t,\xi_t)$ with $\xi_t\sim \mathcal{D}(t,x_t)$.

Finally, we fix some notation to be used henceforth. We define the positive parameter $$\bar{\mu}:=\mu - \gamma,$$ and we define the \emph{equilibrium drift} $\bar{\Delta}_t$ and the \emph{temporal gradient drift} $\bar{G}_{i,t}$ to be the random variables $$\bar{\Delta}_t := \|\bar{x}_t - \bar{x}_{t+1}\|\quad\text{and}\quad\bar{G}_{i,t}:=\sup_{u,x\in\mathbb{R}^d}\|\nabla f_{i,x}(u) - \nabla f_{t,x}(u)\|.$$ Note that in the setting of Example~\ref{ex:perfpred}, the estimate (\ref{jointgradvar}) implies $\bar{G}_{i,t}\leq \theta\beta|i-t|$ and hence $\bar{\Delta}_t \leq \theta\beta/\bar\mu$ by Lemma~\ref{lem:grad_vs_dist}, provided the regularizers $r_t\equiv r$ are identical for all times $t$. We also set $$\varphi_t :=f_{t,x_t} + r_t,\quad  x_t^\star:=\argmin\varphi_t,\quad\varphi_t^\star := \min\varphi_t$$ and $$\psi_t := f_{t,\bar{x}_t}+r_t\quad\text{and}\quad\psi_t^\star:=\min\psi_t.$$ In particular, the equilibrium point $\bar{x}_t$ is the minimizer of the \emph{equilibrium function} $\psi_t$, and $\psi_t^\star$ denotes its minimum value. Observe that when $\gamma=0$, we have $\varphi_t = \psi_t +c_t$ for some constant of integration $c_t$ and hence we recover the setting of Section~\ref{sec:framework} with $x_t^\star = \bar{x}_t$.

\subsection{Tracking the Equilibrium Point}\label{sec:eqtrck}

In a nutshell, the results of Section~\ref{sec:trckminimizer} extend directly to tracking the equilibrium points $\bar x_t$, with $\mu$ replaced by $\bar \mu$ and $\Delta$ replaced by $\bar \Delta$ (defined in Assumption~\ref{assump:sec_moment_dec} below). We begin with bounding the expected value $\mathbb{E}\|x_t-\bar{x}_t\|^2$. Due to the fact that Algorithm~\ref{alg:decsgd} takes steps on the current functions $\varphi_t$ but the minimizers we aim to track are those of the equilibrium functions $\psi_t$, we will rely at the outset on controlling the function gaps
\begin{equation*}
	\big[f_{i,x}(u) - f_{i,x}(v)\big] - \big[f_{t,y}(u) - f_{t,y}(v)\big] 
\end{equation*}
(at first for $i = t$, and later for general $i$ and $t$). We achieve this control in terms of the temporal gradient drift. 
\begin{lemma}[Function gap variation]\label{lem:funcgap}
	For all $(i,x),(t,y)\in\mathbb{N}\times\mathbb{R}^d$ and $u,v\in\mathbb{R}^d$, we have
	\begin{equation*}
		\big\lvert\big[f_{i,x}(u) - f_{i,x}(v)\big] - \big[f_{t,y}(u) - f_{t,y}(v)\big]\big\rvert \leq \big(\bar{G}_{i,t}+\gamma\|x-y\|\big)\|u-v\|.  
	\end{equation*}
\end{lemma}
\begin{proof}
	Fix $(i,x),(t,y)\in\mathbb{N}\times\mathbb{R}^d$ and $u,v\in\mathbb{R}^d$, and set $u_\tau := v + \tau(u-v)$ for all $\tau\in[0,1]$. By the fundamental theorem of calculus and Cauchy-Schwarz, we have
	\begin{equation*}
		\begin{split}
			\big[f_{i,x}(u) - f_{i,x}(v)\big] - \big[f_{t,y}(u) - f_{t,y}(v)\big] &= \int_{0}^{1}\big\langle \nabla f_{i,x}(u_\tau) - \nabla f_{t,y}(u_\tau), u-v\big\rangle\,d\tau \\
			&\leq \big(\bar{G}_{i,t}+\gamma\|x-y\|\big)\|u-v\|.
		\end{split}
	\end{equation*}
	Switching $(i,x)$ and $(t,y)$ completes the proof.
\end{proof}

Using Lemmas~\ref{lem:onestep0} and \ref{lem:funcgap}, we obtain the following equilibrium one-step improvement. 

\begin{lemma}[Equilibrium one-step improvement]\label{lem:onestep0dec}
	The iterates $\{x_t\}$ produced by Algorithm~\ref{alg:decsgd} with $\eta_t<1/L$ satisfy the bound:
	\begin{equation*}
		\begin{split}
			2\eta_t(\psi_t(x_{t+1})-\psi_t^\star) \leq (1-\bar\mu\eta_t)\|x_t-\bar{x}_t\|^2 - (1 &- \gamma\eta_t)\|x_{t+1}-\bar{x}_t\|^2 \\
			&+ 2\eta_t\langle z_t, x_t - \bar{x}_t \rangle + \tfrac{\eta_t^2}{1-L\eta_t}\|z_t\|^2.
		\end{split}
	\end{equation*}
\end{lemma}
\begin{proof}
	By Lemma~\ref{lem:funcgap}, we have
	\begin{equation*}
		\begin{split}	
			\big[\psi_t(x_{t+1})-\psi_t(\bar{x}_t)\big] - &\big[\varphi_t(x_{t+1}) - \varphi_t(\bar{x}_t)\big]\\ &= 	\big[f_{t,\bar{x}_t}(x_{t+1})-f_{t,\bar{x}_t}(\bar{x}_t)\big] - \big[f_{t,x_t}(x_{t+1})-f_{t,x_t}(\bar{x}_t)\big]\\
			&\leq \gamma\|x_t-\bar{x}_t\|\|x_{t+1}-\bar{x}_t\|.
		\end{split} 
	\end{equation*}
	Hence
	\begin{equation*}
		\psi_t(x_{t+1})-\psi_t^\star \leq \varphi_t(x_{t+1}) - \varphi_t(\bar{x}_t) + \gamma\|x_t-\bar{x}_t\|\|x_{t+1}-\bar{x}_t\|.
	\end{equation*}
	Moreover, Young's inequality implies
	\begin{equation*}
		 \gamma\|x_t-\bar{x}_t\|\|x_{t+1}-\bar{x}_t\| \leq \tfrac{\gamma}{2}\|x_t - \bar{x}_t\|^2 + \tfrac{\gamma}{2}\|x_{t+1}-\bar{x}_t\|^2.
	\end{equation*}
	Multiplying through by $2\eta_t$ and applying Lemma~\ref{lem:onestep0} completes the proof.
\end{proof}

For simplicity, we state the main results under the assumption that the second moments $\mathbb{E}\,\bar{\Delta}_t^2$ and $\mathbb{E}\|z_t\|^2$ are uniformly bounded; more general guarantees that take into account weighted averages of the moments and allow for time-dependent learning rates follow from Lemma~\ref{lem:onestep0dec} as well. 

\begin{assumption}[Bounded second moments]\label{assump:sec_moment_dec}
	\textup{There exist constants $\bar{\Delta},\sigma>0$ such that the following two conditions hold for all $t\geq0$:
		\begin{enumerate}[label=(\roman*)]
			\item {\bf (Drift) } The equilibrium drift $\bar{\Delta}_t$ satisfies $\mathbb{E}\,\bar{\Delta}_t^2 \leq \bar{\Delta}^2$. 
			\item {\bf (Noise) } The gradient noise $z_t$ satisfies $\mathbb{E}\|z_t\|^2\leq\sigma^2$.
		\end{enumerate}
	}
\end{assumption}

The following theorem establishes an expected improvement guarantee for Algorithm~\ref{alg:decsgd}, thereby extending Theorem~\ref{thm:exp_dist}; see Section~\ref{minimizerproofsexp} for the precise statement (Corollary~\ref{cor1}) and proof. 

\begin{theorem}[Expected distance]\label{thm:exp_dist_dec}
	Suppose that Assumption~\ref{assump:sec_moment_dec} holds. Then the iterates produced by Algorithm~\ref{alg:decsgd} with constant learning rate $\eta\leq 1/2L$ satisfy the  bound:  
	\begin{equation*}
		\mathbb{E}\|x_t-\bar{x}_t\|^2 \lesssim \underbrace{(1-\bar{\mu}\eta)^t\|x_0-\bar{x}_0\|^2}_\text{optimization} + \underbrace{\frac{\eta\sigma^2}{\bar\mu}}_\text{noise} + \underbrace{\bigg(\frac{\bar\Delta}{\bar{\mu}\eta}\bigg)^2}_\text{drift}\!.
	\end{equation*}
\end{theorem}

With Theorem~\ref{thm:exp_dist_dec} in hand, we are led to define the following asymptotic tracking error of Algorithm~\ref{alg:decsgd} corresponding to $\mathbb{E}\|x_t-\bar{x}_t\|^2$, together with the corresponding optimal step size:
\begin{equation*}
	\bar{\mathcal{E}}:=\min_{\eta\in (0,1/2L]} \left\{\frac{\eta\sigma^2}{\bar\mu} + \left(\frac{\bar\Delta}{\bar\mu\eta}\right)^2\right\} \quad\text{and}\quad \bar\eta_\star := \min\left\{\frac{1}{2L}, \left(\frac{2\bar{\Delta}^2}{\bar\mu\sigma^2}\right)^{1/3}\right\}\!.
\end{equation*}
Plugging $\bar\eta_{\star}$ into the definition of $\bar{\mathcal{E}}$, we see that Algorithm~\ref{alg:decsgd} exhibits qualitatively different behaviors in settings corresponding to high or low drift-to-noise ratio $\bar\Delta/\sigma$:
\begin{equation*}
	\bar{\mathcal{E}}\asymp\begin{cases} 
		\frac{\sigma^2}{\bar\mu L}+\left(\frac{L\bar\Delta}{\bar\mu}\right)^2 & \text{if } \frac{\bar\Delta}{\sigma}\geq \sqrt{\frac{\bar\mu}{16L^3}} \\
		\left(\frac{\bar\Delta\sigma^2}{\bar\mu^2}\right)^{2/3} & \text{otherwise}.
	\end{cases}
\end{equation*}
As before, the \emph{high drift-to-noise regime} $\bar\Delta/\sigma \geq \sqrt{\bar\mu/16L^3}$ is uninteresting from the viewpoint of stochastic optimization and we focus on the \emph{low drift-to-noise regime} $\bar\Delta/\sigma< \sqrt{\bar\mu/16L^3}$. The following theorem extends Theorem~\ref{thm:dist_schedul}; see Theorem~\ref{thm:time_to_track} for the formal statement and proof.

\begin{theorem}[Time to track in expectation, informal]\label{thm:dist_schedul_dist}
	Suppose that Assumption~\ref{assump:sec_moment_dec} holds. Then there is a learning rate schedule $\{\eta_t\}$ such that Algorithm~\ref{alg:decsgd} produces a point $x_t$ satisfying $$\mathbb{E}\|x_t-\bar{x}_t\|^2\lesssim \bar{\mathcal{E}} \quad\text{in time}\quad t\lesssim \frac{L}{\bar\mu}\log\!\left(\frac{ \|x_0-\bar{x}_0\|^2}{\bar{\mathcal{E}}}\right) + \frac{\sigma^2}{\bar\mu^2\bar{\mathcal{E}}}.$$
\end{theorem}

Next, we present high-probability guarantees for the tracking error $\|x_t-\bar{x}_t\|^2$ under the following standard light-tail assumption on the equilibrium drift and gradient noise. 

\begin{assumption}[Sub-Gaussian drift and noise]\label{assumpt_light_dec}
	\textup{There exist constants $\bar{\Delta},\sigma >0$ such that the following two conditions hold for all $t\geq0$:
		\begin{enumerate}[label=(\roman*)]
			\item {\bf (Drift) } The drift $\bar{\Delta}_t^2$ is sub-exponential conditioned on $\mathcal{F}_t$ with  parameter $\bar{\Delta}^2$:
			\begin{equation*}
				\mathbb{E}\big[{\exp}(\lambda \bar{\Delta}_t^2)\,|\, \mathcal{F}_t\big]\leq \exp(\lambda \bar{\Delta}^2) \quad \text{for all} \quad  0\leq\lambda\leq \bar{\Delta}^{-2}. 
			\end{equation*} 
			\item {\bf (Noise) } The noise $z_t$ is norm sub-Gaussian conditioned on $\mathcal{F}_t$ with parameter $\sigma/2$:
			\begin{equation*}
				\mathbb{P}\big\{\|z_t\| \geq \tau \,|\,\mathcal{F}_t\big\}\leq 2\exp(-2\tau^2/\sigma^2) \quad \text{for all} \quad \tau>0.
			\end{equation*}
		\end{enumerate}
	}
\end{assumption}

Note that the first item of Assumption~\ref{assumpt_light_dec} is equivalent to asserting that the equilibrium drift $\bar{\Delta}_t$ is sub-Gaussian conditioned on $\mathcal{F}_t$, and that this condition holds trivially in the setting of Example~\ref{ex:perfpred} with $\bar\Delta = \theta\beta/\bar\mu$ provided the regularizers $r_t\equiv r$ are identical for all times $t$. Clearly Assumption~\ref{assumpt_light_dec} implies Assumption~\ref{assump:sec_moment_dec} with the same constants $\bar{\Delta},\sigma$. The following theorem shows that if Assumption~\ref{assumpt_light_dec} holds, then the expected bound on $\|x_t-\bar{x}_t\|^2$ derived in Theorem~\ref{thm:exp_dist_dec} holds with high probability.

\begin{theorem}[High-probability distance tracking]\label{HPtrackdec}
	Let $\{x_t\}$ be the iterates produced by Algorithm~\ref{alg:decsgd} with constant learning rate $\eta\leq1/2L$, and suppose that Assumption~\ref{assumpt_light_dec} holds. 
	Then there is an absolute constant $c>0$ such that for any specified $t\in \mathbb{N}$ and $\delta\in(0,1)$, the following estimate holds with probability at least $1-\delta$:
	\begin{equation}\label{ineq:HPtrackdec}
		\|x_t-\bar{x}_t\|^2\leq\big(1-\tfrac{\bar\mu\eta}{2}\big)^t\|x_0-\bar{x}_0\|^2+c\left(\frac{\eta\sigma^2}{\bar\mu} + \bigg(\frac{\bar\Delta}{\bar{\mu}\eta}\bigg)^2\right)\log\!\left(\frac{e}{\delta}\right)\!.
	\end{equation}
\end{theorem}

Theorem~\ref{HPtrackdec} is an extension of Theorem~\ref{HPtrack}. As a consequence of Theorem~\ref{HPtrackdec}, we can again implement a step decay schedule in the low drift-to-noise regime to obtain the following efficiency estimate with high probability, thereby extending Theorem~\ref{thm:dist_schedul2}; see Section~\ref{minimizerproofsprob} for the precise statements (Theorems~\ref{HPtrack1} and \ref{thm:time_to_track2}) and proofs.

\begin{theorem}[Time to track with high probability, informal]\label{thm:dist_schedul2_dec} Suppose that Assumption \ref{assumpt_light_dec} holds and that we are in the low drift-to-noise regime $\bar{\Delta}/\sigma<\sqrt{\bar{\mu}/16L^3}$. Then there is a learning rate schedule $\{\eta_t\}$ such that for any specified $\delta\in(0,1)$, Algorithm~\ref{alg:decsgd} produces a point $x_t$ satisfying 
	\begin{equation*}
		\|x_t-\bar{x}_t\|^2\lesssim \bar{\mathcal{E}}\log\!\left(\frac{e}{\delta}\right)
	\end{equation*}
	with probability at least $1-\delta$ in time 
	\begin{equation*}
		t\lesssim \frac{L}{\bar\mu}\log\!\left(\frac{ \|x_0-\bar{x}_0\|^2}{\bar{\mathcal{E}}}\right) + \frac{\sigma^2}{\bar\mu^2\bar{\mathcal{E}}}.
	\end{equation*}
\end{theorem}

\subsection{Tracking the Equilibrium Value}\label{sec:eqvaltrck}

The results outlined so far have focused on tracking the equilibrium point $\bar{x}_t$, i.e., the minimizer of $\psi_t$. In this section, we present results for tracking the equilibrium value $\psi_{t}^{\star}$ in the parameter regime 
\begin{equation}\label{eqgapregime}
	\gamma<\mu/2.
\end{equation}
The regime \eqref{eqgapregime} matches the one used in Theorem 7.3 of \citet{drusvyatskiy2020stochastic} to obtain function gap bounds for biased PSG along an average iterate, and we employ a similar averaging technique to obtain our bounds.

Imposing the regime (\ref{eqgapregime}), we define the positive parameter $$\hat\mu:=\mu - 2\gamma.$$ 
Our strategy is to track the equilibrium value $\psi_t^\star$ along the running average $\hat x_t$ of the iterates $x_t$ produced by Algorithm~\ref{alg:decsgd}, as defined in Algorithm~\ref{alg:decsgdavg} below. 
In a nutshell, the results of Section~\ref{sec:trckminimum} extend directly to tracking the equilibrium value $\psi_t^\star$, with $\mu$ replaced by $\hat \mu$ and $\Delta$ replaced by $\bar \Delta$ (defined in Assumption~\ref{assmp:funcgap0dec} below).

\begin{algorithm}[h!]
	\caption{Averaged Decision-Dependent PSG \hfill $\mathtt{D\text{-}\overline{PSG}}(x_{0},\mu,\gamma,\{\eta_t\},T)$} \label{alg:decsgdavg}
	
	{\bf Input}: initial $x_0 = \hat{x}_0$, strong convexity parameter $\mu$, gradient drift parameter $\gamma\in [0,\mu/2)$, and step sizes $\{\eta_t\}_{t=0}^{T-1}\subset (0,1/\bar{\mu})$, where $\bar{\mu} = \mu - \gamma$; set $\hat\mu = \mu-2\gamma$
		
	{\bf Step} $t=0,\ldots,T-1$: 
	\begin{equation*}
		\begin{aligned}
			&\text{Select~} g_t = \widetilde{\nabla}f_{t, x_t}(x_t)\\
			&\text{Set~} x_{t+1}=\prox_{\eta_t r_t}(x_t-\eta_t g_t)\\
			&\text{Set~} \hat{x}_{t+1}=\Big(1-\tfrac{\hat\mu\eta_t}{2-\mu\eta_t}\Big){\hat x}_t+\tfrac{\hat\mu\eta_t}{2-\mu\eta_t}x_{t+1}
		\end{aligned}
	\end{equation*}
	{\bf Return} $\hat{x}_T$
\end{algorithm}

We begin with bounding the expected value $\mathbb{E}[\psi_t(\hat x_t)-\psi_t^\star]$. This requires a weak form of uncorrelatedness between the gradient noise $z_i$ and the future equilibrium point $\bar{x}_t$, which we stipulate in the following analogue of Assumption~\ref{assmp:funcgap0}.

	\begin{assumption}[Bounded second moments]\label{assmp:funcgap0dec}
		\textup{The regularizers $r_t\equiv r$ are identical for all times $t$ and there exist constants $\bar{\Delta},\sigma>0 $ such that the following two conditions hold for all $0\leq i<t$:
			\begin{enumerate}[label=(\roman*)]
				\item {\bf (Drift)} The temporal gradient drift $\bar{G}_{i,t}$ satisfies $\mathbb{E}\,\bar{G}_{i,t}^2 \leq (\hat{\mu}\bar{\Delta}|i-t|)^2.$
				\item {\bf (Noise)} The gradient noise $z_i$ satisfies $\mathbb{E}\|z_i\|^2\leq\sigma^2$ and $\mathbb{E}\langle z_i, \bar{x}_t \rangle=0$.
			\end{enumerate}
		}
	\end{assumption}

Taking into account Lemma~\ref{lem:grad_vs_dist}, it is clear that Assumption~\ref{assmp:funcgap0dec} implies the earlier Assumption~\ref{assump:sec_moment_dec} with the same constants $\bar\Delta, \sigma$. Further, the condition on the drift holds trivially in the setting of Example~\ref{ex:perfpred} with $\bar\Delta = \theta\beta/\hat\mu$ provided $\mu>2\varepsilon\beta$. The following theorem presents an expected improvement guarantee for Algorithm~\ref{alg:decsgdavg}, thereby extending Theorem~\ref{thm:exp_gap}; see Corollary~\ref{cor2} for the precise statement and proof. 

\begin{theorem}[Expected function gap]\label{thm:exp_gap_dec}
	Let $\{\hat{x}_t\}$ be the iterates produced by Algorithm~\ref{alg:decsgdavg} with constant learning rate $\eta \leq 1/2L$, and suppose that Assumption~\ref{assmp:funcgap0dec} holds. Then the following bound holds for all $t\geq 0$:
	\begin{equation*}
		\mathbb{E}\big[\psi_t(\hat x_t)-\psi_t^\star\big] \lesssim  \underbrace{\big(1 - \tfrac{\hat\mu\eta}{2}\big)^t\big(\psi_0(x_0) - \psi_0^\star\big)}_\text{optimization} + \underbrace{\eta\sigma^2}_\text{noise} + \underbrace{\frac{\bar\Delta^2}{\hat\mu\eta^2}}_\text{drift}.
	\end{equation*}
\end{theorem}
 
With Theorem~\ref{thm:exp_gap_dec} in hand, we are led to define the following asymptotic tracking error of Algorithm~\ref{alg:decsgdavg} corresponding to $\mathbb{E}[\psi_t(\hat{x}_t) - \psi_t^\star]$, together with the corresponding optimal step size:
\begin{equation*}
	\widehat{\mathcal{G}}:=\min_{\eta\in (0,1/2L]} \left\{\eta\sigma^2 + \frac{\bar\Delta^2}{\hat\mu\eta^2}\right\} \quad\text{and}\quad \hat\eta_\star := \min\left\{\frac{1}{2L}, \left(\frac{2\bar\Delta^2}{\hat\mu\sigma^2}\right)^{1/3}\right\}\!.
\end{equation*}
A familiar dichotomy governed by the drift-to-noise ratio $\bar\Delta/\sigma$ arises. We again focus on the \emph{low drift-to-noise regime} $\bar\Delta/\sigma < \sqrt{\hat\mu/16L^3}$. The following theorem extends Theorem~\ref{thm:gap_schedul}; see Theorem~\ref{thm:time_to_track_gap_exp} for the formal statement.

\begin{theorem}[Time to track in expectation, informal]\label{thm:gap_schedul_dec}
	Suppose that Assumption~\ref{assmp:funcgap0dec} holds. Then there is a learning rate schedule $\{\eta_t\}$ such that Algorithm~\ref{alg:decsgdavg} produces a point $\hat{x}_t$ satisfying $$\mathbb{E}[\psi_t(\hat{x}_t) - \psi_t^\star] \lesssim \widehat{\mathcal{G}} \quad\text{in time}\quad t\lesssim \frac{L}{\hat\mu}\log\!\left(\frac{\psi_0(x_0) - \psi_0^\star}{\widehat{\mathcal{G}}}\right) + \frac{\sigma^2}{\hat\mu\widehat{\mathcal{G}}}.$$
\end{theorem}

Next, we obtain high-probability analogues of Theorems~\ref{thm:exp_gap_dec} and \ref{thm:gap_schedul_dec}. Naturally, such results should rely on light-tail assumptions on the temporal gradient drift $\bar{G}_{i,t}$ and the norm of the gradient noise $\|z_i\|$. We state the guarantees under an assumption of sub-Gaussian drift and noise (Assumption~\ref{assmp:funcgap2dec} below). In particular, we require that the gradient noise $z_i$ is mean-zero conditioned on the $\sigma$-algebra $$\mathcal{F}_{i,t}:=\sigma(\mathcal{F}_i,\bar{x}_t)$$ for all $0 \leq i < t$; the property $\mathbb{E}[z_i\,|\,\mathcal{F}_{i,t}]=0$ would follow from independence of the gradient noise $z_i$ and the future equilibrium point $\bar{x}_t$.

\begin{assumption}[Sub-Gaussian drift and noise]\label{assmp:funcgap2dec}
	\textup{The regularizers $r_t\equiv r$ are identical for all times $t$ and there exist constants $\bar\Delta,\sigma>0 $ such that the following two conditions hold for all $0\leq i<t$:
		\begin{enumerate}[label=(\roman*)]
			\item {\bf (Drift)} The drift $\bar{G}_{i,t}^2$ is sub-exponential with parameter $(\hat{\mu}\bar{\Delta}|i-t|)^2$:
			\begin{equation*}
				\mathbb{E}\big[{\exp}\big(\lambda \bar{G}_{i,t}^2\big)\big]\leq {\exp}\big(\lambda (\hat{\mu}\bar{\Delta}|i-t|)^2\big) \quad \text{for all} \quad  0\leq\lambda\leq (\hat{\mu}\bar{\Delta}|i-t|)^{-2}. 
			\end{equation*} 
			\item {\bf (Noise)} The noise $z_i$ is mean-zero norm sub-Gaussian conditioned on $\mathcal{F}_{i,t}$ with parameter $\sigma/2$, i.e.,
			$\mathbb{E}[z_i \,|\, \mathcal{F}_{i,t}]=0$ and  
			\begin{equation*}
				\mathbb{P}\big\{\|z_i\| \geq \tau \,|\,\mathcal{F}_{i,t}\big\}\leq 2\exp(-2\tau^2/\sigma^2) \quad\text{for all}\quad \tau>0.
			\end{equation*}
		\end{enumerate}
	}
\end{assumption}

Clearly the chain of implications
\begin{quote}\centering 
	Assumption~\ref{assmp:funcgap2dec}$\,\implies\,$Assumption~\ref{assmp:funcgap0dec}$\,\implies\,$Assumption~\ref{assump:sec_moment_dec}
\end{quote}
holds, and the condition on the drift in Assumption~\ref{assmp:funcgap2dec} holds trivially in the setting of Example~\ref{ex:perfpred} with $\bar\Delta = \theta\beta/\hat\mu$ provided $\mu>2\varepsilon\beta$. The following theorem shows that if Assumption~\ref{assmp:funcgap2dec} holds, then the expected bound on $\psi_t(\hat x_t)-\psi_t^\star$ derived in Theorem~\ref{thm:exp_gap_dec} holds with high probability, thereby extending Theorem~\ref{thm:hpbound2}.

\begin{theorem}[Function gap with high probability]\label{thm:hpbound2dec}
	Let $\{\hat{x}_t\}$ be the iterates produced by Algorithm \ref{alg:decsgdavg} with constant learning rate $\eta \leq 1/2L$, and suppose that Assumption~\ref{assmp:funcgap2dec} holds. Then there is an absolute constant $c>0$ such that for any specified $t\in\mathbb{N}$ and $\delta\in(0,1)$, the following estimate holds with probability at least $1-\delta$:
	\begin{align}\label{ineq:hpbound2dec}
		\psi_t(\hat x_t)-\psi_t^\star \leq c\left(\big(1 - 	\tfrac{\hat\mu\eta}{2}\big)^t\big(\psi_0(x_0) - \psi_0^\star\big) + \eta\sigma^2 + \frac{\bar\Delta^2}{\hat\mu\eta^2}\right)\log\!\left(\frac{e}{\delta}\right)\!.
	\end{align}
\end{theorem}

With Theorem~\ref{thm:hpbound2dec} in hand, we may implement a step decay schedule as before to obtain the following efficiency estimate, thereby extending Theorem~\ref{thm:gap_schedulhp2}; see Section~\ref{minimalvalproofsprob} for the precise statements (Theorems~\ref{thm:hpbound2precise} and \ref{thm:time_to_track_gap_hp2}) and proofs.

\begin{theorem}[Time to track with high probability, informal]\label{thm:gap_schedulhp2dec}
	Suppose that Assumption~\ref{assmp:funcgap2dec} holds and that we are in the low drift-to-noise regime $\bar{\Delta}/\sigma<\sqrt{\hat{\mu}/16L^3}$. Fix $\delta\in(0,1)$. Then there is a learning rate schedule $\{\eta_t\}$ such that Algorithm~\ref{alg:decsgdavg} produces a point $\hat{x}_t$ satisfying 
	\begin{equation*}
		\psi_t(\hat{x}_t) - \psi_t^\star\lesssim \widehat{\mathcal{G}}\log\!\left(\frac{e}{\delta}\right)
	\end{equation*}
	with probability at least $1-K\delta$ in time 
	\begin{equation*}
		t\lesssim \frac{L}{\hat\mu}\log\!\left(\frac{\psi_0(x_0) - \psi_0^\star}{\widehat{\mathcal{G}}}\right) + \frac{\sigma^2}{\hat\mu\widehat{\mathcal{G}}}\log\!\left(\log\!\left(\frac{e}{\delta}\right)\!\right)\!, \quad \text{where~~}K\lesssim \log_2\!\left(\frac{1}{L}\cdot\left(\frac{\sigma^2\hat\mu}{\bar\Delta^2}\right)^{1/3}\right)\!.
	\end{equation*}
\end{theorem}
 \section{Proofs of Main Results}
\label{mainproofs}
\paragraph{Roadmap.} In this section, we derive the results of the preceding sections under the unified framework presented in Section~\ref{sec:decframework}; we impose the assumptions and notation of Section~\ref{sec:decframework} henceforth. Sections~\ref{minimizerproofsexp} and \ref{minimizerproofsprob} handle distance tracking in expectation and with high probability, respectively; this corresponds to the results presented in Section~\ref{sec:eqtrck} (entailing those of Sections~\ref{sec:minimizerexp} and \ref{sec:hp1}). Then Sections~\ref{minimalvalproofsexp} and \ref{minimalvalproofsprob} handle function gap tracking in expectation and with high probability, respectively; this corresponds to the results presented in Section~\ref{sec:eqvaltrck} (entailing those of Sections~\ref{sec:minvalexp} and \ref{sec:minvalprob}).

\subsection{Tracking the Equilibrium Point: Bounds in Expectation}
\label{minimizerproofsexp}
The proof of Theorem~\ref{thm:exp_dist_dec} follows a familiar pattern in stochastic optimization. We begin by recalling Lemma~\ref{lem:onestep0}, which gives a standard one-step improvement guarantee for the proximal stochastic gradient method on the fixed problem $\min \varphi_t$.

\begin{lemma}[One-step improvement]\label{lem:onestep1}
	For all $x\in \mathbb{R}^d$, the iterates $\{x_t\}$ produced by Algorithm~\ref{alg:decsgd} with $\eta_t<1/L$ satisfy the bound:
	\begin{equation*}
		2\eta_t(\varphi_t(x_{t+1})-\varphi_t(x)) \leq (1-\mu\eta_t)\|x_t-x\|^2 - \|x_{t+1}-x\|^2 + 2\eta_t\langle z_t,x_t-x \rangle + \tfrac{\eta^2_t}{1-L\eta_t}\|z_t\|^2.
	\end{equation*}
\end{lemma}

\begin{proof}
Since $f_t:=f_{t,x_t}$ is $L$-smooth, we have
\begin{equation*}
	\begin{split}
		\varphi_t(x_{t+1})&=f_t(x_{t+1})+r_t(x_{t+1})\\
		&\leq f_t(x_t)+\langle\nabla f_t(x_t),x_{t+1}-x_t\rangle+\tfrac{L}{2}\|x_{t+1}-x_t\|^2+r_t(x_{t+1})\\
		&=f_t(x_t)+r_t(x_{t+1})+\langle g_t,x_{t+1}-x_t\rangle+\tfrac{L}{2}\|x_{t+1}-x_t\|^2+\langle z_t,x_{t+1}-x_t\rangle.
	\end{split}
\end{equation*}	
Next, given any $\delta_t>0$, Cauchy-Schwarz and Young's inequality yield
\begin{equation*}
	\langle z_t,x_{t+1}-x_t\rangle\leq\tfrac{\delta_t}{2}\|z_t\|^2+\tfrac{1}{2\delta_t}\|x_{t+1}-x_t\|^2.
\end{equation*}
Therefore, given any $x\in\mathbb{R}^d$, we have
\begin{equation*}
	\begin{split}
		\varphi_t(x_{t+1})&\leq f_t(x_t)+r_t(x_{t+1})+\langle g_t,x_{t+1}-x_t\rangle+\tfrac{\delta_t^{-1}+L}{2}\|x_{t+1}-x_t\|^2+\tfrac{\delta_t}{2}\|z_t\|^2 \\
		&= f_t(x_t)+r_t(x_{t+1})+\langle g_t,x_{t+1}-x_t\rangle+\tfrac{1}{2\eta_t}\|x_{t+1}-x_t\|^2\\&\qquad +\tfrac{\delta_t^{-1}+L-\eta_t^{-1}}{2}\|x_{t+1}-x_t\|^2+\tfrac{\delta_t}{2}\|z_t\|^2 \\
		&\leq f_t(x_t)+r_t(x)+\langle g_t,x-x_t\rangle+\tfrac{1}{2\eta_t}\|x-x_t\|^2-\tfrac{1}{2\eta_t}\|x-x_{t+1}\|^2\\
		&\qquad+\tfrac{\delta_t^{-1}+L-\eta_t^{-1}}{2}\|x_{t+1}-x_t\|^2+\tfrac{\delta_t}{2}\|z_t\|^2,	
	\end{split}
\end{equation*}
where the last inequality holds because $x_{t+1}=\prox_{\eta_t r_t}(x_t-\eta_t g_t)$ is the minimizer of the $\eta_t^{-1}$-strongly convex function $r_t+\langle g_t,\cdot-x_t\rangle+\frac{1}{2\eta_t}\|\cdot-x_t\|^2$. Now we estimate 
\begin{equation*}
	\begin{split}
		f_t(x_t)+r_t(x)+\langle g_t,x-x_t\rangle &= f_t(x_t)+\langle \nabla f_t (x_t), x - x_t\rangle + r_t(x) + \langle z_t, x_t - x\rangle \\
		&\leq f_t(x)-\tfrac{\mu}{2}\|x-x_t\|^2 + r_t(x)+\langle z_t, x_t - x\rangle \\
		&=\varphi_t(x)-\tfrac{\mu}{2}\|x-x_t\|^2+\langle z_t, x_t - x\rangle
	\end{split}
\end{equation*}
using the $\mu$-strong convexity of $f_t$. Thus,
\begin{equation*}
	\begin{split}
		\varphi_t(x_{t+1})\leq\varphi_t(x)-\tfrac{\mu}{2}\|x-x_t\|^2+\langle z_t, x_t - x\rangle + \tfrac{1}{2\eta_t}\|x-x_t\|^2-\tfrac{1}{2\eta_t}\|x-x_{t+1}\|^2  \\
		+\tfrac{\delta_t^{-1}+L-\eta_t^{-1}}{2}\|x_{t+1}-x_t\|^2+\tfrac{\delta_t}{2}\|z_t\|^2. 
	\end{split}
\end{equation*}
Finally, taking
$
\delta_t=\eta_t/(1-L\eta_t)
$
and rearranging (note that $\varphi_t(x_{t+1})$ is finite) yields
\begin{equation*}
	2\eta_t(\varphi_t(x_{t+1})-\varphi_t(x)) \leq (1-\mu\eta_t)\|x_t-x\|^2 - \|x_{t+1}-x\|^2 + 2\eta_t\langle z_t,x_t-x \rangle + \tfrac{\eta^2_t}{1-L\eta_t}\|z_t\|^2,
\end{equation*}
as claimed.
\end{proof}

It is critically important that the one-step improvement estimate in Lemma~\ref{lem:onestep1} holds with respect to any reference point $x$. In particular, as we already showed in Section~\ref{sec:eqtrck}, taking $x=\bar{x}_t$ and applying Lemma~\ref{lem:funcgap} yields Lemma~\ref{lem:onestep0dec}:

\begin{lemma}[Equilibrium one-step improvement]\label{lem:onestep0dec2}
	The iterates $\{x_t\}$ produced by Algorithm~\ref{alg:decsgd} with $\eta_t<1/L$ satisfy the bound:
	\begin{equation*}
		\begin{split}
			2\eta_t(\psi_t(x_{t+1})-\psi_t^\star) \leq (1-\bar\mu\eta_t)\|x_t-\bar{x}_t\|^2 - (1 &- \gamma\eta_t)\|x_{t+1}-\bar{x}_t\|^2 \\
			&+ 2\eta_t\langle z_t, x_t - \bar{x}_t \rangle + \tfrac{\eta_t^2}{1-L\eta_t}\|z_t\|^2.
		\end{split}
	\end{equation*}
\end{lemma}

With Lemma~\ref{lem:onestep0dec2} in  hand, we obtain the following recursion on $\|x_t - \bar{x}_t\|^2$.

\begin{lemma}[Distance recursion]\label{lem:distrecur}
	The iterates $\{x_t\}$ produced by Algorithm~\ref{alg:decsgd} with step size $\eta_t<1/L$ satisfy the bound:
	\begin{equation*}
		\|x_{t+1}-\bar{x}_{t+1}\|^2 \leq (1-\bar{\mu}\eta_t)\|x_t-\bar{x}_t\|^2  + 2\eta_t\langle z_t,x_t-\bar{x}_t \rangle + \tfrac{\eta^2_t}{1-L\eta_t}\|z_t\|^2 +\Big(1+\tfrac{1}{\bar{\mu}\eta_t}\Big)\bar{\Delta}^2_t.
	\end{equation*}
\end{lemma}
\begin{proof} 
	First, note
	\begin{align*}
		\|x_{t+1}-\bar{x}_{t+1}\|^2  &= \|x_{t+1} - \bar{x}_t\|^2 + \|\bar{x}_t - \bar{x}_{t+1}\|^2 + 2 \langle x_{t+1}-\bar{x}_t, \bar{x}_t - \bar{x}_{t+1}\rangle \\
		&\leq (1+\bar{\mu}\eta_t)\|x_{t+1}-\bar{x}_t\|^2 + \Big(1+\tfrac{1}{\bar{\mu}\eta_t}\Big)\|\bar{x}_t - \bar{x}_{t+1}\|^2
	\end{align*}
	by Cauchy-Schwarz and Young's inequality. Further, the $\mu$-strong convexity of $\psi_t$ implies 
	$\frac{\mu}{2}\|x_{t+1}-\bar{x}_t\|^2 \leq \psi_t(x_{t+1})-\psi_t^\star$, which together with Lemma~\ref{lem:onestep0dec2} implies
	\begin{equation*}
		(1+\bar{\mu}\eta_t)\|x_{t+1}-\bar{x}_t\|^2 \leq (1-\bar{\mu}\eta_t)\|x_t-\bar{x}_t\|^2  + 2\eta_t\langle z_t,x_t-\bar{x}_t \rangle + \tfrac{\eta^2_t}{1-L\eta_t}\|z_t\|^2.
	\end{equation*}
	The result follows.
\end{proof}

Applying Lemma \ref{lem:distrecur} recursively furnishes a bound on $\|x_t - \bar{x}_t\|^2$. When the step size is constant, the next proposition follows immediately.

\begin{proposition}[Last-iterate progress]\label{distbound}
	The iterates $\{x_t\}$ produced by Algorithm~\ref{alg:decsgd} with constant step size $\eta<1/L$ satisfy the bound: 
	\begin{equation*}
		\begin{split}
			\|x_t - \bar{x}_t\|^2 \leq (1&-\bar{\mu}\eta)^t\|x_0-\bar{x}_0\|^2+ 2\eta\sum_{i=0}^{t-1}\langle z_i,x_i - \bar{x}_i\rangle(1-\bar{\mu}\eta)^{t-1-i} \\
			&+
			\tfrac{\eta^2}{1-L\eta}\sum_{i=0}^{t-1}\|z_i\|^{2}(1-\bar{\mu}\eta)^{t-1-i} + \left(1+\tfrac{1}{\bar{\mu}\eta}\right)\sum_{i=0}^{t-1}\bar{\Delta}_i^{2}(1-\bar{\mu}\eta)^{t-1-i}.	
		\end{split}	
	\end{equation*}
\end{proposition}

By taking expectations in Proposition~\ref{distbound}, we obtain the following precise version of Theorem~\ref{thm:exp_dist_dec}.

\begin{corollary}[Expected distance]\label{cor1}
	Suppose that Assumption~\ref{assump:sec_moment_dec} holds. Then the iterates $\{x_t\}$ generated by Algorithm~\ref{alg:decsgd} with constant learning rate $\eta\leq 1/2L$ satisfy the bound:  
	\begin{equation*}
		\mathbb{E}\|x_t-\bar{x}_t\|^2 \leq (1-\bar{\mu}\eta)^t\|x_0-\bar{x}_0\|^2 + 2\bigg(\frac{\eta\sigma^2}{\bar\mu} + \left(\frac{\bar\Delta}{\bar{\mu}\eta}\right)^2\bigg).
	\end{equation*}
\end{corollary}

With Corollary~\ref{cor1} in hand, we can now prove an expected efficiency estimate for the online proximal stochastic gradient method using a step decay schedule, wherein the algorithm is implemented in epochs with the new learning rate chosen to be the midpoint between the current learning rate and the asymptotically optimal learning rate $\bar{\eta}_{\star}$. The following theorem provides a formal version of Theorem~\ref{thm:dist_schedul_dist} (note that in the high drift-to-noise regime $\bar{\Delta}/\sigma \geq \sqrt{\bar{\mu}/16L^3}$, Theorem~\ref{thm:dist_schedul_dist} holds trivially with the constant learning rate $\bar{\eta}_\star=1/2L$). The argument is close in spirit to the justifications of the restart schemes used by \citet{ghadimi2013optimal}.
\begin{theorem}[Time to track in expectation]\label{thm:time_to_track}
	Suppose that Assumption \ref{assump:sec_moment_dec} holds and that we are in the low drift-to-noise regime $\bar{\Delta}/\sigma<\sqrt{\bar\mu/16L^3}$. Set $\bar{\eta}_\star = (2\bar{\Delta}^2/\bar{\mu}\sigma^2)^{1/3}$ and $\bar{\mathcal{E}}=(\bar{\Delta}\sigma^2/\bar{\mu}^2)^{2/3}$. Suppose moreover that we have available a positive upper bound on the initial squared distance $D\geq \|x_0-\bar{x}_0\|^2$.
	Consider running Algorithm~\ref{alg:decsgd} in $k=0,\ldots, K-1$ epochs, namely, set $X_0=x_0$ and iterate the process
	$$X_{k+1}=\mathtt{D\text{-}PSG}(X_{k},\eta_k,T_k)\quad \text{for}\quad k=0,\ldots, K-1,$$ where the number of epochs is $$K=1+\left\lceil\log_2\!\left(\frac{1}{L}\cdot\left(\frac{\sigma^2\bar\mu}{\bar{\Delta}^2}\right)^{1/3}\right)\right\rceil$$ 
	and we set\,\footnote{We use here the notation $a^+ = a \vee 0 = \max\{a, 0\}$ to denote the positive part of a real number $a$; note that for small $D$, the logarithms $\log(\bar\mu L D/\sigma^2)$ and $\log(D/\bar{\mathcal{E}})$ may be negative.}
	$$\eta_0=\frac{1}{2L},\quad T_0=\left\lceil{\frac{2L}{\bar\mu}\log\!\left(\frac{\bar\mu LD}{\sigma^2}\right)^+}\right\rceil\quad \text{and}\quad \eta_k=\frac{\eta_{k-1}+ \bar{\eta}_{\star}}{2},\quad T_k=\left\lceil{\frac{\log(4)}{\bar\mu\eta_k}}\right\rceil\quad \forall k\geq 1.$$ Then the time horizon $T=T_0+\cdots+T_{K-1}$ satisfies
	$$T\lesssim \frac{L}{\bar\mu}\log\!\left(\frac{\bar\mu L D}{\sigma^2}\right)^+ + \frac{\sigma^2}{\bar{\mu}^2\bar{\mathcal{E}}} \leq \frac{L}{\bar\mu}\log\!\left(\frac{D}{\bar{\mathcal{E}}}\right)^+ + \frac{\sigma^2}{\bar{\mu}^2\bar{\mathcal{E}}}$$
	and the corresponding tracking error satisfies
	$\mathbb{E}\|X_K-\bar{X}_K\|^2\lesssim \bar{\mathcal{E}}$,
	where $\bar{X}_K$ denotes the minimizer of $\psi_T$.
\end{theorem}
\begin{proof}
	For each index $k$, let $t_k:=T_0+\cdots+T_{k-1}$ (with $t_0 := 0$), $\bar{X}_k$ be the minimizer of the corresponding equilibrium function $\psi_{t_k}$, and $$\bar{E}_k:=\frac{2}{\bar\mu}\left(\eta_k\sigma^2 + \frac{\bar{\Delta}^2}{\bar{\mu}\bar{\eta}^2_{\star}}\right)\!.$$
	Then taking into account $\eta_k \geq \bar{\eta}_{\star}$, Corollary~\ref{cor1} directly implies
	\begin{align*}
		\mathbb{E}\|X_{k+1} - \bar{X}_{k+1}\|^2 &\leq (1-\bar{\mu}\eta_k)^{T_k}\mathbb{E}\|X_k - \bar{X}_k\|^2 + \frac{2}{\bar\mu}\!\left(\eta_k\sigma^2 + \frac{\bar{\Delta}^2}{\bar\mu\eta^2_k}\right)\\
		&\leq e^{-\bar\mu\eta_kT_k}\mathbb{E}\|X_k - \bar{X}_k\|^2 + \bar{E}_k.
	\end{align*}
	
	We will verify by induction that the estimate $\mathbb{E}\|X_{k} - \bar{X}_{k}\|^2\leq 2\bar{E}_{k-1}$
	holds for all indices $k\geq1$. To see the base case, observe
	$$\mathbb{E}\|X_1 - \bar{X}_1\|^2 \leq e^{-\bar\mu\eta_0 T_0}\|X_0 - \bar{X}_0\|^2 + \bar{E}_0 \leq 2\bar{E}_0.$$
	Now assume that the claim holds for some index $k\geq1$. We then conclude
	\begin{align*}
		\mathbb{E}\|X_{k+1} - \bar{X}_{k+1}\|^2&\leq e^{-\bar\mu\eta_k T_k}\mathbb{E}\|X_k - \bar{X}_k\|^2 + \bar{E}_k\\
		&\leq \frac{1}{4}\mathbb{E}\|X_k - \bar{X}_k\|^2 + \bar{E}_k\\
		&\leq \frac{\bar{E}_k}{2\bar{E}_{k-1}}\mathbb{E}\|X_k - \bar{X}_k\|^2 + \bar{E}_k \leq 2\bar{E}_k,
	\end{align*}
	thereby completing the induction. Hence $\mathbb{E}\|X_K - \bar{X}_K\|^2\leq  2\bar{E}_{K-1}$.
	
	Next, observe
	$$\bar{E}_{K-1}-\sqrt[3]{54}\left(\frac{\bar\Delta\sigma^2}{\bar{\mu}^2}\right)^{2/3}=\frac{2\sigma^2}{\bar\mu}(\eta_{K-1} - \bar{\eta}_\star) = \frac{2\sigma^2}{\bar\mu}\cdot \frac{\eta_0 - \bar{\eta}_\star}{2^{K-1}} \leq \left(\frac{\bar\Delta\sigma^2}{\bar{\mu}^2}\right)^{2/3} = \bar{\mathcal{E}},$$ so $$\mathbb{E}\|X_K - \bar{X}_K\|^2 \leq  2(1+\sqrt[3]{54})\left(\frac{\bar\Delta\sigma^2}{\bar{\mu}^2}\right)^{2/3} \asymp \bar{\mathcal{E}}.$$ 
	Finally, note 
	$$T \lesssim \frac{L}{\bar\mu}\log\!\left(\frac{\bar\mu LD}{\sigma^2}\right)^+ + \frac{1}{\bar\mu}\sum_{k=1}^{K-1}\frac{1}{\eta_k}$$ and $$\sum_{k=1}^{K-1}\frac{1}{\eta_k}\leq 2L\sum_{k=1}^{K-1} 2^k \leq 2L\cdot 2^{K} = 8L\cdot 2^{K-2} \leq 8 \left(\frac{\sigma^2\bar\mu}{\bar{\Delta}^2}\right)^{1/3} = \frac{8\sigma^2}{\bar\mu} \cdot \left(\frac{\bar\Delta\sigma^2}{\bar{\mu}^2}\right)^{-2/3} \asymp \frac{\sigma^2}{\bar\mu\bar{\mathcal{E}}}.$$  
	This completes the proof.
\end{proof}

\subsection{Tracking the Equilibrium Point: High-Probability Guarantees}
\label{minimizerproofsprob}
The proof of Theorem~\ref{HPtrackdec} is based on recursively controlling the moment generating function of $\|x_t - \bar{x}_t\|^2$.
Namely, Lemma \ref{lem:distrecur} in the regime $\eta_t\leq1/2L$ directly yields
\begin{equation}\label{onestep}
	\|x_{t+1}-\bar{x}_{t+1}\|^2 \leq (1-\bar\mu\eta_t)\|x_t - \bar{x}_t\|^2  + 2\eta_t\langle z_t, v_t\rangle\|x_t - \bar{x}_t\| + 2\eta_t^2\|z_t\|^2 +\tfrac{2}{\bar\mu\eta_t}\bar{\Delta}_t^2,
\end{equation}
where we set 
\begin{equation*}
	v_t := \begin{cases} 
		\frac{x_t - \bar{x}_t}{\|x_t - \bar{x}_t\|} & \text{if~} x_t \neq \bar{x}_t \\
		0 & \text{otherwise}.
	\end{cases}
\end{equation*}
The goal is now to control the moment generating function $\mathbb{E}\big[e^{\lambda\|x_t - \bar{x}_t\|^2}\big]$ through the recursive inequality \eqref{onestep}. The basic probabilistic tool to achieve this in similar settings under bounded noise assumptions was developed by \citet{pmlr-v99-harvey19a}; the following proposition is a slight generalization of Claim D.1 of \citet{pmlr-v99-harvey19a} to the light-tail setting we require. 

\begin{proposition}[Recursive control on MGF]\label{proprec}
	Consider scalar stochastic processes $(V_t)$, $(D_t)$, and $(X_t)$ on a probability space with filtration $(\mathcal{H}_t)$ such that $V_t$ is nonnegative and $\mathcal{H}_t$-measurable and the inequality 
	\begin{equation*}
		V_{t+1}\leq \alpha_tV_t+D_t\sqrt{V_t}+X_t+\kappa_t 
	\end{equation*} 
	holds for some deterministic constants $\alpha_t\in(-\infty,1]$ and $\kappa_t\in\R$. Suppose that the moment generating functions of $D_t$ and $X_t$ conditioned on $\mathcal{H}_t$ satisfy the following inequalities for some deterministic constants $\sigma_t, \nu_t>0$:
	\begin{itemize}[leftmargin=0.5cm]
		
		\item $\mathbb{E}[\exp(\lambda D_t)\,|\,\mathcal{H}_t]\leq \exp(\lambda^2\sigma_t^2/2)$ for all $\lambda\geq 0$ (e.g., $D_t$ is mean-zero sub-Gaussian conditioned on $\mathcal{H}_t$ with parameter $\sigma_t$).
		\item $\mathbb{E}[\exp(\lambda X_t)\,|\,\mathcal{H}_t]\leq \exp(\lambda \nu_t)$ for all $0\leq\lambda\leq 1/\nu_t$ (e.g., $X_t$ is nonnegative and sub-exponential conditioned on $\mathcal{H}_t$ with parameter $\nu_t$).	
	\end{itemize}
	Then the inequality
	\begin{equation*}
		\mathbb{E}[\exp(\lambda V_{t+1})]\leq {\exp}\big(\lambda(\nu_t+\kappa_t)\big)\,\mathbb{E}\bigg[{\exp}\bigg(\lambda\bigg(\frac{1+\alpha_t}{2}\bigg)V_t\bigg)\bigg]
	\end{equation*}
	holds for all $
	0\leq\lambda\leq\min\!\big\{\frac{1-\alpha_t}{2\sigma_t^2},\frac{1}{2\nu_t}\big\}$.
\end{proposition}
\begin{proof}
For any index $t$ and any scalar $\lambda\geq0$, the tower rule implies
\begin{equation*}
	\begin{split}
		\mathbb{E}[\exp(\lambda V_{t+1})]&\leq \mathbb{E}\big[{\exp}\big(\lambda\big(\alpha_tV_t+D_t\sqrt{V_t}+X_t+\kappa_t\big)\big)\big] \\
		&= \exp(\lambda\kappa_t)\,\mathbb{E}\Big[{\exp}(\lambda\alpha_tV_t)\,\mathbb{E}\big[{\exp}\big(\lambda D_t\sqrt{V_t}\big)\exp(\lambda X_t)\,|\,\mathcal{H}_t\big]\Big]. 
	\end{split}
\end{equation*}
H\"{o}lder's inequality in turn yields
\begin{equation*}
	\begin{split}
		\mathbb{E}\big[{\exp}\big(\lambda D_t\sqrt{V_t}\big)\exp(\lambda X_t)\,|\,\mathcal{H}_t\big] &\leq \sqrt{\mathbb{E}\big[{\exp}\big(2\lambda\sqrt{V_t}D_t\big)\,|\,\mathcal{H}_t\big]\cdot\mathbb{E}\big[{\exp}(2\lambda X_t)\,|\,\mathcal{H}_t\big]} \\
		&\leq \sqrt{\exp(2\lambda^2V_t\sigma_t^2)\exp(2\lambda \nu_t)} \\
		&= \exp(\lambda^2\sigma_t^2V_t)\exp(\lambda \nu_t)
	\end{split}
\end{equation*}
provided
$
0\leq\lambda\leq\frac{1}{2\nu_t}.$
Thus, if $
0\leq\lambda\leq\min\!\big\{\frac{1-\alpha_t}{2\sigma_t^2},\frac{1}{2\nu_t}\big\}
$, then the following estimate holds:
\begin{equation*}
	\begin{split}
		\mathbb{E}\big[{\exp}(\lambda V_{t+1})\big]&\leq \exp(\lambda\kappa_t)\,\mathbb{E}\big[{\exp}(\lambda\alpha_tV_t)\exp(\lambda^2\sigma_t^2V_t)\exp(\lambda \nu_t)\big] \\
		&=  {\exp}\big(\lambda(\nu_t+\kappa_t)\big)\,\mathbb{E}\big[{\exp}(\lambda(\alpha_t+\lambda\sigma_t^2)V_t)\big] \\
		&\leq  {\exp}\big(\lambda(\nu_t+\kappa_t)\big)\,\mathbb{E}\bigg[{\exp}\bigg(\lambda\bigg(\frac{1+\alpha_t}{2}\bigg)V_t\bigg)\bigg].
	\end{split}
\end{equation*}
The proof is complete.
\end{proof}

We may now use Proposition \ref{proprec} to derive the following precise version of Theorem \ref{HPtrackdec}.

\begin{theorem}[High-probability distance tracking]\label{HPtrack1}
	Let $\{x_t\}$ be the iterates produced by Algorithm~\ref{alg:decsgd} with constant learning rate $\eta\leq1/2L$, and suppose that Assumption \ref{assumpt_light_dec} holds. Then there exists an absolute constant\,\footnote{Explicitly, one can take any $c\geq1$ such that $\|z_t\|^2$ is sub-exponential conditioned on $\mathcal{F}_t$ with parameter $c\sigma^2$ and $z_t$ is mean-zero sub-Gaussian conditioned on $\mathcal{F}_t$ with parameter $c\sigma$ for all $t$.} $c>0$ such that for any specified $t\in \mathbb{N}$ and $\delta\in(0,1)$, the following estimate holds with probability at least $1-\delta$:
	$$\|x_t - \bar{x}_t\|^2 \leq \big(1-\tfrac{\bar\mu\eta}{2}\big)^t\|x_0 - \bar{x}_0\|^2 + \left(\frac{8\eta(c\sigma)^2}{\bar\mu} + 4\bigg(\frac{\bar\Delta}{\bar\mu\eta}\bigg)^2\right)\log\!\left(\frac{e}{\delta}\right)\!.$$
\end{theorem}
\begin{proof} Note first that under Assumption \ref{assumpt_light_dec}, there exists an absolute constant $c\geq1$ such that $\|z_t\|^2$ is sub-exponential conditioned on $\mathcal{F}_t$ with parameter $c\sigma^2$ and $z_t$ is mean-zero sub-Gaussian conditioned on $\mathcal{F}_t$ with parameter $c\sigma$ for all $t$ \citep[see][Lemma~3]{jin2019short}. Therefore $\langle z_t, v_t \rangle$ is mean-zero sub-Gaussian conditioned on $\mathcal{F}_t$ with parameter $c\sigma$, while $\bar\Delta_t^2$ is sub-exponential conditioned on $\mathcal{F}_t$ with parameter $\bar{\Delta}^2$ by Assumption \ref{assumpt_light_dec}. Thus, in light of inequality (\ref{onestep}), we may apply Proposition \ref{proprec} with $\mathcal{H}_t = \mathcal{F}_t$, $V_t = \|x_t - \bar{x}_t\|^2$, $D_t = 2\eta_t\langle z_t, v_t \rangle$, $X_t=2\eta_t^2\|z_t\|^2+2\bar{\Delta}_t^2/\bar\mu\eta_t$, $\alpha_t=1-\bar\mu\eta_t$, $\kappa_t=0$, $\sigma_t = 2\eta_tc\sigma$, and $\nu_t = 2\eta_t^2c\sigma^2+2\bar{\Delta}^2/\bar\mu\eta_t$, yielding the estimate
	\begin{equation}\label{MGFrec}
		\mathbb{E}\big[{\exp}\big(\lambda \|x_{t+1}-\bar{x}_{t+1}\|^2\big)\big] \leq \exp\!\left(\!\lambda\!\left(2\eta_t^2c\sigma^2+\frac{2\bar{\Delta}^2}{\bar\mu\eta_t}\right)\!\right) \mathbb{E}\big[{\exp}\big(\lambda\big(1-\tfrac{\bar\mu\eta_t}{2}\big)\|x_t - \bar{x}_t\|^2\big)\big]
	\end{equation}
	for all
	\begin{equation*}
		0\leq\lambda\leq\min{\left\{\frac{\bar\mu}{8\eta_t(c\sigma)^2},\frac{1}{4\eta_t^2c\sigma^2+4\bar{\Delta}^2/\bar\mu\eta_t}\right\}}. 
	\end{equation*}

	Taking into account $\eta_t\equiv\eta$ and iterating the recursion (\ref{MGFrec}), we deduce
	\begin{equation*}
		\begin{split}
			\mathbb{E}\big[{\exp}\big(\lambda\|x_t - \bar{x}_t\|^2\big)\big]&\leq \exp\!\left(\lambda\big(1-\tfrac{\bar\mu\eta}{2}\big)^t\|x_0 - \bar{x}_0\|^2+\lambda\!\left(2\eta^2c\sigma^2+\frac{2\bar{\Delta}^2}{\bar\mu\eta}\right)\!\sum_{i=0}^{t-1}\big(1-\tfrac{\bar\mu\eta}{2}\big)^i\right) \\
			&\leq \exp\!\left(\lambda\!\left(\big(1-\tfrac{\bar\mu\eta}{2}\big)^t\|x_0 - \bar{x}_0\|^2 +\frac{4\eta c\sigma^2}{\bar\mu} + 4\bigg(\frac{\bar\Delta}{\bar\mu\eta}\bigg)^2\right)\right) \\
		\end{split}
	\end{equation*}
	for all
	\begin{equation*}
		0\leq\lambda\leq\min{\left\{\frac{\bar\mu}{8\eta(c\sigma)^2},\frac{1}{4\eta^2c\sigma^2+4\bar{\Delta}^2/\bar\mu\eta}\right\}}. 
	\end{equation*}
	Moreover, setting 
	$$ \nu := \frac{8\eta(c\sigma)^2}{\bar\mu}+4\bigg(\frac{\bar\Delta}{\bar\mu\eta}\bigg)^2 $$
	and taking into account $c\geq1$ and $\bar\mu\eta\leq1$, we have
	\begin{equation*}
		\frac{4\eta c\sigma^2}{\bar\mu} + 4\bigg(\frac{\bar\Delta}{\bar\mu\eta}\bigg)^2 \leq \nu
	\end{equation*}
	and
	\begin{equation*}
		\frac{1}{\nu} = \frac{\bar\mu}{8\eta(c\sigma)^2+4\bar{\Delta}^2/\bar\mu\eta^2} \leq \min{\left\{\frac{\bar\mu}{8\eta(c\sigma)^2},\frac{1}{4\eta^2c\sigma^2+4\bar{\Delta}^2/\bar\mu\eta}\right\}}.
	\end{equation*}
	Hence
	\begin{equation*}
		\mathbb{E}\Big[{\exp}\!\left(\lambda\!\left(\|x_t - \bar{x}_t\|^2 - \big(1-\tfrac{\bar\mu\eta}{2}\big)^t\|x_0 - \bar{x}_0\|^2\right)\!\right)\!\Big] \leq \exp(\lambda \nu) \quad \text{for all} \quad 0\leq\lambda\leq 1/\nu.
	\end{equation*}
	Taking $\lambda = 1/\nu$ and applying Markov's inequality completes the proof. 
\end{proof}

With Theorem \ref{HPtrack1} in hand, we can now prove a high-probability efficiency estimate for Algorithm~\ref{alg:decsgd} using a step decay schedule. The following theorem provides a formal version of Theorem \ref{thm:dist_schedul2_dec}. The argument follows the same reasoning as in the proof of Theorem~\ref{thm:time_to_track}, with Theorem~\ref{HPtrack1} playing the role of Corollary~\ref{cor1}. The proof appears in Appendix~\ref{missingproofs} (see Section~\ref{pf:time_to_track2}).

\begin{theorem}[Time to track with high probability]\label{thm:time_to_track2}
	Suppose that Assumption \ref{assumpt_light_dec} holds and that we are in the low drift-to-noise regime $\bar{\Delta}/\sigma<\sqrt{\bar{\mu}/16L^3}$. Set $\bar{\eta}_\star = (2\bar{\Delta}^2/\bar{\mu}\sigma^2)^{1/3}$ and $\bar{\mathcal{E}}=(\bar\Delta\sigma^2/\bar{\mu}^2)^{2/3}$. Suppose moreover that we have available an upper bound on the initial squared distance $D\geq \|x_0 - \bar{x}_0\|^2$. Consider running Algorithm~\ref{alg:decsgd} in $k=0,\ldots, K-1$ epochs, namely, set $X_0=x_0$ and iterate the process
	$$X_{k+1}=\mathtt{D\text{-}PSG}(X_{k},\eta_k,T_k)\quad \text{for}\quad k=0,\ldots, K-1,$$ where the number of epochs is $$K=1+\left\lceil{\log_2\!\left(\frac{1}{L}\cdot\left(\frac{\sigma^2\bar\mu}{\bar{\Delta}^2}\right)^{1/3}\right)}\right\rceil$$ and we set
	$$\eta_0=\frac{1}{2L},\quad T_0=\left\lceil{\frac{4L}{\bar\mu}\log\!\left(\frac{\bar\mu LD}{\sigma^2}\right)^+}\right\rceil\quad \text{and}\quad \eta_k=\frac{\eta_{k-1}+ \bar{\eta}_{\star}}{2},\quad T_k=\left\lceil{\frac{2\log(12)}{\bar\mu\eta_k}}\right\rceil\quad \forall k\geq 1.$$
	Then the time horizon $T=T_0+\cdots+T_{K-1}$ satisfies
	$$T\lesssim \frac{L}{\bar\mu}\log\!\left(\frac{\bar\mu LD}{\sigma^2}\right)^+ + \frac{\sigma^2}{\bar\mu^2\bar{\mathcal{E}}} \leq \frac{L}{\bar\mu}\log\!\left(\frac{D}{\bar{\mathcal{E}}}\right)^+ + \frac{\sigma^2}{\bar{\mu}^2\bar{\mathcal{E}}},$$
	and for any specified $\delta\in(0,1)$, the corresponding tracking error satisfies
	$$\|X_K - \bar{X}_K\|^2 \lesssim \bar{\mathcal{E}} \log\!\left(\frac{e}{\delta}\right)$$
	with probability at least $1-\delta$, where $\bar{X}_K$ denotes the minimizer of $\psi_T$.
\end{theorem}

\subsection{Tracking the Equilibrium Value: Bounds in Expectation}
\label{minimalvalproofsexp}
We turn now to tracking the equilibrium value. To begin, we require a more flexible version of Lemma~\ref{lem:onestep0dec2} which holds in the static regularizer setting $r_t\equiv r$.

\begin{lemma}[Equilibrium one-step improvement]\label{lem:avgrecdec}
	The iterates $\{x_t\}$ produced by Algorithm~\ref{alg:decsgd} with $r_t\equiv r$ and $\eta_t<1/L$ satisfy the following bound for all indices $i,t\in\mathbb{N}$ and arbitrary $\alpha>0$:
	\begin{equation*}
		\begin{split}
			2\eta_i\big(\psi_t(x_{i+1})-\psi_t^\star\big) \leq (1-\bar\mu\eta_i)\|x_i - \bar{x}_t\|^2 &- \big(1 - (\gamma + \alpha)\eta_i\big)\|x_{i+1} - \bar{x}_t\|^2 \\
			&+ 2\eta_i\langle z_i, x_i - \bar{x}_t \rangle + \tfrac{\eta_i^2}{1-L\eta_i}\|z_i\|^2 + \tfrac{\eta_i}{\alpha}\bar{G}_{i,t}^2.
		\end{split}
	\end{equation*}
\end{lemma}
\begin{proof}
	Taking into account $r_t\equiv r$ and applying Lemma~\ref{lem:funcgap}, we have
	\begin{equation*}
		\begin{split}	
			\big[\psi_t(x_{i+1}) - \psi_t(\bar{x}_t)\big] - &\big[\varphi_i(x_{i+1}) - \varphi_i(\bar{x}_t)\big]\\ &~~~~~~~~~~~~= 	\big[f_{t,\bar{x}_t}(x_{i+1})-f_{t,\bar{x}_t}(\bar{x}_t)\big] - \big[f_{i,x_i}(x_{t+1})-f_{i,x_i}(\bar{x}_t)\big]\\
			&~~~~~~~~~~~~\leq \big(\bar{G}_{i,t} + \gamma\|x_i - \bar{x}_t\|\big)\|x_{i+1}-\bar{x}_t\|.
		\end{split} 
	\end{equation*}
	Hence
	\begin{equation*}
		\psi_t(x_{i+1})-\psi_t^\star \leq \varphi_i(x_{i+1}) - \varphi_i(\bar{x}_t) + \big(\bar{G}_{i,t} + \gamma\|x_i - \bar{x}_t\|\big)\|x_{i+1}-\bar{x}_t\|.
	\end{equation*}
	Moreover, Young's inequality implies
	\begin{equation*}
		\big(\bar{G}_{i,t} + \gamma\|x_i - \bar{x}_t\|\big)\|x_{i+1}-\bar{x}_t\| \leq \tfrac{\gamma}{2}\|x_i - \bar{x}_t\|^2 + \tfrac{\gamma + \alpha}{2}\|x_{i+1}-\bar{x}_t\|^2 + \tfrac{1}{2\alpha}\bar{G}_{i,t}^2.
	\end{equation*}
	Multiplying through by $2\eta_i$ and applying Lemma~\ref{lem:onestep1} completes the proof.
\end{proof}

Turning the estimate in Lemma~\ref{lem:avgrecdec} into an efficiency guarantee for the average iterate is essentially standard and follows for example from the averaging techniques used by \citet{drusvyatskiy2020stochastic}, \citet{ghadimi2012optimal}, and \citet{kulunchakov2020estimate}. The resulting progress along the average iterate is summarized in the following proposition, while the description of the key averaging lemma is placed in Appendix~\ref{apdx:avg}. Henceforth, we impose the regime (\ref{eqgapregime}): $\gamma<\mu/2$. 

\begin{proposition}[Progress along the average iterate]\label{gapbound}
	The iterates $\{\hat{x}_t\}$ produced by Algorithm \ref{alg:decsgdavg} with $r_t\equiv r$ and constant step size $\eta\leq1/2L$ satisfy the bound 
	\begin{equation*}
		\begin{split}
			\psi_t(\hat x_t) - \psi_t^\star \leq (1&- \hat \rho)^t \big(\psi_t(x_0) - \psi_t^\star + \tfrac{\hat\mu}{4}\|x_0 - \bar{x}_t\|^2\big) +
			\hat{\rho}\sum_{i=0}^{t-1} \langle z_i, x_i - \bar{x}_t \rangle(1 - \hat \rho)^{t-1-i} \\
			&+ \hat{\rho}\eta \sum_{i=0}^{t-1}\|z_i\|^2(1 - \hat \rho)^{t-1-i} + \tfrac{\hat\rho}{\hat\mu}\sum_{i=0}^{t-1}\bar{G}_{i,t}^2(1 - \hat \rho)^{t-1-i},
		\end{split}
	\end{equation*}
	where $\hat{\rho}:=\hat\mu\eta/(2-\mu\eta)$.
\end{proposition}
\begin{proof}
	Setting $\alpha = \hat\mu/2$ in Lemma \ref{lem:avgrecdec}, we obtain the following recursion for all indices $k\geq0$ and $t\geq1$:
	\begin{equation*}
		\rho\big(\psi_k(x_t)-\psi_k^\star\big) \leq (1-c_1\rho)V_{t-1}-(1+c_2\rho)V_t+\omega_t,
	\end{equation*} 
	where $\rho=2\eta$, $c_1=\bar\mu/2$, $c_2 = -\mu/4$, $V_i = \|x_i - \bar{x}_k\|^2$, and
	$\omega_t = 2\eta\langle z_{t-1}, x_{t-1} - \bar{x}_k \rangle + 2\eta^2\|z_{t-1}\|^2 + (2\eta/\hat\mu)\bar{G}_{t-1,k}^2$. The result follows by applying Lemma~\ref{lem:avg} with $h = \psi_k - \psi_k^\star$ and then taking $k=t$.
\end{proof}

Taking expectations in Proposition~\ref{gapbound} yields the following precise version of Theorem \ref{thm:exp_gap_dec}.

\begin{corollary}[Expected function gap]\label{cor2}
	Let $\{\hat{x}_t\}$ be the iterates produced by Algorithm \ref{alg:decsgdavg} with constant step size $\eta\leq1/2L$, set $\hat{\rho}:=\hat\mu\eta/(2-\mu\eta)$, and suppose that Assumption~\ref{assmp:funcgap0dec} holds. Then 
	\begin{equation}\label{eqn:cor2bd}
		\mathbb{E}\big[\psi_t(\hat x_t) - \psi_t^\star\big] \leq (1 - \hat \rho)^t\,\mathbb{E}\big[\psi_t(x_0)-\psi_t^\star + \tfrac{\hat\mu}{4}\|x_0 - \bar{x}_t\|^2\big] + \eta\sigma^2 + \frac{8\bar{\Delta}^2}{\hat\mu\eta^2}
	\end{equation}
	for all $t\geq0$. Consequently, we have
	\begin{equation*}
		\mathbb{E}\big[\psi_t(\hat x_t) - \psi_t^\star\big] \lesssim  (1-\hat{\rho})^t\big(\psi_0(x_0) - \psi_0^\star\big) + \eta\sigma^2 + \frac{\bar{\Delta}^2}{\hat\mu\eta^2}
	\end{equation*}
	for all $t\geq0$, and the following asymptotic error bound holds:
	\begin{equation*}
		\limsup_{t\to\infty}\mathbb{E}\big[\psi_t(\hat x_t) - \psi_t^\star\big] \leq \eta\sigma^2 + \frac{8\bar{\Delta}^2}{\hat\mu\eta^2}.
	\end{equation*}
\end{corollary}
\begin{proof}
	The bound (\ref{eqn:cor2bd}) follows by taking expectations in Proposition~\ref{gapbound} and noting
	\begin{equation*}
		\sum_{i=0}^{t-1}\mathbb{E}\|z_i\|^2(1-\hat{\rho})^{t-1-i}\leq\frac{\sigma^2}{\hat{\rho}}\quad\text{~~and~~}\quad\sum_{i=0}^{t-1}\mathbb{E}\,\bar{G}_{i,t}^2(1-\hat{\rho})^{t-1-i}\leq\frac{(\hat\mu\bar\Delta)^2(2-\hat{\rho})}{\hat{\rho}^3}
	\end{equation*} 
	by Assumption~\ref{assmp:funcgap0dec}. Next, applying Lemma~\ref{lem:funcgap}, Lemma~\ref{lem:grad_vs_dist}, and Young's inequality together with the $\mu$-strong convexity of $\psi_0$ yields
	\begin{equation}\label{initialgapbound0}
		\psi_t(x_0) - \psi_t^\star + \tfrac{\hat\mu}{4}\|x_0 - \bar{x}_t\|^2 \leq 3\big(\psi_0(x_0) - \psi_0^\star\big) + 5\bar{G}_{0,t}^2/\bar\mu,
	\end{equation}
	and then taking expectations and invoking Assumption~\ref{assmp:funcgap0dec} gives
	\begin{equation}\label{initialgapbound}
		\mathbb{E}\big[\psi_t(x_0) - \psi_t^\star + \tfrac{\hat\mu}{4}\|x_0 - \bar{x}_t\|^2\big] \leq 3\big(\psi_0(x_0) - \psi_0^\star\big) + 5\hat\mu\bar{\Delta}^2t^2.
	\end{equation}
	Further, the inequality 
	\begin{equation}\label{ineq:expquad}
		e^{-\hat\mu\eta t/2}\hat\mu t^2\leq16/\hat\mu\eta^2  \qquad\forall\hat\mu,\eta,t>0
	\end{equation}
	combines with inequality (\ref{initialgapbound}) to yield 
	\begin{equation*}
		(1 - \hat{\rho})^t\, \mathbb{E}\big[\psi_t(x_0) - \psi_t^\star + \tfrac{\hat\mu}{4}\|x_0 - \bar{x}_t\|^2\big]\leq 3(1 - \hat{\rho})^t\big(\psi_0(x_0) - \psi_0^\star\big) + \frac{80\bar{\Delta}^2}{\hat\mu\eta^2},
	\end{equation*}
	and the remaining assertions of the corollary follow.
\end{proof}

We may now apply Corollary~\ref{cor2} to obtain a formal version of Theorem~\ref{thm:gap_schedul_dec}; the proof closely follows that of Theorem~\ref{thm:time_to_track} and is included in Appendix~\ref{missingproofs} (see Section~\ref{pf:gap_schedul}).

\begin{theorem}[Time to track in expectation]\label{thm:time_to_track_gap_exp}
	Suppose that Assumption~\ref{assmp:funcgap0dec} holds and that we are in the low drift-to-noise regime $\bar\Delta/\sigma<\sqrt{\hat\mu/16L^3}$. Set $\hat\eta_\star = (2\bar{\Delta}^2/\hat\mu\sigma^2)^{1/3}$ and $\widehat{\mathcal{G}}=
	\hat\mu(\bar\Delta\sigma^2/\hat{\mu}^2)^{2/3}$. Suppose moreover that we have available a positive upper bound on the initial gap $D\geq \psi_0(x_0) - \psi_0^\star$.
	Consider running Algorithm~\ref{alg:decsgdavg} in $k=0,\ldots, K-1$ epochs, namely, set $X_0=x_0$ and iterate the process
	$$X_{k+1}=\mathtt{D\text{-}\overline{PSG}}(X_k,\mu,\gamma,\eta_k,T_k)\quad \text{for}\quad k=0,\ldots, K-1,$$ where the number of epochs is $$K=1+\left\lceil{\log_2\!\left(\frac{1}{L}\cdot\left(\frac{\sigma^2\hat\mu}{\bar{\Delta}^2}\right)^{1/3}\right)}\right\rceil$$ and we set
	$$\eta_0=\frac{1}{2L},\quad T_0=\left\lceil{\frac{4L}{\hat\mu}\log\!\left(\frac{LD}{\sigma^2}\right)^+}\right\rceil\quad \text{and}\quad \eta_k=\frac{\eta_{k-1}+ \hat{\eta}_{\star}}{2},\quad T_k=\left\lceil{\frac{2\log(12)}{\hat\mu\eta_k}}\right\rceil\quad \forall k\geq 1.$$ Then the time horizon $T=T_0+\cdots+T_{K-1}$ satisfies
	$$T\lesssim \frac{L}{\hat\mu}\log\!\left(\frac{L D}{\sigma^2}\right)^+ +\frac{\sigma^2}{\hat\mu\widehat{\mathcal{G}}} \leq \frac{L}{\hat\mu}\log\!\left(\frac{D}{\widehat{\mathcal{G}}}\right)^+ +\frac{\sigma^2}{\hat\mu\widehat{\mathcal{G}}}$$
	and the corresponding tracking error satisfies
	$\mathbb{E}\big[\psi_T(X_K) - \psi_T^\star\big] \lesssim \widehat{\mathcal{G}}.$
\end{theorem}

\subsection{Tracking the Equilibrium Value: High-Probability Guarantees}
\label{minimalvalproofsprob}
In this section, we derive the high-probability analogues of the results in Section~\ref{minimalvalproofsexp}. In light of Proposition \ref{gapbound}, we seek upper bounds on the sums
$$\sum_{i=0}^{t-1} \langle z_i, x_i - \bar{x}_t \rangle(1 - \hat \rho)^{t-1-i},\quad 
\sum_{i=0}^{t-1}\|z_i\|^2(1 - \hat \rho)^{t-1-i}, \quad \sum_{i=0}^{t-1}\bar{G}_{i,t}^2(1 - \hat \rho)^{t-1-i}
$$
that hold with high probability. The last two sums can easily be estimated under boundedness or light-tail assumptions on $\|z_i\|$ and $\bar{G}_{i,t}$. Controlling the first sum is more challenging because the error  $\|x_i - \bar{x}_t\|$ may in principle grow large. In order to control this term, we will use a remarkable generalization of Freedman's inequality, recently proved by \citet{pmlr-v99-harvey19a} for the purpose of analyzing the stochastic gradient method on static nonsmooth problems (without a regularizer). 

The main idea is as follows. Fix a horizon $t$, assume $\mathbb{E}[z_i\,|\,\mathcal{F}_{i,t}]=0$ for all $0 \leq i < t$ (recall that $\mathcal{F}_{i,t}:=\sigma(\mathcal{F}_i,\bar{x}_t)$), and define the martingale difference sequence
$$d_i:=\langle z_i, x_i - \bar{x}_t \rangle(1 - \hat \rho)^{t-1-i}$$
adapted to the filtration $(\mathcal{F}_{i+1,t})_{i=0}^{t-1}$. Roughly speaking, under mild light-tail assumptions, the total conditional variance of the corresponding martingale $\sum_{i=0}^{t-1}d_i$ can be bounded above with high probability by an affine transformation of itself, i.e., by an affine combination of the sequence $\{d_i\}_{i=0}^{t-1}$. In this way, the martingale is self-regulating. This is the content of the following proposition. The proof follows from Lemma~\ref{lem:avgrecdec} and algebraic manipulation and is placed in Appendix~\ref{missingproofs} (see Section~\ref{pf:avgrec2}).

\begin{proposition}[Self-regulation]\label{prop:avgrec2}
	The iterates $\{x_t\}$ produced by Algorithm \ref{alg:decsgd} with $r_t\equiv r$ and constant step size $\eta\leq1/2L$ satisfy the following bound for all $\lambda\in(0, \bar\mu\eta]$:
	\begin{equation*}
		\begin{split}
			\sum_{i=0}^{t-1}\|x_i - \bar{x}_t\|^2(1-\lambda)^{2(t-1-i)}\leq &\sum_{j=0}^{t-2}\left(2\eta\sum_{i=j+1}^{t-1}(1-\lambda)^{t-2-i}\right)\!\langle z_j, x_j - \bar{x}_t\rangle (1-\lambda)^{t-1-j} \\ 
			&+ \tfrac{1}{\lambda}(1-\lambda)^{t-1}\|x_0 - \bar{x}_t\|^2 + \tfrac{2\eta^2}{\lambda}\sum_{j=0}^{t-2}\|z_j\|^2(1-\lambda)^{t-2-j} \\
			&+ \tfrac{\eta}{\bar\mu\lambda}\sum_{j=0}^{t-2}\bar{G}_{j,t}^2(1-\lambda)^{t-2-j}.
		\end{split}
	\end{equation*}
\end{proposition}

In order to bound the self-regulating martingale $\sum_{i=0}^{t-1}d_i$, we will use the generalized Freedman inequality developed by \citet{pmlr-v99-harvey19a}, or rather a direct consequence thereof \citep[see][Lemma~C.3]{pmlr-v99-harvey19a}.

\begin{theorem}[Consequence of generalized Freedman]\label{thmfreed}
	Let $(D_i)_{i=0}^n$ and $(V_i)_{i=0}^n$ be scalar stochastic processes on a probability space with filtration $(\mathcal{H}_i)_{i=0}^{n+1}$ satisfying
	\begin{equation*}
		\mathbb{E}[\exp(\lambda D_i)\,|\,\mathcal{H}_i]\leq \exp(\lambda^2 V_i/2) \quad \textit{for all} \quad \lambda\geq0.
	\end{equation*}
	Suppose that $D_i$ is $\mathcal{H}_{i+1}$-measurable with $\mathbb{E}|D_i|<\infty$ and $\mathbb{E}[D_i\,|\,\mathcal{H}_i]=0$, and that $V_i$ is nonnegative and $\mathcal{H}_i$-measurable. Suppose moreover that there are constants $\alpha_0,\ldots,\alpha_n\geq0$, $\delta\in[0,1]$, and $\beta(\delta)\geq0$ satisfying
	\begin{equation*}
		\mathbb{P}\left\{\sum_{i=0}^{n} V_i\leq \sum_{i=0}^{n}\alpha_iD_i + \beta(\delta) \right\}\geq 1-\delta.
	\end{equation*}
	Set $\alpha:=\max\{\alpha_0,\ldots, \alpha_n\}$. Then for all $\tau>0$, the following bound holds:
	\begin{equation*}
		\mathbb{P}\left\{\sum_{i=0}^{n} D_i\geq \tau \right\}\leq \delta + \exp\!\left(-\frac{\tau}{4\alpha+8\beta(\delta)/\tau}\right)\!.
	\end{equation*}
\end{theorem}

Combining Proposition~\ref{prop:avgrec2} and Theorem~\ref{thmfreed} yields the following tail bound for $\sum_{i=0}^{t-1}d_i$. 

\begin{proposition}[Noise martingale tail bound]\label{prop:tailbd1}
	Let $\{x_t\}$ be the iterates produced by Algorithm \ref{alg:decsgd} with constant step size $\eta\leq1/2L$, set $\hat{\rho}:=\hat\mu\eta/(2-\mu\eta)$, and suppose that Assumption~\ref{assmp:funcgap2dec} holds. Then there is an absolute constant $c>0$ such that for any specified $t\in\mathbb{N}$, $\delta\in(0,1)$, and $\tau>0$, the following bound holds:
	\begin{equation*}
		\mathbb{P}\left\{\sum_{i=0}^{t-1}\langle z_i,x_i - \bar{x}_t \rangle(1-\hat\rho)^{t-1-i}\geq \tau \right\} \leq \delta + \exp\!\left(-\frac{\tau}{4\alpha+8\beta_t\log(3e/\delta)/\tau}\right)\!,
	\end{equation*}
	where $\alpha := 3\eta(c\sigma)^2/\hat{\rho}$ and 
	\begin{equation*}
		\beta_t:=(1-\hat{\rho})^{t-1}\big(\|x_0 - \bar{x}_0\|^2 + \bar{\Delta}^2t^2\big)\frac{2(c\sigma)^2}{\hat{\rho}} + \frac{2\eta^2(c\sigma)^4}{\hat{\rho}^2} + \frac{3\hat\mu\bar{\Delta}^2\eta(c\sigma)^2}{\hat{\rho}^4}.
	\end{equation*}
\end{proposition}
\begin{proof}
		By Assumption~\ref{assmp:funcgap2dec}, there exists an absolute constant $c\geq1$ such that $\|z_i\|^2$ is sub-exponential conditioned on $\mathcal{F}_{i,t}$ with parameter $c\sigma^2$ and $z_i$ is mean-zero sub-Gaussian conditioned on $\mathcal{F}_{i,t}$ with parameter $c\sigma$ for all indices $0\leq i < t$. Then for each $0\leq i < t$, the $\mathcal{F}_{i+1,t}$-measurable random variable $\langle z_i,x_i - \bar{x}_t \rangle$ is mean-zero sub-Gaussian conditioned on $\mathcal{F}_{i,t}$ with parameter $c\sigma\|x_i - \bar{x}_t\|$, so
	\begin{equation*}
		\mathbb{E}\big[{\exp}\big(\lambda\langle z_i, x_i - \bar{x}_t \rangle(1-\hat\rho)^{t-1-i}\big)\,|\,\mathcal{F}_{i,t}\big]\leq \exp\!\big(\lambda^2(c\sigma)^2\|x_i - \bar{x}_t\|^2(1-\hat\rho)^{2(t-1-i)}/2\big) \quad\forall\lambda\in\mathbb{R};
	\end{equation*}
	note also that $\mathbb{E}\vert\langle z_i,x_i - \bar{x}_t \rangle\rvert<\infty$ by H\"{o}lder's inequality, Assumption~\ref{assump:sec_moment_dec}, and Corollary~\ref{cor1}. Now fix $t\geq1$ and observe that Proposition~\ref{prop:avgrec2} yields the total conditional variance bound
	\begin{equation*}
		\sum_{i=0}^{t-1}(c\sigma)^2\|x_i - \bar{x}_t\|^2(1-\hat\rho)^{2(t-1-i)}\leq \sum_{j=0}^{t-2}\alpha_j\langle z_j, x_j - \bar{x}_t \rangle(1-\hat{\rho})^{t-1-j} + R_t,
	\end{equation*}
	where $0\leq\alpha_j\leq\alpha$ for all $0\leq j\leq t-2$ and
	\begin{equation*}
		R_t := \tfrac{(c\sigma)^2}{\hat\rho}(1-\hat{\rho})^{t-1}\|x_0 - \bar{x}_t\|^2 + \tfrac{2\eta^2(c\sigma)^2}{\hat{\rho}}\sum_{j=0}^{t-2}\|z_j\|^2(1-\hat\rho)^{t-2-j} + \tfrac{\eta(c\sigma)^2}{\bar\mu\hat{\rho}}\sum_{j=0}^{t-2}\bar{G}_{j,t}^2(1-\hat\rho)^{t-2-j}.
	\end{equation*}
	We claim
	\begin{equation}\label{eqn:beta-bound}
		\mathbb{P}\left\{R_t \leq \beta_t\log\!\left(\frac{3e}{\delta}\right)\right\}\geq 1-\delta \qquad\forall\delta\in(0,1).
	\end{equation}
	To verify (\ref{eqn:beta-bound}), observe first that for all $n\geq0$, the sum $\sum_{i=0}^{n}\|z_i\|^2(1-\hat\rho)^{n-i}$ is sub-exponential with parameter $\sum_{i=0}^{n}c\sigma^2(1-\hat\rho)^{n-i}\leq(c\sigma)^2/\hat{\rho}$, so Markov's inequality implies
	\begin{equation}\label{sqnoisetail}
		\mathbb{P}\left\{	\sum_{i=0}^{n}\|z_i\|^2(1-\hat\rho)^{n-i} \leq \frac{(c\sigma)^2}{\hat{\rho}}\log\!\left(\frac{e}{\delta}\right)\right\}\geq 1-\delta\qquad\forall\delta\in(0,1).
	\end{equation}
	Further, for all $0\leq n<t$, it follows from Assumption~\ref{assmp:funcgap2dec} and Lemma~\ref{lem:grad_vs_dist} that $\|x_0 - \bar{x}_t\|^2$ is sub-exponential with parameter $2(\|x_0 - \bar{x}_0\|^2 + \bar{\Delta}^2t^2)$ and $\sum_{i=0}^{n}\bar{G}_{i,t}^2(1-\hat\rho)^{n-i}$ is sub-exponential with parameter
	\begin{equation*}
		\sum_{i=0}^{n}(\hat\mu\bar\Delta)^2(t-i)^2(1-\hat\rho)^{n-i} = (\hat\mu\bar\Delta)^2(1-\hat{\rho})^{n+1-t}\sum_{i=0}^{n}(t-i)^2(1-\hat\rho)^{t-i-1} \leq \frac{2(\hat\mu\bar\Delta)^2}{\hat{\rho}^3(1-\hat{\rho})^{t-1-n}},
	\end{equation*}
	so Markov's inequality implies
	\begin{equation}
		\mathbb{P}\left\{\|x_0 - \bar{x}_t\|^2\leq 2(\|x_0 - \bar{x}_0\|^2 + \bar{\Delta}^2t^2)\log\!\left(\frac{e}{\delta}\right)\right\}\geq 1-\delta\qquad\forall\delta\in(0,1)
	\end{equation}
	and
	\begin{equation}\label{funcgaptail}
		\mathbb{P}\left\{	\sum_{i=0}^{n}\bar{G}_{i,t}^2(1-\hat\rho)^{n-i}\leq\frac{2(\hat\mu\bar\Delta)^2}{\hat{\rho}^3(1-\hat{\rho})^{t-1-n}}\log\!\left(\frac{e}{\delta}\right)\right\}\geq 1-\delta \qquad\forall\delta\in(0,1).
	\end{equation}
	Thus, (\ref{sqnoisetail})--(\ref{funcgaptail}) and a union bound yield (\ref{eqn:beta-bound}). Consequently, Theorem \ref{thmfreed} implies that the following bound holds for all $\delta\in(0,1)$ and $\tau>0$:
	\begin{equation*}
		\mathbb{P}\left\{\sum_{i=0}^{t-1}\langle z_i, x_i - \bar{x}_t \rangle(1-\hat\rho)^{t-1-i}\geq\tau \right\}\leq\delta + \exp\!\left(-\frac{\tau}{4\alpha+8\beta_t\log(3e/\delta)/\tau}\right)\!,
	\end{equation*}
	as claimed.
\end{proof}

We may now deduce the following precise version of Theorem~\ref{thm:hpbound2dec} using the tail bound furnished by Proposition~\ref{prop:tailbd1}.

\begin{theorem}[Function gap with high probability]\label{thm:hpbound2precise}
	Let $\{\hat{x}_t\}$ be the iterates produced by Algorithm~\ref{alg:decsgdavg} with constant step size $\eta\leq1/2L$, set $\hat{\rho}:=\hat\mu\eta/(2-\mu\eta)$, and suppose that Assumption~\ref{assmp:funcgap2dec} holds. Then there is an absolute constant $c>0$ such that for any specified $t\in\mathbb{N}$ and $\delta\in(0,1)$, the following estimate holds with probability at least $1-\delta$: 
	\begin{equation*}
		\psi_t(\hat x_t) - \psi_t^\star\leq (1 - \hat \rho)^t \big(\psi_t(x_0) - \psi_t^\star + \tfrac{\hat\mu}{4}\|x_0 - \bar{x}_t\|^2\big) + \left(\eta(c\sigma)^2 + \frac{8\bar{\Delta}^2}{\hat\mu\eta^2} + 5\hat{\rho}\sqrt{8\beta_t}\right)\!\log\!\left(\frac{4e}{\delta}\right)\!,
	\end{equation*}
	where
	\begin{equation*}
		\beta_t:=(1-\hat{\rho})^{t-1}\big(\|x_0 - \bar{x}_0\|^2 + \bar{\Delta}^2t^2\big)\frac{2(c\sigma)^2}{\hat{\rho}} + \frac{2\eta^2(c\sigma)^4}{\hat{\rho}^2} + \frac{3\hat\mu\bar{\Delta}^2\eta(c\sigma)^2}{\hat{\rho}^4}.
	\end{equation*}
\end{theorem}
\begin{proof}
	A quick computation shows that given any $\delta\in(0,1)$, we may take
	\begin{equation*}
		\tau=5\sqrt{8\beta_t}\log\!\left(\frac{e}{\delta}\right)
	\end{equation*} 
	in Proposition~\ref{prop:tailbd1} to obtain 
	\begin{equation}\label{noisetail2}
		\mathbb{P}\left\{\sum_{i=0}^{t-1}\langle z_i,x_i - \bar{x}_t \rangle(1-\hat\rho)^{t-1-i} < 5\sqrt{8\beta_t}\log\!\left(\frac{e}{\delta}\right) \right\} \geq 1 - 2\delta.
	\end{equation}
	We may now combine (\ref{sqnoisetail}), (\ref{funcgaptail}), and (\ref{noisetail2}) together with Proposition \ref{gapbound} and a union bound to conclude that for all $\delta\in(0,1)$, the estimate
	\begin{equation*}
			\psi_t(\hat x_t) - \psi_t^\star \leq (1 - \hat \rho)^t\big(\psi_t(x_0) - \psi_t^\star + \tfrac{\hat\mu}{4}\|x_0 - \bar{x}_t\|^2\big)  + \left(\eta(c\sigma)^2 + \frac{2\hat\mu\bar{\Delta}^2}{\hat{\rho}^2} + 5\hat{\rho}\sqrt{8\beta_t}\right)\!\log\!\left(\frac{e}{\delta}\right)
	\end{equation*}
	holds with probability at least $1-4\delta$; noting $\hat{\rho}\geq \hat{\mu}\eta/2$ completes the proof. 
\end{proof}

\begin{remark}\label{rem:simp2}
	\textup{To see that Theorem \ref{thm:hpbound2precise} entails Theorem \ref{thm:hpbound2dec}, observe first that in the setting of Theorem~\ref{thm:hpbound2precise}, upon setting $C:=\max\{c,1\}$ and selecting any $t\in\mathbb{N}$, we have
	\begin{equation*}
		\begin{split}
			\hat{\rho}\sqrt{8\beta_t}\leq 4C^2\bigg(\sqrt{(1-\hat{\rho})^t\big(\|x_0 - \bar{x}_0\|^2+\bar{\Delta}^2t^2\big)\hat\mu\eta\sigma^2} + \eta\sigma^2+\sqrt{6}\frac{\bar\Delta\sigma}{\sqrt{\hat\mu\eta}}\bigg), 
		\end{split}
	\end{equation*}
	while the AM-GM inequality implies $$2\sqrt{(1-\hat{\rho})^t\big(\|x_0 - \bar{x}_0\|^2+\bar{\Delta}^2t^2\big)\hat\mu\eta\sigma^2} \leq (1-\hat{\rho})^t\big(\hat\mu\|x_0 - \bar{x}_0\|^2+\hat\mu\bar{\Delta}^2t^2\big) + \eta\sigma^2,$$ inequality (\ref{ineq:expquad}) implies
	\begin{equation*}
		(1-\hat{\rho})^t\big(\hat\mu\|x_0 - \bar{x}_0\|^2+\hat\mu\bar{\Delta}^2t^2\big) \leq 2(1- \hat \rho)^t \big(\psi_0(x_0) - \psi_0^\star\big) + \frac{16\bar{\Delta}^2}{\hat\mu\eta^2},
	\end{equation*}
	and Young's inequality implies $$\frac{2\bar{\Delta}\sigma}{\sqrt{\hat\mu\eta}} \leq \eta\sigma^2 + \frac{\bar{\Delta}^2}{\hat\mu\eta^2}.$$ Hence
	$$\hat{\rho}\sqrt{8\beta_t} \lesssim (1 - \hat{\rho})^t\big(\psi_0(x_0) - \psi_0^\star\big) + \eta\sigma^2 + \frac{\bar{\Delta}^2}{\hat\mu\eta^2}.$$
	Further, inequalities (\ref{initialgapbound0}) and (\ref{ineq:expquad}) together with Assumption~\ref{assmp:funcgap2dec} imply that the estimate
	\begin{equation*}
		(1- \hat \rho)^t \big(\psi_t(x_0) - \psi_t^\star + \tfrac{\hat\mu}{4}\|x_0 - \bar{x}_t\|^2\big) \leq 3(1- \hat \rho)^t \big(\psi_0(x_0) - \psi_0^\star\big) + \frac{80\bar{\Delta}^2}{\hat\mu\eta^2}\log\!\left(\frac{e}{\delta}\right)
	\end{equation*}
	holds with probability at least $1-\delta$ for all $\delta\in(0,1)$. 
	On the other hand, Theorem~\ref{thm:hpbound2precise} shows that the estimate
		\begin{align*}
			\psi_t(\hat x_t) - \psi_t^\star \leq (1 - \hat \rho)^t \big(\psi_t(x_0) - \psi_t^\star + \tfrac{\hat\mu}{4}\|x_0 - \bar{x}_t\|^2\big) + \left(\eta(c\sigma)^2 + \frac{8\bar{\Delta}^2}{\hat\mu\eta^2} + 5\hat{\rho}\sqrt{8\beta_t}\right)\!\log\!\left(\frac{4e}{\delta}\right)
		\end{align*}
	holds with probability at least $1 - \delta$ for all $\delta\in(0,1)$.
	Thus, a union bound reveals that the estimate
	\begin{equation*}
		\psi_t(\hat x_t) - \psi_t^\star \lesssim \left((1 - \hat{\rho})^t\big(\psi_0(x_0) - \psi_0^\star\big) + \eta\sigma^2 + \frac{\bar{\Delta}^2}{\hat\mu\eta^2}\right)\!\log\!\left(\frac{e}{\delta}\right)
	\end{equation*}
	holds with probability at least $1-\delta$ for all $\delta\in(0,1)$.}
\end{remark}

We may now apply Theorem~\ref{thm:hpbound2dec} to obtain a formal version of Theorem~\ref{thm:gap_schedulhp2dec}; the proof is analogous to that of Theorem~\ref{thm:time_to_track2} and appears in Appendix~\ref{missingproofs} (see Section~\ref{pf:time_to_track_gap_hp2}). 

\begin{theorem}[Time to track with high probability]\label{thm:time_to_track_gap_hp2}
	Suppose that Assumption~\ref{assmp:funcgap2dec} holds and that we are in the low drift-to-noise regime $\bar{\Delta}/\sigma<\sqrt{\hat\mu/16L^3}$. Set $\hat{\eta}_\star = (2\bar{\Delta}^2/\hat\mu\sigma^2)^{1/3}$ and $\widehat{\mathcal{G}}=
	\hat\mu(\bar\Delta\sigma^2/\hat{\mu}^2)^{2/3}$. Suppose moreover that we have available a positive upper bound on the initial gap $D\geq \psi_0(x_0) - \psi_0^\star$.
	Fix $\delta\in(0,1)$ and consider running Algorithm~\ref{alg:decsgdavg} in $k=0,\ldots, K-1$ epochs, namely, set $X_0=x_0$ and iterate the process
	$$X_{k+1}=\mathtt{D\text{-}\overline{PSG}}(X_k,\mu,\gamma,\eta_k,T_k)\quad \text{for}\quad k=0,\ldots, K-1,$$ where the number of epochs is $$K=1+\left\lceil{\log_2\!\left(\frac{1}{L}\cdot\left(\frac{\sigma^2\hat\mu}{\bar{\Delta}^2}\right)^{1/3}\right)}\right\rceil$$ and we set
	$$\eta_0=\frac{1}{2L},\quad T_0=\left\lceil\frac{4L}{\hat\mu}\log\!\left(\frac{LD}{\sigma^2}\right)^+\right\rceil\quad \text{and}\quad \eta_k=\frac{\eta_{k-1}+ \hat{\eta}_{\star}}{2},\quad T_k=\left\lceil{\frac{2\log\!\big(4c\log(e/\delta)\big)^+}{\hat\mu\eta_k}}\right\rceil$$ for all $k\geq1$, where $c>0$ is the absolute constant furnished by the bound (\ref{ineq:hpbound2dec}). Then the time horizon $T=T_0+\cdots+T_{K-1}$ satisfies
	$$T\lesssim \frac{L}{\hat\mu}\log\!\left(\frac{L D}{\sigma^2}\right)^+ + \frac{\sigma^2}{\hat\mu\widehat{\mathcal{G}}}\!\left(1 \vee \log\log\frac{e}{\delta} \right) \leq \frac{L}{\hat\mu}\log\!\left(\frac{D}{\widehat{\mathcal{G}}}\right)^+ + \frac{\sigma^2}{\hat\mu\widehat{\mathcal{G}}}\left(1 \vee \log\log\frac{e}{\delta}\right)$$
	and the corresponding tracking error satisfies
	$$\psi_T(X_K) - \psi_T^\star \lesssim \widehat{\mathcal{G}}\log\!\left(\frac{e}{\delta}\right)$$
	with probability at least $1-K\delta$.
\end{theorem}
 \section{Numerical Illustrations}
\label{sec:exp} 
We investigate the empirical behavior of our finite-time bounds on numerical examples with synthetic data. 
We consider examples of
a) least-squares recovery; 
b) sparse least-squares recovery; c) $\ell_2^2$-regularized logistic regression; 
and investigate the behavior of  $\|x_t - x_t^\star\|^2$ and $\varphi_t(\hat{x}_t) - \varphi_t^\star$ in each case. 
The main findings are that our bounds exhibit: 1) the correct dependence on $\eta$, $\sigma$, and $\Delta$; 2) excellent coverage in Monte-Carlo simulations. Code is available online at \url{https://github.com/joshuacutler/TimeDriftExperiments}.

\smallskip
\paragraph{Least-squares recovery.}

Fix $x_0, x_0^\star\in\mathbb{R}^d$ and consider a Gaussian random walk $\{x_t^\star\}$  given by $x_{t+1}^\star = x_{t}^\star + v_t$, where $v_t$ is drawn uniformly from the sphere of radius $\Delta$ in $ \mathbb{R}^d$. Given a fixed rank-$d$ matrix $A\in\mathbb{R}^{n\times d}$ with minimum singular value $\sqrt{\mu}$ and maximum singular value $\sqrt{L}$, we aim to recover $\{x_t^\star\}$ via the online least-squares problem
\begin{equation*}
	\min_{x\in\mathbb{R}^d}\,\underset{y\sim\mathcal{P}_t}{\mathbb{E}}\tfrac{1}{2}\|Ax - y\|^2,
\end{equation*} 
where $\mathcal{P}_t = \mathsf{N}(Ax_t^\star,\Sigma_t)$ with covariance matrix $\Sigma_t$ satisfying $\tr{\Sigma_t} \leq \sigma^2/L$.  This amounts to the problem \eqref{eqn:prob_online} with $f_t(x) = \mathbb{E}_{y\sim\mathcal{P}_t}\tfrac{1}{2}\|Ax - y\|^2$ and $r_t = 0$, and the minimizer and gradient drift satisfy 
\begin{equation*}
	\|\nabla f_t(x) - \nabla f_{t+1}(x)\| = \|A^{\top}A(x^{\star}_t - x^{\star}_{t+1})\| \leq L\|x^{\star}_t - x^{\star}_{t+1}\| = L\Delta 
\end{equation*}
for all $x\in\mathbb{R}^d$. We implement Algorithms~\ref{alg:sgd} and \ref{alg:sgdavg} using the sample gradient $g_t = A^T(Ax_t - y_t)$ at step $t$ with $y_t\sim\mathcal{P}_t$; the gradient noise $z_t = A^{\top}(y_t - Ax^\star_t)\sim\mathsf{N}(0,A^{\top}\Sigma_{t}A)$ satisfies $\mathbb{E}\|z_t\|^2 \leq L\tr{\Sigma_t} \leq \sigma^2$.

In our simulations, we set $d=50$, $n=100$, and  $\Sigma_t = (\sigma^2/nL)I_n$ for all $t$, where $I_n$ denotes the $n\times n$ identity matrix. We initialize $x_0$ and $x_0^{\star}$ using standard Gaussian entries and generate $A$ via singular value decomposition with Haar-distributed orthogonal matrices. In Figures~\ref{fig:fig1} and \ref{fig:fig2}, we use default parameter values $\mu=L=1$, $\sigma = 10$,  $\Delta = 1$, and the corresponding asymptotically optimal step size $\eta = \eta_{\star}$. Since $f_t$ is $\mu$-strongly convex and $L$-smooth, this puts us in the low drift/noise regime in Figure~\ref{fig:fig1}: $\Delta/\sigma < \sqrt{\mu/16L^3}=1/4$. To estimate the expected values and confidence intervals of $\|x_t - x_t^\star\|^2$ and $\varphi_t(\hat{x}_t) - \varphi_t^\star$, we run $100$ trials with horizon $T=100$. 

\begin{figure}[h!]
	\centering
	\includegraphics[width=0.98\textwidth]{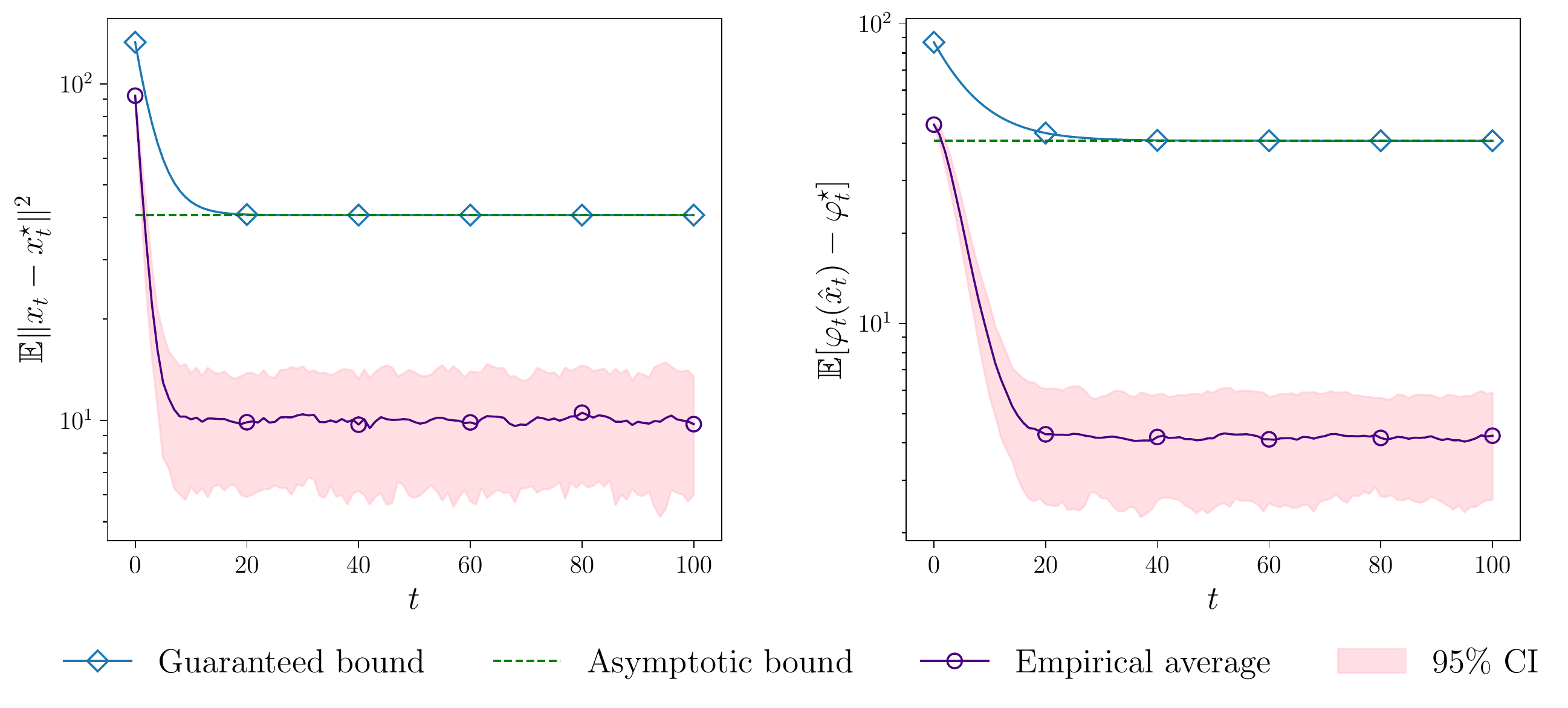}
	\caption{\small Semilog plots of guaranteed bounds and empirical tracking errors with respect to iteration $t$ for least-squares recovery. Shaded regions indicate the $95\%$ confidence intervals for $\|x_t - x_t^\star\|^2$ and $\varphi_t(\hat{x}_t) - \varphi_t^\star$; empirical averages and confidence intervals are computed over $100$ trials. Default parameter values: $\mu=L=1$, $\sigma = 10$,  $\Delta = 1$, and $\eta = \eta_{\star}$.}
	\label{fig:fig1}
\end{figure}

\begin{figure}[h!]
	\centering
	\includegraphics[width=0.98\textwidth]{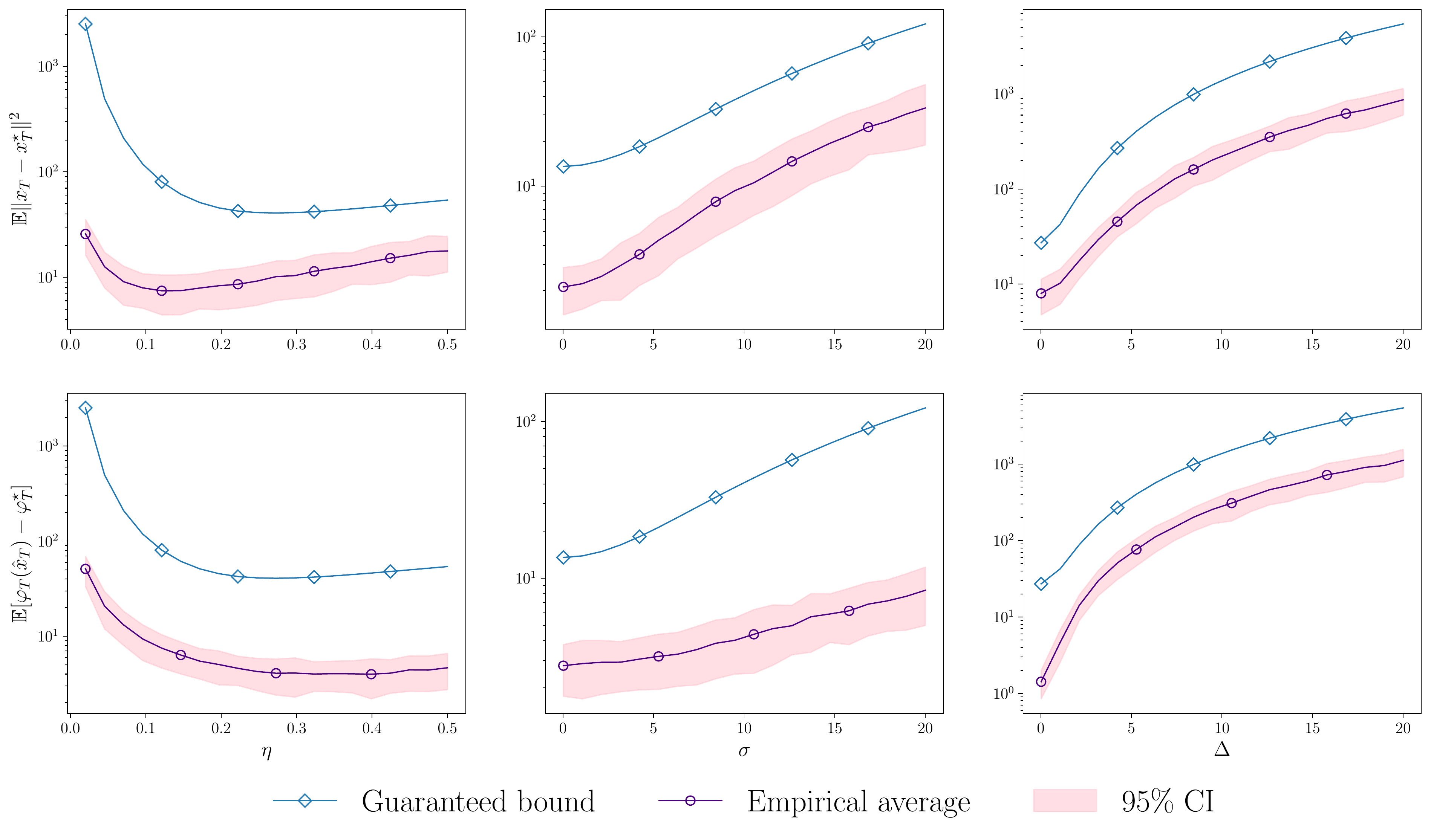}
	\caption{\small Semilog plots of guaranteed bounds and empirical tracking errors at horizon $T= 100$ with respect to $\eta$, $\sigma$, and $\Delta$ for least-squares recovery. Shaded regions indicate the $95\%$ confidence intervals for  $\|x_T - x_T^\star\|^2$ and $\varphi_T(\hat{x}_T) - \varphi_T^\star $; empirical averages and confidence intervals are computed over $100$ trials. Default parameter values: $\mu=L=1$, $\sigma = 10$,  $\Delta = 1$, and $\eta = \eta_{\star}$.}
	\label{fig:fig2}
\end{figure}

\clearpage
\paragraph{Sparse least-squares recovery.}
Next, we consider least-squares recovery constrained to the closed $\ell_1$-ball in $\mathbb{R}^d$, which we denote by $B_1$. We aim to recover a sequence of sparse vectors in $B_1$ generated as follows. Set $s = \lfloor\log{d} \rfloor$, draw a vector $u$ uniformly from the  $\ell_1$-ball in $\mathbb{R}^s$, fix $x_0^\star = (u,0)\in\mathbb{R}^d$, and select $\Delta\in(0,\sqrt{2}]$. At step $t$, with probability $p = (4-2\Delta^2)/(4-\Delta^2)$, we set $x_{t+1}^\star=x_t^\star + v_t $, where $v_t$ is selected to have the same support as $x_t^\star$ and satisfy $\|v_t\|=\Delta/\sqrt{2}$ and $x_t^\star + v_t \in B_1$; otherwise, with probability $1-p$, we obtain $x_{t+1}^\star$ from $x_t^\star$ by swapping precisely one nonzero coordinate with a zero coordinate. The resulting sequence $\{x_t^\star\}$ in $B_1$ satisfies $\mathbb{E}\| x_t^\star - x_{t+1}^\star\|^2 \leq \Delta^2$. Given a fixed rank-$d$ matrix $A\in\mathbb{R}^{n\times d}$ with minimum singular value $\sqrt{\mu}$ and maximum singular value $\sqrt{L}$, we aim to recover $\{x_t^\star\}$ via the online constrained least-squares problem
\begin{equation*}
	\min_{x\in B_1}\,\underset{y\sim\mathcal{P}_t}{\mathbb{E}}\tfrac{1}{2}\|Ax - y\|^2,
\end{equation*} 
where $\mathcal{P}_t = \mathsf{N}(Ax_t^\star,\Sigma_t)$ with covariance matrix $\Sigma_t$ satisfying $\tr{\Sigma_t} \leq \sigma^2/L$. This amounts to the problem \eqref{eqn:prob_online} with $f_t(x) = \mathbb{E}_{y\sim\mathcal{P}_t}\tfrac{1}{2}\|Ax - y\|^2$ and $r_t = \delta_{B_1}$ (the convex indicator of $B_1$), and the minimizer and gradient drift satisfy
\begin{equation*}
	\mathbb{E}\big[\textstyle{\sup_{x}}\|\nabla f_t(x) - \nabla f_{t+1}(x)\|^2\big] \leq L^2\,\mathbb{E}\|x^{\star}_t - x^{\star}_{t+1}\|^2 \leq (L\Delta)^2. 
\end{equation*}
Fixing $x_0$ drawn uniformly from $B_1$, we implement Algorithms~\ref{alg:sgd} and \ref{alg:sgdavg} initialized at $x_0$ using the sample gradient $g_t = A^T(Ax_t - y_t)$ at step $t$ with $y_t\sim\mathcal{P}_t$; hence $\mathbb{E}\|\nabla f_t(x_t) - g_t\|^2 \leq \sigma^2$. 

In our simulations, we set $d=50$, $n=100$, and  $\Sigma_t = (\sigma^2/nL)I_n$ for all $t$. We generate $A$ via singular value decomposition with Haar-distributed orthogonal matrices. In Figures~\ref{fig:fig3} and \ref{fig:fig4}, we use default parameter values $\mu=L=1$, $\sigma = 1/2$,  $\Delta = 1/20$, and the corresponding asymptotically optimal step size $\eta = \eta_{\star}$. Since $f_t$ is $\mu$-strongly convex and $L$-smooth, this puts us in the low drift/noise regime in Figure~\ref{fig:fig3}: $\Delta/\sigma < \sqrt{\mu/16L^3}=1/4$. To estimate the expected values and confidence intervals of $\|x_t - x_t^\star\|^2$ and $\varphi_t(\hat{x}_t) - \varphi_t^\star$, we run $100$ trials with horizon $T=100$. 

\begin{figure}[h!]
	\centering
	\includegraphics[width=0.98\textwidth]{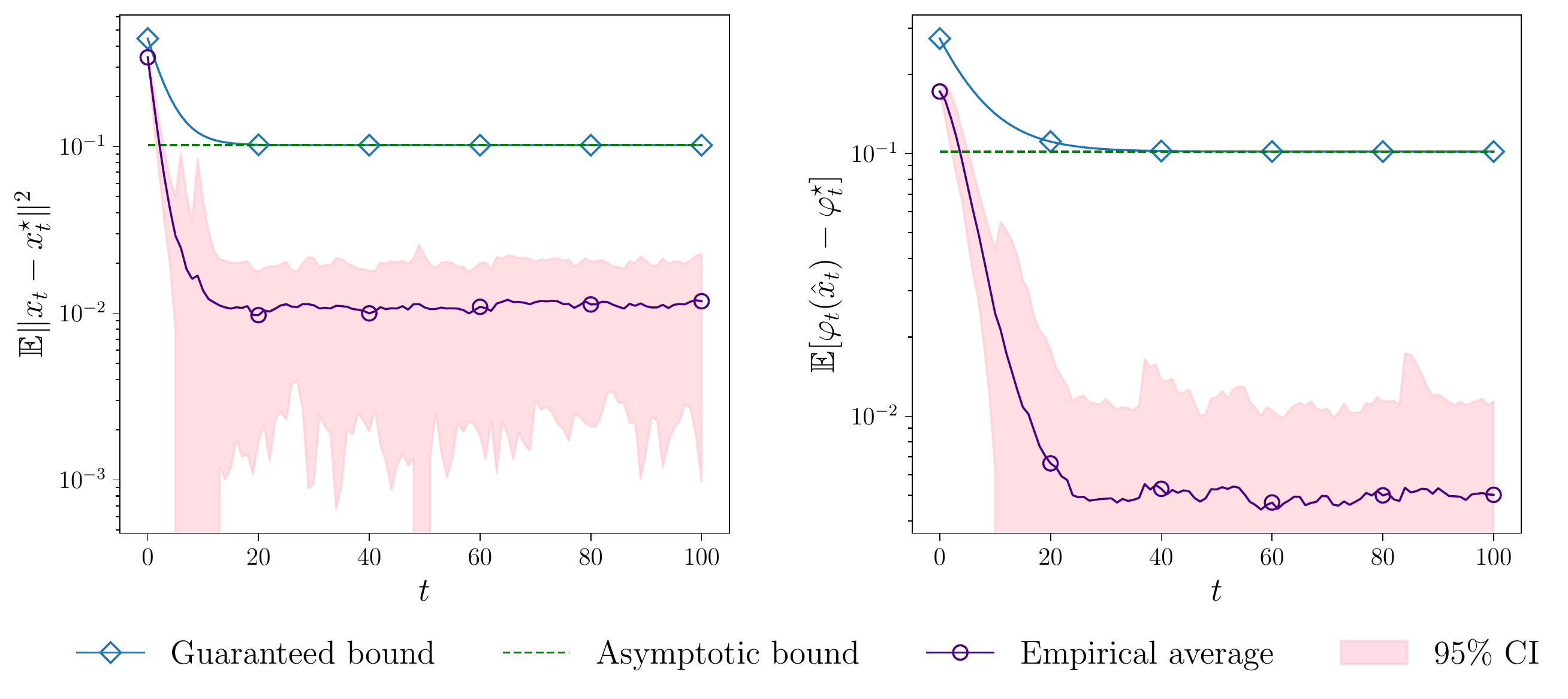}
	\caption{\small Semilog plots of guaranteed bounds and empirical tracking errors with respect to iteration $t$ for sparse least-squares recovery. Shaded regions indicate the $95\%$ confidence intervals for $\|x_t - x_t^\star\|^2$ and $\varphi_t(\hat{x}_t) - \varphi_t^\star$; empirical averages and confidence intervals are computed over $100$ trials. Default parameter values: $\mu=L=1$, $\sigma = 1/2$,  $\Delta = 1/20$, and $\eta = \eta_{\star}$.}
	\label{fig:fig3}
	\vspace{-.2in}
\end{figure}

\begin{figure}[h!]
	\centering
	\includegraphics[width=0.98\textwidth]{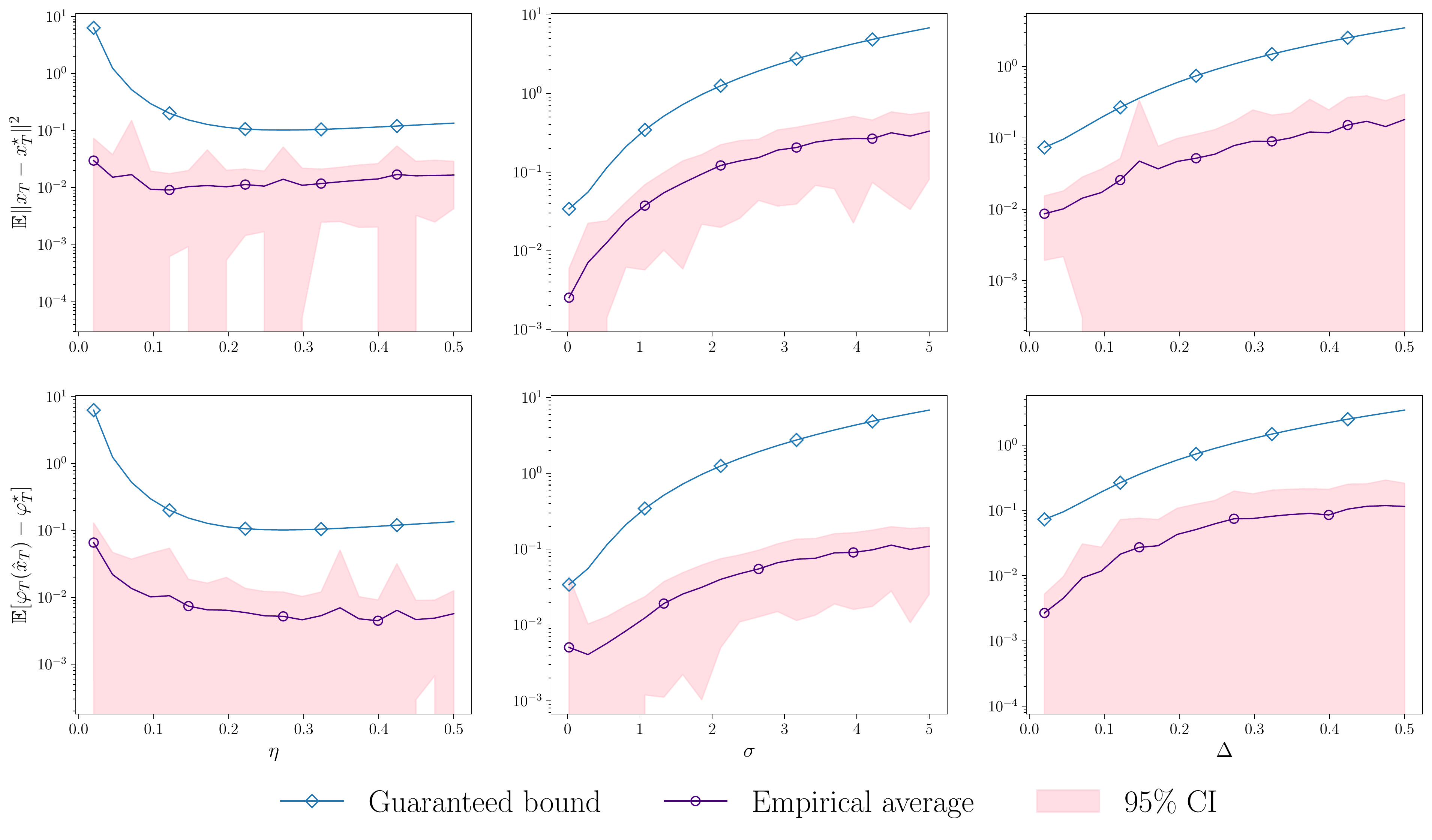}
	\caption{\small Semilog plots of guaranteed bounds and empirical tracking errors at horizon $T= 100$ with respect to $\eta$, $\sigma$, and $\Delta$ for sparse least-squares recovery. Shaded regions indicate the $95\%$ confidence intervals for  $\|x_T - x_T^\star\|^2$ and $\varphi_T(\hat{x}_T) - \varphi_T^\star$; empirical averages and confidence intervals are computed over $100$ trials. Default parameter values: $\mu=L=1$, $\sigma = 1/2$,  $\Delta = 1/20$, and $\eta = \eta_{\star}$.}
	\label{fig:fig4}
\end{figure}

\paragraph{$\ell_2^2$-regularized logistic regression.}
Finally, we consider the time-varying $\ell_2^2$-regularized logistic regression problem
\begin{equation*}
	\min_{x\in\mathbb{R}^d}\,\frac{1}{n}\!\left(\sum_{i=1}^{n}\log\!\big(1 + \exp\langle a_i, x\rangle\big) - \langle Ax,b_t \rangle\!\right)+ \frac{\mu}{2}\|x\|^2,
\end{equation*}
where the matrix $A\in\mathbb{R}^{n\times d}$ has fixed rows $a_1,\ldots,a_n\in\R^d$, $\{b_t\}$ is a random sequence of label vectors in $\{0,1\}^n$ such that $b_t$ and $b_{t+1}$ differ in precisely one coordinate for each $t$, and $\mu>0$. This amounts to the problem \eqref{eqn:prob_online} with $f_t(x) = \frac{1}{n}\!\left(\sum_{i=1}^{n}\log(1 + \exp\langle a_i, x\rangle) - \langle Ax,b_t \rangle \right)+ \frac{\mu}{2}\|x\|^2$ and $r_t = 0$; setting $L = \frac{1}{4n}\|A\|_{\text{op}}^2 + \mu$, it follows that $f_t$ is $\mu$-strongly convex and $L$-smooth. Letting $\{x_t^\star\}$ denote the corresponding sequence of minimizers and setting $\Delta = \frac{1}{\mu n}\max_{i=1,\ldots,n}\|a_i\|$, it follows that the minimizer and gradient drift satisfy
 \begin{equation*}
\mu\|x_t^\star-x_{t+1}^\star\| \leq \sup_{x}\|\nabla f_t(x)-\nabla f_{t+1}(x)\|\leq\mu\Delta.  	
 \end{equation*}
We implement Algorithms~\ref{alg:sgd} and \ref{alg:sgdavg} using the random summand sample gradient 
\begin{equation*}
		g_t = \bigg(\frac{\exp\langle a_k, x_t \rangle}{1+\exp\langle a_k, x_t \rangle} - b_t^k\bigg)a_k + \mu x_t 
\end{equation*}
at step $t$, where $k\sim\mathsf{Unif}\{1,\ldots,n\}$ and $b_t^k$ denotes the $k^\text{th}$ coordinate of $b_t$; the gradient noise satisfies $\mathbb{E}\|\nabla f_t(x_t) - g_t\|^2 \leq\sigma^2$, where
\begin{equation*}
	\sigma^2 = \frac{1}{n^2}\!\left(\!(n-2)\sum_{i=1}^{n}\|a_i\|^2 + \sum_{i,j=1}^{n}\|a_i\|\|a_j\|\!\right) \leq 2\!\left(\max_{i=1,\ldots,n}\|a_i\|^2\right)\!.
\end{equation*}
\bigskip

In our simulations, we set $d=20$ and $n=200$, fix standard Gaussian vectors $x_0\in\R^d$ and $a_1,\ldots,a_n\in\R^d$, fix $b_0$ drawn uniformly from $\{0,1\}^n$, and generate $b_{t+1}$ from $b_t$ by flipping a single coordinate selected uniformly at random. In Figure~\ref{fig:fig5}, we use default parameter values $\mu=1$ and the corresponding asymptotically optimal step size $\eta=\eta_\star$. In Figure~\ref{fig:fig6}, we illustrate the dependence of tracking error on the regularization parameter $\mu$; here, the asymptotically optimal step size $ \eta_\star$ is used (which itself depends on $\mu$). In Figure~\ref{fig:fig7}, we use the default parameter value $\mu =1$. To estimate the expected values and confidence intervals of $\|x_t - x_t^\star\|^2$ and $\varphi_t(\hat{x}_t) - \varphi_t^\star$, we run $100$ trials with horizon $T=600$. The results confirm our bounds and show that they capture the correct dependence on  $\mu$ and $\eta$. In particular, Figure~\ref{fig:fig7} illustrates that $\eta_\star$ is close to empirically optimal.

\bigskip
\bigskip

\begin{figure}[h!]
	\centering
	\includegraphics[width=0.98\textwidth]{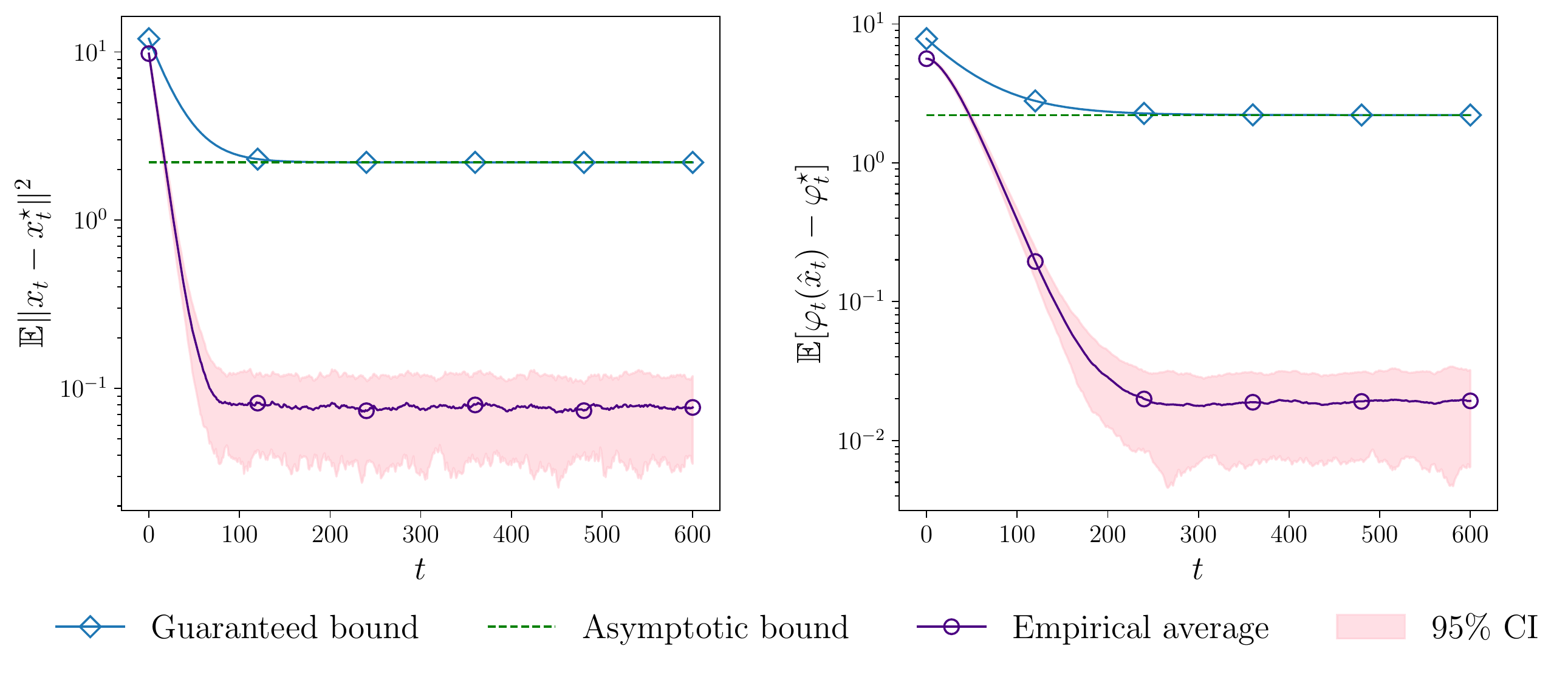}
	\caption{\small Semilog plots of guaranteed bounds and empirical tracking errors with respect to iteration $t$ for $\ell_2^2$-regularized logistic regression. Shaded regions indicate the $95\%$ confidence intervals for  $\|x_t - x_t^\star\|^2$ and $\varphi_t(\hat{x}_t) - \varphi_t^\star$; empirical averages and confidence intervals are computed over $100$ trials. Default parameter values: $\mu = 1$ and $\eta = \eta_{\star}$.}
	\label{fig:fig5}
\end{figure}

\begin{figure}[h!]
	\centering
	\includegraphics[width=0.98\textwidth]{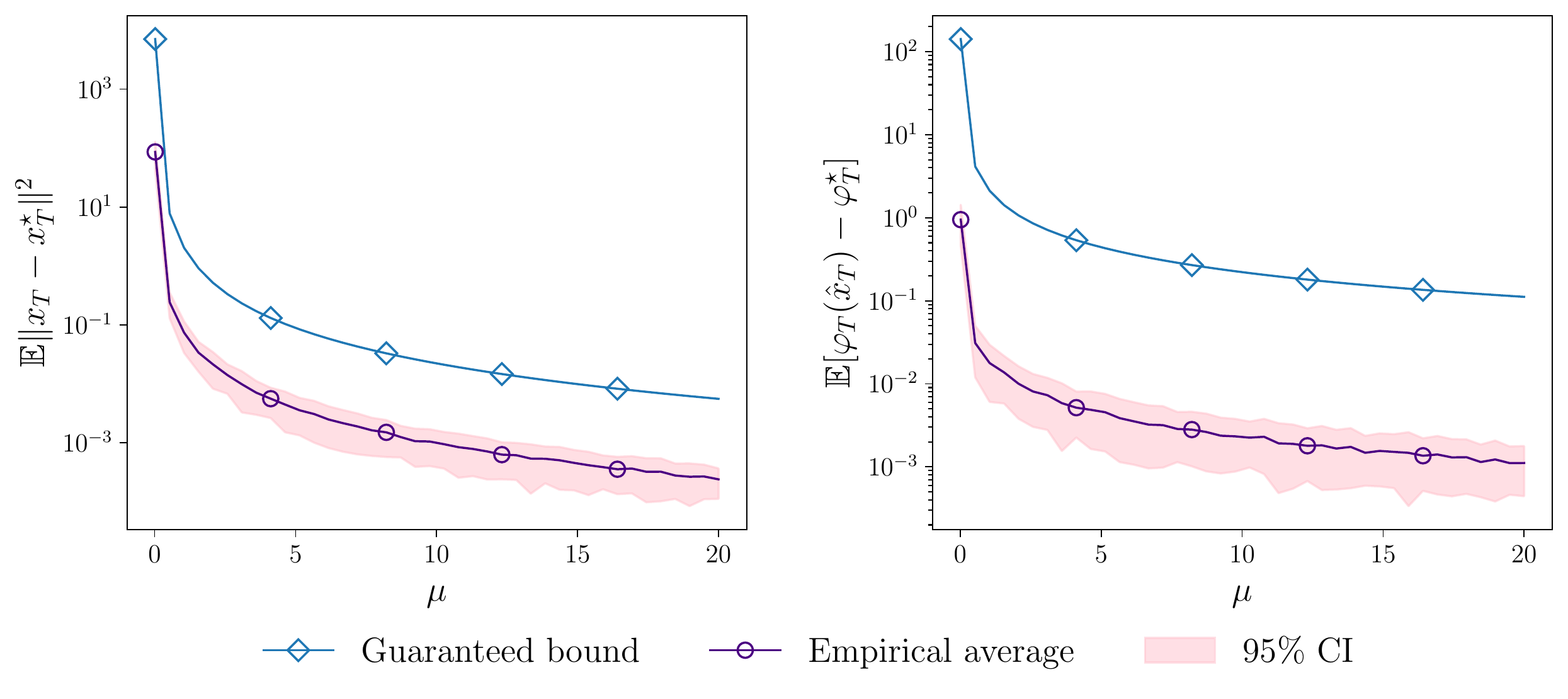}
	\caption{\small Semilog plots of guaranteed bounds and empirical tracking errors at horizon $T=600$ with respect to the strong convexity parameter $\mu$ for $\ell_2^2$-regularized logistic regression. Shaded regions indicate the $95\%$ confidence intervals for  $\|x_T - x_T^\star\|^2$ and $\varphi_T(\hat{x}_T) - \varphi_T^\star$; empirical averages and confidence intervals are computed over $100$ trials, using the asymptotically optimal step size $\eta_\star$ (which itself depends on $\mu$).}
	\label{fig:fig6}
	\vspace{-.2in}
\end{figure}

\begin{figure}[h!]
	\centering
	\includegraphics[width=0.98\textwidth]{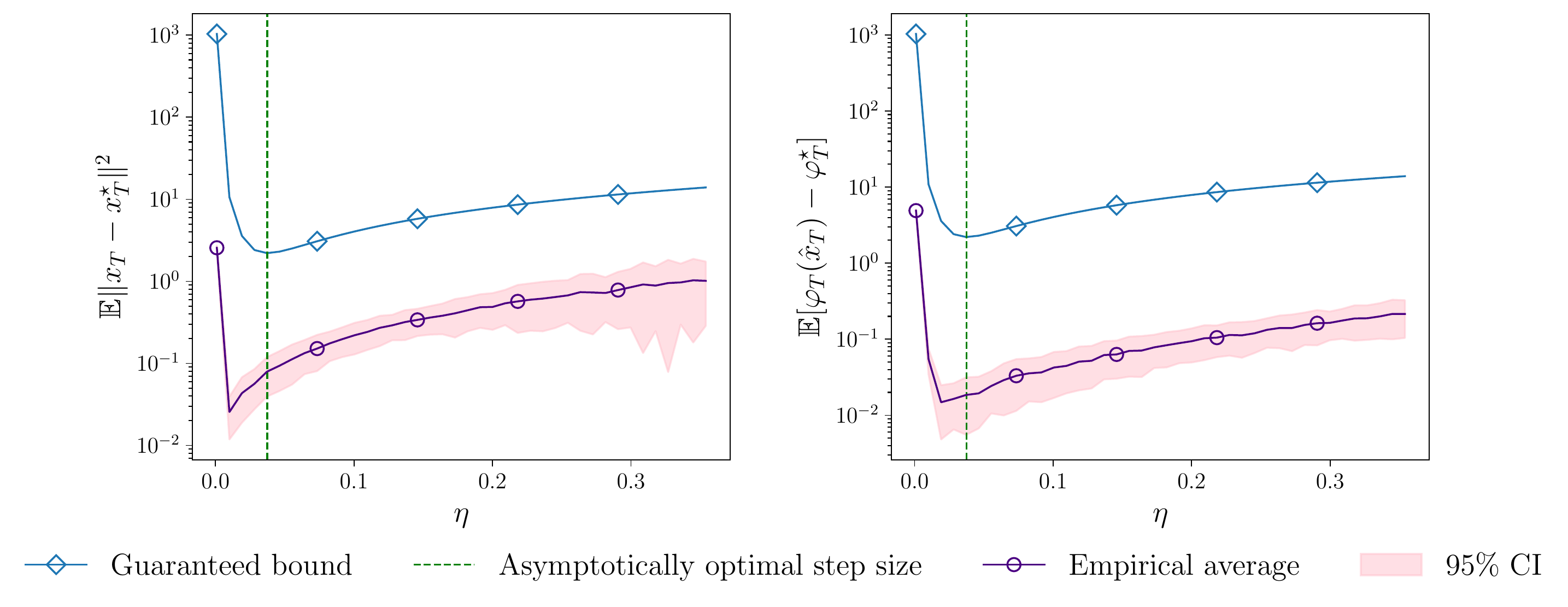}
	\caption{\small Semilog plots of guaranteed bounds and empirical tracking errors at horizon $T=600$ with respect to the step size $\eta$ for $\ell_2^2$-regularized logistic regression. Shaded regions indicate the $95\%$ confidence intervals for $\|x_T - x_T^\star\|^2$ and $\varphi_T(\hat{x}_T) - \varphi_T^\star$; empirical averages and confidence intervals are computed over $100$ trials. Default parameter value: $\mu=1$. Observe that $\eta_\star$ is close to empirically optimal.}
	\label{fig:fig7}
\end{figure}

\clearpage

 \acks{This work was supported by NSF DMS-2023166, DMS-2134012, DMS-1651851, DMS-2133244, CCF-2019844, and CIFAR LMB. Part of this work was done while Z. Harchaoui was visiting the Simons Institute for the Theory of Computing.}
\appendix

\section{Averaging Lemma}\label{apdx:avg}

We will use a variation of the averaging strategy used by \citet{ghadimi2012optimal}; our approach here follows \citet[Section A]{drusvyatskiy2020stochastic} and  \citet[Sections A.2 and A.3]{kulunchakov2020estimate}. To begin, consider a convex function $h\colon\mathbb{R}^d\to\mathbb{R}\cup\{\infty\}$ and let $\{x_t\}_{t\geq 0}$ be a sequence of vectors in $\R^d$. Suppose that there are constants $c_1,c_2\in \mathbb{R}$, a sequence of nonnegative weights $\{\rho_t\}_{t\geq 1}$, and scalar sequences $\{V_t\}_{t\geq 0}$ and $\{\omega_t\}_{t\geq 1}$ satisfying the recursion
\begin{equation}\label{avgrec}
	\rho_th(x_t)\leq (1-c_1\rho_t)V_{t-1}-(1+c_2\rho_t)V_t+\omega_t
\end{equation}
for all $t\geq1$. The goal is to bound the function value $h(\hat x_t)$ evaluated along an ``average iterate'' $\hat x_t$. 

Suppose that the relations $c_1+c_2>0$, $1- c_1\rho_t>0$, and $1+c_2\rho_t> 0$  hold for all $t\geq 1$. Define the augmented weights and products
\begin{equation*}
	\hat\rho_t=\frac{(c_1+c_2)\rho_t}{1+c_2\rho_t} \quad\text{~~and~~} \quad \hat\Gamma_t=\prod_{i=1}^t(1-\hat \rho_i)
\end{equation*}
for each $t\geq1$, and set $\hat{\Gamma}_0 = 1$. A straightforward induction yields the relation
\begin{equation}\label{eq:convexco}
	1+\sum_{i=1}^{t}\frac{\hat{\rho}_i}{\hat{\Gamma}_i} = \frac{1}{\hat{\Gamma}_t}.
\end{equation}
Now set  $\hat x_0 = x_0$ and recursively define the average iterates	\begin{equation*}
	\hat{x}_t=(1-\hat\rho_t)\hat x_{t-1}+\hat\rho_t x_t
\end{equation*}
for all $t\geq1$. Unrolling this recursion, we may equivalently write 
\begin{equation}\label{avgaltdef}
	\hat{x}_t=\hat{\Gamma}_t\!\left(x_0+\sum_{i=1}^{t}\frac{\hat{\rho}_i}{\hat{\Gamma}_i}x_i \right)\!.
\end{equation}
The following lemma provides the key estimate we will need.	

\begin{lemma}[Averaging]\label{lem:avg}
	The estimate holds for all $t\geq0$:
	\begin{equation*}
		\frac{h(\hat x_t)}{c_1+c_2} + V_t \leq  \hat \Gamma_t\!\left( \frac{h(x_0)}{c_1+c_2}+ V_0+\sum_{i=1}^t\frac{ \omega_i}{ \hat \Gamma_i(1+c_2\rho_i)}\right)\!.
	\end{equation*}
\end{lemma}

\begin{proof}
	Observe that \eqref{avgaltdef} expresses $\hat{x}_t$ as a convex combination of $x_0,\ldots,x_t$ by virtue of \eqref{eq:convexco}. Therefore, by the convexity of $h$, we may apply Jensen's inequality to obtain
	\begin{equation}\label{eq:jensen}
		h(\hat{x}_t)\leq\hat{\Gamma}_t\!\left(h(x_0)+\sum_{i=1}^{t}\frac{\hat{\rho}_i}{\hat{\Gamma}_i}h(x_i)\right)\!.
	\end{equation} 
	On the other hand, for each $i\geq1$, we may divide the recursion \eqref{avgrec} by $\hat{\Gamma}_i(1+c_2\rho_i)$ to obtain
	\begin{equation*}
		\frac{\hat{\rho}_i}{(c_1+c_2)\hat{\Gamma}_i}h(x_i)\leq\frac{V_{i-1}}{\hat{\Gamma}_{i-1}}-\frac{V_i}{\hat{\Gamma}_i}+\frac{\omega_i}{\hat{\Gamma}_i(1+c_2\rho_i)},
	\end{equation*} 
	which telescopes to yield
	\begin{equation*}
		\frac{1}{c_1+c_2}\sum_{i=1}^{t}\frac{\hat{\rho}_i}{\hat{\Gamma}_i}h(x_i)\leq V_0 - \frac{V_t}{\hat{\Gamma}_t}+\sum_{i=1}^{t}\frac{\omega_i}{\hat{\Gamma}_i(1+c_2\rho_i)}.
	\end{equation*} 
	Multiplying this inequality by $\hat{\Gamma}_t$ and applying \eqref{eq:jensen} yields
	\begin{equation*}
		\frac{h(\hat{x}_t)}{c_1+c_2}\leq\hat{\Gamma}_t\!\left(\frac{h(x_0)}{c_1+c_2}+V_0 - \frac{V_t}{\hat{\Gamma}_t}+\sum_{i=1}^{t}\frac{\omega_i}{\hat{\Gamma}_i(1+c_2\rho_i)}\right)\!,
	\end{equation*} 
	as claimed.
\end{proof}

\section{Additional Proofs}\label{missingproofs}
\subsection{Proof of Theorem~\ref{thm:time_to_track2}}\label{pf:time_to_track2}
For each index $k$, let $t_k:=T_0+\cdots+T_{k-1}$ (with $t_0 := 0$), $\bar{X}_k$ be the minimizer of the corresponding function $\psi_{t_k}$, and $$\bar{E}_k:=c\!\left(\frac{\eta_k\sigma^2}{\bar\mu}+\left(\frac{\bar\Delta}{\bar\mu\bar\eta_\star}\right)^2\right)\!,$$ where $c\geq1$ is an absolute constant satisfying the bound (\ref{ineq:HPtrackdec}) in Theorem~\ref{HPtrackdec}. Taking into account $\eta_k\geq\bar{\eta}_{\star}$ and our selection of $c$, Theorem~\ref{HPtrackdec} implies that for any specified index $k$ and $\delta\in(0,1)$, the following estimate holds with probability at least $1-\delta$:
\begin{align*}
	\|X_{k+1} - \bar{X}_{k+1}\|^2 &\leq \big(1-\tfrac{\bar\mu\eta_k}{2}\big)^{T_k}\|X_k - \bar{X}_k\|^2 + c\!\left(\frac{\eta_k\sigma^2}{\bar\mu}+\left(\frac{\bar\Delta}{\bar\mu\eta_k}\right)^2\right)\!\log\!\left(\frac{e}{\delta}\right)\\
	&\leq e^{-\bar\mu\eta_kT_k/2}\|X_{k} - \bar{X}_{k}\|^2 + \bar{E}_k\log\!\left(\frac{e}{\delta}\right)\!.
\end{align*}

We will verify by induction that for each index $k\geq1$, the estimate $\|X_k - \bar{X}_k\|^2\leq 3\bar{E}_{k-1}\log(e/\delta)$ holds with probability at least $1-\delta$ for all $\delta\in(0,1)$. To see the base case, observe that the estimate $$\|X_{1} - \bar{X}_{1}\|^2\leq e^{-\bar\mu\eta_0 T_0/2}\|X_0 - \bar{X}_0\|^2 + \bar{E}_0\log\!\left(\frac{e}{\delta}\right)\leq 3\bar{E}_0\log\!\left(\frac{e}{\delta}\right)$$ holds with probability at least $1-\delta$ for all $\delta\in(0,1)$. Now assume that the claim holds for some index $k\geq1$, and let $\delta\in(0,1)$; then $\|X_k - \bar{X}_k\|^2\leq 3\bar{E}_{k-1}\log(2e/\delta)$ with probability at least $1 - \delta/2$. Thus, since we also have 
\begin{align*}
	\|X_{k+1} - \bar{X}_{k+1}\|^2&\leq e^{-\bar\mu\eta_k T_k/2}\|X_k - \bar{X}_k\|^2 + \bar{E}_k\log\!\left(\frac{2e}{\delta}\right)\\
	&\leq \frac{1}{12}\|X_k - \bar{X}_k\|^2 + \bar{E}_k\log\!\left(\frac{2e}{\delta}\right)\\
	&\leq \frac{\bar{E}_k}{6\bar{E}_{k-1}}\|X_k - \bar{X}_k\|^2 + \bar{E}_k\log\!\left(\frac{2e}{\delta}\right)
\end{align*}
with probability at least $1 - \delta/2$, a union bound reveals $$\|X_{k+1} - \bar{X}_{k+1}\|^2 \leq \frac{3}{2}\bar{E}_k\log\!\left(\frac{2e}{\delta}\right)\leq 3\bar{E}_k\log\!\left(\frac{e}{\delta}\right) $$ with probability at least $1-\delta$, thereby completing the induction. Hence, upon fixing $\delta\in(0,1)$, we have $\|X_K - \bar{X}_K\|^2 \leq 3\bar{E}_{K-1}\log(e/\delta)$ with probability at least $1-\delta$. 

Next, observe
$$\frac{2}{c}\bar{E}_{K-1}-\sqrt[3]{54}\left(\frac{\bar\Delta\sigma^2}{\bar{\mu}^2}\right)^{2/3} = \frac{2\sigma^2}{\bar\mu}(\eta_{K-1} - \bar{\eta}_\star) = \frac{2\sigma^2}{\bar\mu}\cdot \frac{\eta_0 - \bar{\eta}_\star}{2^{K-1}} \leq \left(\frac{\bar\Delta\sigma^2}{\bar{\mu}^2}\right)^{2/3} = \bar{\mathcal{E}},$$ so $$\|X_K - \bar{X}_K\|^2 \leq \frac{3c}{2}\big(1+\sqrt[3]{54}\big)\bar{\mathcal{E}}\log\!\left(\frac{e}{\delta}\right) \asymp \bar{\mathcal{E}}\log\!\left(\frac{e}{\delta}\right)$$ with probability at least $1-\delta$. Finally, note $$T \lesssim \frac{L}{\bar\mu}\log\!\left(\frac{\bar\mu LD}{\sigma^2}\right)^+ + \frac{1}{\bar\mu}\sum_{k=1}^{K-1}\frac{1}{\eta_k}$$ and $$\sum_{k=1}^{K-1}\frac{1}{\eta_k}\leq 2L\sum_{k=1}^{K-1} 2^k \leq 2L\cdot 2^{K} = 8L\cdot 2^{K-2} \leq 8 \left(\frac{\sigma^2\bar\mu}{\bar{\Delta}^2}\right)^{1/3} = \frac{8\sigma^2}{\bar\mu} \cdot \left(\frac{\bar\Delta\sigma^2}{\bar{\mu}^2}\right)^{-2/3} \asymp \frac{\sigma^2}{\bar{\mu}\bar{\mathcal{E}}}.$$ 
This completes the proof. 

\subsection{Proof of Theorem~\ref{thm:time_to_track_gap_exp}}\label{pf:gap_schedul}
For each index $k$, let $t_k:=T_0+\cdots+T_{k-1}$ (with $t_0 := 0$) and $\widehat{G}_k:=\eta_k\sigma^2 + 8\bar{\Delta}^2/\hat{\mu}\hat{\eta}^2_{\star}$.
Then taking into account $\eta_k\geq\hat{\eta}_{\star}$, Corollary~\ref{cor2} and inequality (\ref{initialgapbound0}) directly imply
\begin{align*}
	\mathbb{E}\big[\psi_{t_{k+1}}(X_{k+1}) - \psi_{t_{k+1}}^\star\big] &\leq \big(1-\tfrac{\hat\mu\eta_k}{2}\big)^{T_k}\mathbb{E}\big[3\big(\psi_{t_k}(X_k) - \psi_{t_k}^\star\big) + 5\hat\mu\bar{\Delta}^2T_k^2\big] + \eta_k\sigma^2 + \frac{8\bar{\Delta}^2}{\hat\mu\eta^2_k} \\
	&\leq 3e^{-\hat\mu\eta_kT_k/2}\mathbb{E}\big[\psi_{t_k}(X_k) - \psi_{t_k}^\star\big] + 5e^{-\hat\mu\eta_kT_k/2}\hat\mu\bar{\Delta}^2T_k^2 + \widehat{G}_k.
\end{align*}

We will verify by induction that the estimate $\mathbb{E}\big[\psi_{t_{k}}(X_{k}) - \psi_{t_{k}}^\star\big] \leq 11\widehat{G}_{k-1}$ holds for all indices $k\geq1$. To see the base case, observe that inequality (\ref{ineq:expquad}) facilitates the estimation
$$\mathbb{E}\big[\psi_{t_1}(X_1) - \psi_{t_1}^\star\big] \leq 3e^{-\hat\mu\eta_0T_0/2}\big(\psi_0(x_0) - \psi_{0}^\star\big) + 5e^{-\hat\mu\eta_0T_0/2}\hat\mu\bar{\Delta}^2T_0^2 + \widehat{G}_0 \leq 11\widehat{G}_0.$$
Now assume that the claim holds for some index $k \geq 1$. We then conclude
\begin{align*}
	\mathbb{E}\big[\psi_{t_{k+1}}(X_{k+1}) - \psi_{t_{k+1}}^\star\big]&\leq 3e^{-\hat\mu\eta_kT_k/2}\mathbb{E}\big[\psi_{t_k}(X_k) - \psi_{t_k}^\star\big] + 5e^{-\hat\mu\eta_kT_k/2}\hat\mu\bar{\Delta}^2T_k^2 + \widehat{G}_k\\
	&\leq \frac{1}{4}\mathbb{E}\big[\psi_{t_k}(X_k) - \psi_{t_k}^\star\big]+\frac{13\bar{\Delta}^2}{\hat\mu\eta_k^2} + \widehat{G}_k\\
	&\leq \frac{\widehat{G}_k}{2 \widehat{G}_{k-1}}\mathbb{E}\big[\psi_{t_k}(X_k) - \psi_{t_k}^\star\big]+\frac{13\bar{\Delta}^2}{\hat\mu\eta_k^2} + \widehat{G}_k \leq 11\widehat{G}_k,
\end{align*}
thereby completing the induction. Hence $\mathbb{E}\big[\psi_T(X_K) - \psi_T^\star\big] \leq 11\widehat{G}_{K-1}$.

Next, observe
$$\widehat{G}_{K-1}-\sqrt[3]{250}\cdot\hat\mu\bigg(\frac{\bar\Delta\sigma^2}{\hat{\mu}^2}\bigg)^{2/3}=\sigma^2(\eta_{K-1}-\hat{\eta}_\star) = \sigma^2\cdot \frac{\eta_0 - \hat{\eta}_\star}{2^{K-1}} \leq \frac{\hat\mu}{2}\bigg(\frac{\bar\Delta\sigma^2}{\hat{\mu}^2}\bigg)^{2/3} = \frac{1}{2}\widehat{\mathcal{G}},$$ 
so 
$$\mathbb{E}\big[\psi_T(X_K) - \psi_T^\star\big] \leq  11\big(\tfrac{1}{2}+\sqrt[3]{250}\big)\cdot\hat\mu\bigg(\frac{\bar\Delta\sigma^2}{\hat{\mu}^2}\bigg)^{2/3}\asymp\widehat{\mathcal{G}}.$$ Finally, note 
$$T \lesssim \frac{L}{\hat\mu}\log\left(\frac{LD}{\sigma^2}\right)^+ + \frac{1}{\hat\mu}\sum_{k=1}^{K-1}\frac{1}{\eta_k}$$ and $$\sum_{k=1}^{K-1}\frac{1}{\eta_k}\leq 2L\sum_{k=1}^{K-1} 2^k \leq 2L\cdot 2^{K} = 8L\cdot 2^{K-2} \leq 8 \bigg(\frac{\sigma^2\hat\mu}{\bar{\Delta}^2}\bigg)^{1/3} = 8\sigma^2 \cdot \hat{\mu}^{-1}\bigg(\frac{\bar\Delta\sigma^2}{\hat{\mu}^2}\bigg)^{-2/3} \asymp \frac{\sigma^2}{\widehat{\mathcal{G}}}.$$  
This completes the proof.

\subsection{Proof of Proposition~\ref{prop:avgrec2}}\label{pf:avgrec2}
Fix $t\geq1$. Given $i\geq1$ and $\alpha>0$, the $\mu$-strong convexity of $\psi_t$ and Lemma \ref{lem:avgrecdec} imply
\begin{equation*}
	\begin{split}
		\mu\eta\|x_i - \bar{x}_t\|^2 \leq 2\eta\big(\psi_t(x_i) - \psi_t^\star\big) \leq (1&-\bar\mu\eta)\|x_{i-1} - \bar{x}_t\|^2 - \big(1-(\gamma+\alpha)\eta\big)\|x_i - \bar{x}_t\|^2 \\ &+ 2\eta\langle z_{i-1}, x_{i-1} - \bar{x}_t \rangle + 2\eta^2\|z_{i-1}\|^2 + \tfrac{\eta}{\alpha}\bar{G}_{i-1,t}^2,
	\end{split}
\end{equation*}
hence
\begin{equation*}
	\begin{split}
	\big(1+(\bar\mu-\alpha)\eta\big)\|x_i - \bar{x}_t\|^2\leq (1-\bar\mu\eta)\|x_{i-1} - \bar{x}_t\|^2 &+ 2\eta\langle z_{i-1}, x_{i-1} - \bar{x}_t \rangle \\ &+ 2\eta^2\|z_{i-1}\|^2 + \tfrac{\eta}{\alpha}\bar{G}_{i-1,t}^2.
	\end{split}
\end{equation*}
Taking $\alpha=\bar\mu$, we obtain
\begin{equation*}
	\begin{split}
		\|x_i - \bar{x}_t\|^2 \leq (1-\bar\mu\eta)\|x_{i-1} - \bar{x}_t\|^2 + 2\eta\langle z_{i-1}, x_{i-1} - \bar{x}_t \rangle + 2\eta^2\|z_{i-1}\|^2 + \tfrac{\eta}{\bar\mu}\bar{G}_{i-1,t}^2.
	\end{split}
\end{equation*}
Thus, given any $\lambda\in(0,\bar\mu\eta]$ and proceeding by induction, we conclude that the following estimate holds for all $i\geq1$:
\begin{equation*}
	\begin{split}
		\|x_i - \bar{x}_t\|^2\leq (1-\lambda)^i\|x_0 - \bar{x}_t\|^2 &+ 2\eta\sum_{j=0}^{i-1}\langle z_j, x_j - \bar{x}_t \rangle(1-\lambda)^{i-1-j} \\
		&+ 2\eta^2\sum_{j=0}^{i-1}\|z_j\|^2(1-\lambda)^{i-1-j} + \tfrac{\eta}{\bar\mu}\sum_{j=0}^{i-1}\bar{G}_{j,t}^2(1-\lambda)^{i-1-j}.
	\end{split} 
\end{equation*}
Therefore
\begin{equation*}
	\begin{split}
		\sum_{i=0}^{t-1}\|x_i - \bar{x}_t\|^2&(1-\lambda)^{2(t-1-i)} \\ &\leq\|x_0 - \bar{x}_t\|^2\sum_{i=0}^{t-1}(1-\lambda)^{2(t-1)-i}  +2\eta\sum_{i=1}^{t-1}\sum_{j=0}^{i-1}\langle z_j, x_j - \bar{x}_t\rangle(1-\lambda)^{2t-3-j-i} \\ 
		&~~~~+ 2\eta^2\sum_{i=1}^{t-1}\sum_{j=0}^{i-1}\|z_j\|^2(1-\lambda)^{2t-3-j-i} +\tfrac{\eta}{\bar\mu}\sum_{i=1}^{t-1}\sum_{j=0}^{i-1}\bar{G}_{j,t}^2(1-\lambda)^{2t-3-j-i}.
	\end{split}
\end{equation*}

Next, we compute
\begin{equation*}
	\sum_{i=0}^{t-1}(1-\lambda)^{2(t-1)-i}=(1-\lambda)^{t-1}\sum_{i=0}^{t-1}(1-\lambda)^{t-1-i} < \tfrac{1}{\lambda}(1-\lambda)^{t-1}
\end{equation*}
and observe that for any scalar sequence $\{a_j\}_{j=0}^{t-2}$, we have
\begin{equation*}
	\sum_{i=1}^{t-1}\sum_{j=0}^{i-1}a_j(1-\lambda)^{2t-3-j-i}=\sum_{j=0}^{t-2}\left(\sum_{i=j+1}^{t-1}(1-\lambda)^{t-2-i}\right)\!a_j(1-\lambda)^{t-1-j}.
\end{equation*}
Further, if $a_j\geq0$ for all $j=0,\ldots,t-2$, then we have
\begin{equation*}
	\begin{split}
	\sum_{i=1}^{t-1}\sum_{j=0}^{i-1}a_j(1-\lambda)^{2t-3-j-i} &=\sum_{j=0}^{t-2}\left(\sum_{i=j+1}^{t-1}(1-\lambda)^{t-1-i}\right)\!a_j(1-\lambda)^{t-2-j} \\ &\leq\tfrac{1}{\lambda}\sum_{j=0}^{t-2}a_j(1-\lambda)^{t-2-j}.
	\end{split}
\end{equation*}
Hence the following estimation holds:
\begin{equation*}
	\begin{split}
		\sum_{i=0}^{t-1}\|x_i - \bar{x}_t\|^2(1-\lambda)^{2(t-1-i)}\leq \sum_{j=0}^{t-2}&\left(2\eta\sum_{i=j+1}^{t-1}(1-\lambda)^{t-2-i}\right)\!\langle z_j,x_j - \bar{x}_t \rangle(1-\lambda)^{t-1-j} \\ 
		&+ \tfrac{1}{\lambda}(1-\lambda)^{t-1}\|x_0 - \bar{x}_t\|^2 + \tfrac{2\eta^2}{\lambda}\sum_{j=0}^{t-2}\|z_j\|^2(1-\lambda)^{t-2-j} \\
		&+ \tfrac{\eta}{\bar\mu\lambda}\sum_{j=0}^{t-2}\bar{G}_{j,t}^2(1-\lambda)^{t-2-j}.
	\end{split}
\end{equation*}
This completes the proof.

\subsection{Proof of Theorem~\ref{thm:time_to_track_gap_hp2}}\label{pf:time_to_track_gap_hp2}
For each index $k$, let $t_k:=T_0+\cdots+T_{k-1}$ (with $t_0 := 0$) and $\widehat{G}_k:=\eta_k\sigma^2 + \bar{\Delta}^2/\hat\mu\hat{\eta}^2_{\star}$.
Then taking into account $\eta_k\geq\hat{\eta}_{\star}$ and our selection of the absolute constant $c>0$ via (\ref{ineq:hpbound2dec}), it follows that for any specified index $k$, the estimate
\begin{align*}
	\psi_{t_{k+1}}(X_{k+1}) - \psi_{t_{k+1}}^\star &\leq c\left(\big(1-\tfrac{\hat\mu\eta_k}{2}\big)^{T_k}\big(\psi_{t_k}(X_k) - \psi_{t_k}^\star\big) + \eta_k\sigma^2 + \frac{\bar{\Delta}^2}{\hat\mu\eta^2_k}\right)\log\!\left(\frac{e}{\delta}\right) \\
	&\leq c\left(e^{-\hat\mu\eta_kT_k/2}\big(\psi_{t_k}(X_k) - \psi_{t_k}^\star\big) + \widehat{G}_k\right)\log\!\left(\frac{e}{\delta}\right)
\end{align*}
holds with probability at least $1-\delta$.

We will verify by induction that for each index $k\geq 1$, the estimate $$\psi_{t_k}(X_k) - \psi_{t_k}^\star \leq 3c\cdot\widehat{G}_{k-1}\log\!\left(\frac{e}{\delta}\right)$$ holds with probability at least $1-k\delta$. To see the base case, observe that the estimate
$$\psi_{t_1}(X_1) - \psi_{t_1}^\star \leq c\left(e^{-\hat\mu\eta_0T_0/2}\big(\psi_0(x_0) - \psi_0^\star\big) + \widehat{G}_0\right)\log\!\left(\frac{e}{\delta}\right)\leq 3c\cdot\widehat{G}_0\log\!\left(\frac{e}{\delta}\right)$$
holds with probability at least $1-\delta$.
Now assume that the claim holds for some index $k\geq1$. Then because we also have
\begin{align*}
	\psi_{t_{k+1}}(X_{k+1}) - \psi_{t_{k+1}}^\star&\leq c\left(e^{-\hat\mu\eta_kT_k/2}\big(\psi_{t_k}(X_k) - \psi_{t_k}^\star\big) + \widehat{G}_k\right)\log\!\left(\frac{e}{\delta}\right)\\
	&\leq c\left(\frac{1}{4c\log(e/\delta)}\big(\psi_{t_k}(X_k) - \psi_{t_k}^\star\big) + \widehat{G}_k\right)\log\!\left(\frac{e}{\delta}\right)\\
	&\leq c\left(\frac{\widehat{G}_k}{2c\cdot\widehat{G}_{k-1}\log(e/\delta)}\big(\psi_{t_k}(X_k) - \psi_{t_k}^\star\big) + \widehat{G}_k\right)\log\!\left(\frac{e}{\delta}\right)
\end{align*}
with probability at least $1-\delta$, a union bound reveals that the estimate $$\psi_{t_{k+1}}(X_{k+1}) - \psi_{t_{k+1}}^\star \leq 3c\cdot\widehat{G}_{k}\log\!\left(\frac{e}{\delta}\right)$$ holds with probability at least $1-(k+1)\delta$, thereby completing the induction. In particular, $\psi_{T}(X_K) - \psi_{T}^\star \leq 3c\cdot\widehat{G}_{K-1}\log(e/\delta)$ with probability at least $1-K\delta$. 

Next, observe
$$\widehat{G}_{K-1}-\sqrt[3]{\tfrac{27}{4}}\cdot\hat\mu\left(\frac{\bar\Delta\sigma^2}{\hat\mu^2}\right)^{2/3}=\sigma^2(\eta_{K-1} - \hat{\eta}_\star) = \sigma^2\cdot \frac{\eta_0 - \hat{\eta}_\star}{2^{K-1}} \leq \frac{\hat\mu}{2}\left(\frac{\bar\Delta\sigma^2}{\hat{\mu}^2}\right)^{2/3} = \frac{1}{2}\widehat{\mathcal{G}},$$ so $$\psi_{T}(X_K) - \psi_{T}^\star \leq  3c\left(\tfrac{1}{2}+\sqrt[3]{\tfrac{27}{4}}\right)\cdot\hat\mu\left(\frac{\bar\Delta\sigma^2}{\hat\mu^2}\right)^{2/3}\log\left(\frac{e}{\delta}\right)\asymp\widehat{\mathcal{G}}\log\!\left(\frac{e}{\delta}\right)$$ with probability at least $1-K\delta$. Finally, note 
$$T \lesssim \frac{L}{\hat\mu}\log\left(\frac{LD}{\sigma^2}\right)^+ + \left(1 \vee \log\log\frac{e}{\delta}\right)\frac{1}{\hat\mu}\sum_{k=1}^{K-1}\frac{1}{\eta_k}$$ and $$\sum_{k=1}^{K-1}\frac{1}{\eta_k}\leq 2L\sum_{k=1}^{K-1} 2^k \leq 2L\cdot 2^{K} = 8L\cdot 2^{K-2} \leq 8 \left(\frac{\sigma^2\hat\mu}{\bar{\Delta}^2}\right)^{1/3} = 8\sigma^2 \cdot \hat{\mu}^{-1}\left(\frac{\bar\Delta\sigma^2}{\hat{\mu}^2}\right)^{-2/3} \asymp \frac{\sigma^2}{\widehat{\mathcal{G}}}.$$  
This completes the proof. 

\bibliography{biblio}

\end{document}